\documentclass[12pt,a4paper]{article}
\usepackage{amsmath,amssymb,amsthm,graphicx,color,bbm}
\usepackage[left=1in,right=1in,top=1in,bottom=1in]{geometry}
\usepackage{cite}
\usepackage[british]{babel}
\usepackage{hyperref} 
\usepackage{enumitem}
\usepackage{caption}
\usepackage{subcaption}
\usepackage{tikz}
\usetikzlibrary{decorations.markings}

\hypersetup{
    colorlinks=false,
    pdfborder={0 0 0},
}

\newtheorem{theorem}{Theorem}[section]

\newtheorem{proposition}[theorem]{Proposition}
\newtheorem{lemma}[theorem]{Lemma}

\numberwithin{equation}{section}

\theoremstyle{definition}

\newtheorem{examplex}[theorem]{Example}

\newenvironment{example}
  {\pushQED{\qed}\examplex}
  {\popQED\endexamplex}

\theoremstyle{remark}
\newtheorem{remark}[theorem]{Remark}
\newtheorem{remarks}[theorem]{Remarks}
\newtheorem*{remark*}{Remark}

 \allowdisplaybreaks

\newcommand{\1}[1]{{\mathbbm{1}\mkern -1.5mu}{\{#1\}}}

\newcommand{\R}{{\mathbb R}}

\newcommand{\N}{{\mathbb N}}
\newcommand{\ZP}{{\mathbb Z}_+}
\newcommand{\RP}{{\mathbb R}_+}
\newcommand{\RPRd}{{\mathbb R}_+\! \times \R^d}
\newcommand{\Sp}[1]{{\mathbb S}^{#1}}

\DeclareMathOperator{\Exp}{\mathbb{E}}
\renewcommand{\Pr}{{\mathbb P}}

\DeclareMathOperator{\Int}{int}

\DeclareMathOperator{\trace}{tr}

\newcommand{\tra}{{\scalebox{0.6}{$\top$}}}

\newcommand{\eps}{\varepsilon}

\newcommand{\ud}{{\mathrm d}}

\newcommand{\cD}{{\mathcal D}}
\newcommand{\cE}{{\mathcal E}}
\newcommand{\cF}{{\mathcal F}}

\newcommand{\cM}{{\mathcal M}}

\newcommand{\cS}{{\mathcal S}}
\newcommand{\cT}{{\mathcal T}}

\newcommand{\as}{\ \text{a.s.}}

\newcommand{\bigmid}{\; \bigl| \;}
\newcommand{\Bigmid}{\; \Bigl| \;}
\newcommand{\biggmid}{\; \biggl| \;}

\newcommand{\taue}{{\tau_\mathcal{E}}}

\newcommand{\pcD}{\partial\cD}

\newcommand{\barC}{\overline{C}}

\newcommand{\barcD}{\overline{\cD}}

\makeatletter
\def\namedlabel#1#2{\begingroup  
    (#2)%
    \def\@currentlabel{#2}%
    \phantomsection\label{#1}\endgroup
}
\makeatother

\newlist{myenumi}{enumerate}{10}
\setlist[myenumi]{leftmargin=0pt, labelindent=\parindent, listparindent=\parindent, labelwidth=0pt, itemindent=!, itemsep=1pt, parsep=4pt}

\newlist{thmenumi}{enumerate}{10}
\setlist[thmenumi]{leftmargin=0pt, labelindent=\parindent, listparindent=\parindent, labelwidth=0pt, itemindent=!}

\begin{document}

\title{Reflecting Brownian motion in generalized parabolic domains: explosion and superdiffusivity}
\author{Mikhail V.\ Menshikov\footnote{Durham University} \and Aleksandar Mijatovi\'c\footnote{University of Warwick and the Alan Turing Institute} \and Andrew R.\ Wade\footnotemark[1]}

\maketitle

\begin{abstract}
For a multidimensional driftless diffusion in an unbounded, smooth, sub-linear generalized parabolic domain, with oblique reflection from the boundary,
we give natural conditions under which either explosion occurs, if the domain narrows sufficiently fast at infinity, or else
there is superdiffusive transience, which we quantify with a strong law of large numbers. For example, in the case of a planar domain, explosion occurs if and only if the area of the domain is finite. 
We develop and apply novel semimartingale criteria for studying explosions and establishing strong laws, which are of independent interest.
\end{abstract}

\medskip


\medskip


\section{Introduction}
\label{sec:intro}

We study the asymptotic behaviour of a multidimensional diffusion in an unbounded, generalized parabolic domain, with oblique reflection from the boundary.
The oblique reflection is such that the diffusion is transient, and the main phenomena we explore here are (i) explosion
(meaning that the process `reaches infinity' in finite time) 
if the domain narrows sufficiently fast at infinity, versus (ii) superdiffusivity, if explosion is absent but the domain grows sub-linearly.
We identify the sharp phase transition between (i) and~(ii) in terms of the growth rate of the boundary,
and quantify~(ii) via a strong law of large numbers. 
Our model can be viewed as a 
stochastic process with constraints exhibiting \emph{anomalous diffusion}.
We emphasize that the phenomena
we exhibit here are present even for the case of reflecting Brownian motion, although we do treat more general diffusions with no interior drift.

Reflecting diffusions are fundamental stochastic processes, motivated from kinetic theory of gases, queueing, communication or inventory theory, and, more recently, financial models: see end of Section~\ref{sec:diffusion} below for a brief discussion.
A large literature studies reflecting diffusions in bounded domains (see e.g.~\cite{sv,ls,lr,kr}).
In unbounded domains, if the  interior drift is constant, then the most
subtle case is when the drift is zero. Domains that are orthants or cones are classical (see e.g.~\cite{fr,vw,williams,mw}), and typically
behaviour is diffusive, even in the transient case.
Generalized parabolic domains\footnote{Pinsky~\cite[p.~677]{pinsky} uses the term \emph{horn-shaped}, which has several distinct uses in the literature.}
were considered by Pinsky~\cite{pinsky}
in the case of normal reflection and canonical covariances in the interior; 
in that case
there can be no explosion, the planar case is always recurrent. It is expected that in the case of normal reflection,  transience, present in higher dimensions, is diffusive.  In a discrete setting~\cite{mmw}, we studied planar generalized parabolic domains with normal reflection 
(more generally, \emph{opposed reflection} where reflection angles from the upper and lower boundaries are equal and opposite),
and general covariance matrices in the interior, but again any transient behaviour is expected to be diffusive. 
Thus to seek anomalous diffusion we are led to considering oblique reflections in domains of sub-linear growth, so that
the reflection is both frequent and sufficiently strong to drive the superdiffusive escape.  
The present paper is
to the best of our knowledge
the first work on sharp quantification of transience for reflecting diffusions in unbounded domains, and the first to
exhibit explosion of Brownian motion in this context.

We describe informally a special case of the model that this paper studies, to provide a sketch of the main phenomena
and to motivate the formal (and more general) definitions that we defer till Section~\ref{sec:diffusion} below.
Let $\cD$ be a domain in $\R^2$ defined by $\cD = \{ (x,y) : x \in \RP, |y| \leq b(x) \}$,
where $b : \RP \to \RP$ is a smooth function with $b(x) >0$ for $x >0$. The full range of phenomena 
are seen already in the case where $b(x) = a x^\beta$ for $x \geq x_0 >0$, say, where $a>0$ and $\beta \in \R$.
Informally, the evolution of $Z_t \in \cD$ is described by the 
stochastic differential equation (SDE)
\begin{equation}
    \label{eq:informal-SDE}
    \ud Z_t = \phi (Z_t ) \ud L_t + \Sigma^{1/2} ( Z_t ) \ud W_t, ~\text{for} ~ 0 \leq t < \taue ,
\end{equation}
where $\taue \in(0,\infty]$ is a potential explosion time,
$W$ is a planar Brownian motion, 
$\Sigma^{1/2}$ is a square root of a bounded covariance matrix $\Sigma$, and $\phi$ is a smooth, bounded vector field on $\pcD$, which governs the oblique reflection
through $L$, the local time of $Z$ on $\pcD$. 
We permit $\Sigma = \Sigma (z)$ to vary smoothly with $z \in \cD$,
but we assume that $e_y^\tra \Sigma (z) e_y \to \sigma^2 \in (0,\infty)$ for $z = (x,y)$ with $x \to \infty$,
where $e_y$ is the unit vector in the vertical direction (our more general assumptions below give a more general meaning to $\sigma^2$).
As an example of $\phi$, we may take reflection
at angle $\alpha > 0$ relative to the inwards pointing normal vector, where positive $\alpha$ means that
the angle is in the direction of increasing horizontal coordinate: see Figure~\ref{fig:picture}. 

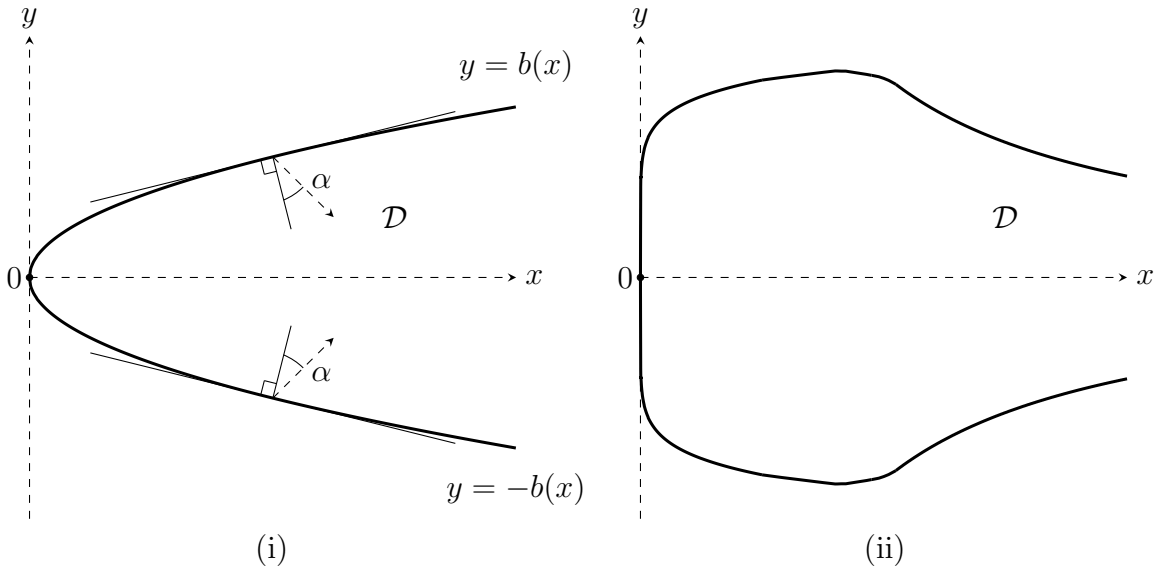
\begin{figure}
\begin{center}
\begin{tikzpicture}[domain=0:8, scale = 0.8]
\filldraw (0,0) circle (1.5pt);
\node at (-0.25,0) {$0$};
\draw[black, line width = 0.40mm]   plot[smooth,domain=0:8,samples=500] ({\x},  {(\x)^(1/2)});
\draw[black, line width = 0.40mm]   plot[smooth,domain=0:8,samples=500] ({\x},  {-(\x)^(1/2)});
\node at (6,1.0) {$\cD$};
\draw[black,->,>=stealth,dashed] (0,0) -- (8,0);
\node at (8.3, 0)       {$x$};
\draw[black,->,>=stealth,dashed] (0,-4) -- (0,4);
\node at (0,4.3)       {$y$};
\draw (4,2) -- (7,2.75);
\draw (1,1.25) -- (4,2);
\draw (4,2) -- (4.3,0.8);
\draw (3.8,1.95) -- (3.86,1.71);
\draw (4.06,1.76) -- (3.86,1.71);
\draw (4.5,1.5) arc (-45:-67.5:1);
\node at (8, 3.5)       {$y = b(x)$};
\node at (8, -3.5)      {$y = - b(x)$};
\draw (1,-1.25) -- (4,-2);
\draw (4,-2) -- (4.3,-0.8);
\draw (4,-2) -- (7,-2.75);
\draw (4.5,-1.5) arc (45:67.5:1);
\draw (3.8,-1.95) -- (3.86,-1.71);
\draw (4.06,-1.76) -- (3.86,-1.71);
\draw[black,->,>=stealth,dashed] (4,-2) -- (5,-1);
\draw[black,->,>=stealth,dashed] (4,2) -- (5,1);
\node at (4.8, 1.6)       {$\alpha$};
\node at (4.8, -1.6)       {$\alpha$};
\end{tikzpicture}
\begin{tikzpicture}[domain=0:8, scale = 0.8]
\filldraw (0,0) circle (1.5pt);
\node at (-0.25,0) {$0$};
\draw[black, line width = 0.40mm] plot[smooth,domain=0:2,samples=500] ({\x},  {3*(\x)^(1/8)});
\draw[black, line width = 0.40mm] plot[smooth,samples=200, tension=1] coordinates {(2,3.271523) (3,3.4) (3.3,3.42) (3.5,3.4) (3.8,3.35) };
\draw[black, line width = 0.40mm] plot[smooth,samples=200, tension=1] coordinates {(3.8,3.35) (4.0,3.3) (4.2, 3.2)};
\draw[black, line width = 0.40mm]   plot[smooth,domain=4.2:8,samples=500] ({\x},  {13.44*(\x)^(-1)});
\draw[black, line width = 0.40mm] plot[smooth,domain=0:2,samples=500] ({\x},  {-3*(\x)^(1/8)});
\draw[black, line width = 0.40mm] plot[smooth,samples=200, tension=1] coordinates {(2,-3.271523) (3,-3.4) (3.3,-3.42) (3.5,-3.4) (3.8,-3.35) };
\draw[black, line width = 0.40mm] plot[smooth,samples=200, tension=1] coordinates {(3.8,-3.35) (4.0,-3.3) (4.2, -3.2)};
\draw[black, line width = 0.40mm]   plot[smooth,domain=4.2:8,samples=500] ({\x},  {-13.44*(\x)^(-1)});
\node at (6,1.0) {$\cD$};
\draw[black,->,>=stealth,dashed] (0,0) -- (8,0);
\node at (8.3, 0)       {$x$};
\draw[black,->,>=stealth,dashed] (0,-4) -- (0,4);
\node at (0,4.3)       {$y$};
\end{tikzpicture}
\qquad\qquad (i)  \qquad\qquad\qquad\qquad\qquad\qquad\qquad\qquad\qquad (ii)
\end{center}
\vspace{-4mm}
\caption{\label{fig:picture} Two planar generalized parabolic domains. Case~(i) has $b(x) = a x^{1/8}$ for large~$x$, while (ii),
  a generalized \emph{reciprocal} parabolic domain, 
has $b(x)= a x^{-1}$ for large~$x$.
We give conditions under which the reflecting diffusion in (i) satisfies a law of large numbers with growth
rate $t^{8/9}$, while (ii) yields exponential growth. If domain (ii) narrows a little faster, so that $\int^\infty_0 b(x) \ud x < \infty$, there is explosion: the diffusion is driven to infinity in finite time.}
\end{figure}

The fact that $\alpha >0$ and $\lim_{x \to \infty} b'(x) = 0$ means that the process accumulates an effective positive drift in the horizontal direction whenever it visits the boundary.
The narrower the domain, the more often the process visits the boundary. For example, in the case $\beta=0$ one has
a strip-like domain, and it is natural to expect that the process is transient to the right with a positive speed (ballisticity):
a formal statement is a special case of our results below (see~Example \ref{ex:lln}). If $\beta <0$ then drift accumulates faster, so transience is
super-linear. In fact, in very narrow domains acceleration is so rapid that explosion can occur: indeed, explosion occurs if and only if $\beta < -1$.
On the other hand, for $\beta = 1$ we are in the classical situation of reflecting diffusion in a \emph{wedge},
and here both recurrence and transience are possible~\cite{vw,williams}. 
When $|\beta| <1$ we quantify the rate of escape via a superdiffusive law of large numbers.

To give some intuition behind our main result (Theorem~\ref{thm:lln} below)
and to illustrate a little more concretely how the main phenomena
that we investigate arise, we describe an heuristic comparison with reflecting Brownian motion in an interval.
For simplicity of the following heuristic discussion, we remain in the planar case displayed in  Figure~\ref{fig:picture}, take $\Sigma$ to be the identity (so the process is Brownian motion in the interior of $\cD$), and we write $c_0 = \cos \alpha$ and $s_0 = \sin \alpha$; more general definitions of $c_0, s_0$ come later.

We try to estimate how much effective horizontal drift the process accumulates via reflections.
Suppose the process is at horizontal position~$x$.
Over short time-scales, 
imagine we may approximate the behaviour of the vertical coordinate by a diffusion on the interval $[-b(x),b(x)]$.
This diffusion has zero drift and infinitesimal variance about~$\sigma^2$.
The (vertical) reflection is effectively of magnitude~$c_0$. After a transformation, this is
equivalent to unit-magnitude reflection for Brownian motion on $[-b(x)/c_0,b(x)/c_0]$ with variance $\sigma^2/c_0^2$.
By an heuristic renewal argument similar to~\cite[pp.~679--680]{pinsky}, this process should accrue boundary local time on average at rate about $\frac{\sigma^2}{2c_0 b(x)}$.
 This manifests in the two-dimensional process as an effective drift in the horizontal direction of roughly $\frac{s_0 \sigma^2 }{2c_0 b(x)}$.
Consequently, one can imagine that the large-scale behaviour of the horizontal coordinate $X_t$ of the reflected Brownian motion $Z_t$ in $\cD$ is by~\eqref{eq:informal-SDE} well-approximated by the SDE
\begin{equation}
\label{eq:X-approximation}
\ud X_t = \frac{s_0 \sigma^2 }{2c_0 b(X_t)} \ud t + \ud \widetilde{W}_t , ~\text{for} ~ 0 \leq t < \taue , \end{equation}
where $\widetilde{W}$ is one-dimensional Brownian motion. 
We take $\beta < 1$, so, roughly speaking, the drift dominates the Brownian martingale;
the discrete-time analogue is the supercritical Lamperti problem~\cite[\S 3.12]{mpw}.
Thus, ignoring the diffusion term in~\eqref{eq:X-approximation} and integrating the resulting ODE, 
one arrives at
$B(X_t) \approx \frac{s_0 \sigma^2 }{2c_0} t$ for $t < \taue$,
where $B(x) := \int_0^x b(y) \ud y$. 
The role of $B(\infty) := \lim_{x \to \infty} B(x)$
now becomes apparent.
Indeed, considering $t = \sigma_r := \inf \{ t \in \RP : X_t \geq r \}$
we get $\sigma_r \approx  \frac{2 c_0}{s_0 \sigma^2 } B(r)$. Explosion, i.e., finiteness of $\taue := \lim_{r \to \infty} \sigma_r$, is thus linked to finiteness of $B(\infty)$.
A natural conjecture on the basis of this heuristic is that
\begin{itemize}
    \item[(i)] explosion occurs if and only if $B(\infty) < \infty$; and
    \item[(ii)] if $B(\infty) = \infty$, we have the strong law $\lim_{t\to\infty} t^{-1} B(X_t) = \frac{s_0 \sigma^2 }{2c_0}$, a.s.
\end{itemize}
  Thus $\beta = -1$ is the critical boundary exponent. Somewhat more formally, one could arrive at the same conjecture~(i) by applying the Feller explosion test to the SDE~\eqref{eq:X-approximation}~\cite[Cor.~4.4, p.~82]{ce}. However, all of this reasoning is based on a one-dimensional approximation, and is a long way from a rigorous proof; our proofs are based instead on some martingale ideas that we anticipate
will have wider applicability. These martingale ideas
are the focus of Section~\ref{sec:martingales} below.

The main contribution of the paper is to formulate and establish precise versions of (i) and (ii) for domains in $\RPRd$, generalizing the type described above, 
and oblique reflections with positive components in the axial ($\RP$) direction. 
Note that in the planar case, the criterion $B(\infty) < \infty$ is equivalent to finite area of the domain, while in higher dimensions, finiteness of the volume is a strictly stronger condition guaranteeing explosion.
We give the detailed formulation of the model and the main result (Theorem~\ref{thm:lln}) in Section~\ref{sec:diffusion}.

\section{Model and main results}
\label{sec:diffusion}

Write $\| \, \cdot \, \|_d$ for the Euclidean norm on $\R^d$, $d \in \N := \{1,2,\ldots\}$. Denote the unite sphere in $\R^d$ by $\Sp{d-1} := \{ u \in \R^{d} : \| u \|_{d} = 1\}$.
Define $\RP:=[0,\infty)$ and let $b : \RP \to \RP$.  Fix $d \in \N$ and consider the closed domain $\cD \subseteq \R^{d+1}$ given by
\begin{equation}
 \label{eq:domain-def}   
 \cD := \bigl\{ z = (x,y) \in \RPRd : \| y \|_d \leq b (x) \bigr\} .
 \end{equation}
Write $\pcD:=\{ z = (x,y) \in \RPRd : \| y \|_d = b (x)\}$ for the boundary of $\cD$ in $\R^{d+1}$. Let $\cM^+_{d+1}$ be the set of real positive definite $(d\times d)$-matrices. We view vectors as column vectors and write
$\langle u, v \rangle = u^\tra v$
for the Euclidean inner product of $u,v \in \R^{d+1}$.

We will consider a process $Z = (Z_t)_{t \in [0,\taue)}$, with $Z_t \in \cD$ for $t < \taue$, where $\taue \in (0,\infty]$ is a possibly finite random explosion time (see Appendix~\ref{sec:construction} below for the definition of the space of trajectories of $Z$). 
The process $Z$ will be driven by a standard Brownian motion $W = (W_t)_{t \in \RP}$ on $\R^{d+1}$ and the dynamics will be specified by
an instantaneous covariance function  $\Sigma : \cD \to \cM^+_{d+1}$ with the symmetric square-root $\Sigma^{1/2}$, and a vector field 
 $\phi : \pcD \to \R^{d+1}$. Accompanying $Z$ will be  $L = (L_t)_{t \in [0,\taue)}$, where $L_t \in \RP$
 is the local time of $Z$ at~$\pcD$ up to time~$t$. The triple $(Z, L, \taue)$ will be our
 object of interest, 
 where 
\begin{equation}
\label{eq:SDE-for-Z}
 \begin{split}
 Z_t & = z + \int_0^t \Sigma^{1/2} ( Z_s) \ud W_s + \int_0^t \phi ( Z_s) \ud L_s , \\
& \qquad\qquad \text{ and } L_t = \int_0^t \1 { Z_s \in \pcD } \ud L_s , ~~~ \text{for } t \in [0,\taue).
\end{split}
  \end{equation}
In Appendix~\ref{sec:construction} below
we give a formal definition of the 
solutions to~\eqref{eq:SDE-for-Z} and 
establish existence and pathwise uniqueness under natural assumptions, 
see Theorem~\ref{thm:existence}. Some extra work is required compared to the standard literature~\cite{ls, sv} since $Z$ may explode because of its local time if it spends too much time close to the boundary $\pcD$.

Our initial assumptions on the domain $\cD$ are as follows.
\begin{description}
\item\namedlabel{ass:domain1}{$\text{D}_1$} 
Let $b$ be continuous on $\RP$, 
with $b(0) = 0$ and $b(x) >0$ for $x >0$. Suppose that $b$ is
twice continuously differentiable on $(0,\infty)$, such that 
(i) $\liminf_{x \to 0} ( b(x) b'(x)  ) > 0$,
and (ii) $\lim_{x \to 0} (b'' (x) / b'(x)^3)$ exists in $(-\infty,0]$.
\end{description}

\begin{remark}
\label{rem:b-small-x}
It follows from~\eqref{ass:domain1} that $\cD$ is a $C^2$ domain: see Lemma~\ref{lem:C2-domain} below.
In particular, the conditions~(i) and~(ii) in assumption~\eqref{ass:domain1}
ensure that $\cD$ is sufficiently smooth at the origin. 
This excludes the possibility of the diffusion being trapped 
at a boundary point (cf.~\cite{vw,williams}).
Conditions~(i) and~(ii)
are satisfied if $b(x) = a_0 x^\alpha( 1 + \eps_0(x))$ as $x \to 0$
for $a_0 \in (0,\infty)$, $\alpha \in (0,1/2]$, and $\eps_0$
 twice continuously differentiable with $\eps_0 (x) = o(1)$, $\eps'_0(x) = o(1/x)$, and $\eps''_0(x) = o(1/x^2)$ as $x \to 0$, for instance.
We emphasize, however, that the precise asymptotics of~$b$ at~$0$ are not
important: the asymptotics of~$b$ at~$\infty$ are crucial for our quantitative
result in Theorem~\ref{thm:lln} below. 
Finally, note that since $\lim_{x \to0} b(x) = 0$, (i)~implies that $\lim_{x \to 0} b' (x) =  \infty$.
\end{remark}

The next assumption applies to the covariance function~$\Sigma$.
\begin{description}
\item\namedlabel{ass:variance}{\text{C}} 
Let $\Sigma : \cD \to \cM^+_{d+1}$ be bounded, (globally) Lipschitz, and uniformly elliptic, i.e., there exists $\delta >0$ such that,
 for every $u \in \Sp{d}$ and all $z \in \cD$, we have
$u^\tra \Sigma (z) u \geq \delta$. 
\end{description}
 When we say $\Sigma$ is bounded and Lipschitz, the implicit norm is the matrix (operator) norm $\| \Sigma \|_{{\rm op}} = \sup_{u \in \Sp{d}} \|  \Sigma u \|_{d+1}$.
Since $\| \Sigma^{1/2} (z) \|^2_{\rm op} = \sup_{u \in \Sp{d}} ( u^\tra \Sigma (z) u )$,
the largest eigenvalue of $\Sigma (z)$, boundedness of $\Sigma$ implies boundedness of $\Sigma^{1/2}$.

We write a generic point $z \in \cD$ in coordinates as $z = (x,y)$, where $x \in \RP$ and $y \in \R^d$ with $\| y \|_d \leq b(x)$.
Let $e_x := (1,0) \in \Sp{d}$ denote the unit vector in the $x$-direction, 
and for $u \in \Sp{d-1}$ let $e_u := (0, u) \in \Sp{d}$.
Any $z = (x,y) \in \cD$ can be written as $z = x e_x + \| y \|_d e_{\hat y}$,
where $\hat y := y / \| y \|_d$ for $\| y \|_d > 0$. In particular,
if $z \in \pcD$ then $z = (x, \hat y b(x)) = x e_x + b(x) e_{\hat y}$.

Next we impose conditions on the boundary vector field~$\phi$. 
We write $\phi_x (u) := \phi ( x, u b(x))$ for the element of the vector field
indexed by $(e_x,e_u)$ coordinates. Let $n_x (u)$ denote the
 inwards-pointing unit normal vector to $\pcD$ at $(x, u b(x)) \in \pcD$; by~\eqref{ass:domain1},
  $n_x(u)$ is uniquely defined. We  assume that $\phi$ is smooth,  
that at each point on $\pcD$ it 
has a uniformly
positive component in the normal direction.
\begin{description}
\item\namedlabel{ass:vector-field}{\text{V}} 
Suppose that $\phi : \pcD \to \R^{d+1}$ is a $C^2$ vector field, with $\sup_{z \in \pcD} \| \phi(z) \|_{d+1} < \infty$,
and 
\begin{align}
\label{eq:orthogonal-component-strict}
\inf_{x\geq0} \inf_{u\in\Sp{d-1}} \langle \phi_x(u) , n_x (u) \rangle  > 0 .
\end{align}
\end{description}

The initial assumptions on the
domain~\eqref{ass:domain1}, instantaneous variance~\eqref{ass:variance} and vector field~\eqref{ass:vector-field}
guarantee the existence and uniqueness of the solutions of SDE~\eqref{eq:SDE-for-Z}, see Appendix~\ref{sec:construction}. 
Note that~\eqref{eq:orthogonal-component-strict} is a minimal extension to unbounded domains of the condition for existence in bounded domains~\cite{ls}. 
Our main result,
Theorem~\ref{thm:lln} below,
gives a precise description of the asymptotic behaviour of the process $Z$.
It requires the following additional assumptions on 
the function~$b$,
the instantaneous covariance function $\Sigma$, and the vector field~$\phi$.
 Define
\begin{equation}
\label{eq:beta-def}
 \beta := \limsup_{x \to \infty} \frac{x b'(x)}{b(x)}   .\end{equation}
Assumption~\eqref{ass:domain2} concerns the behaviour of~$b$ near infinity,
while~\eqref{ass:lln} introduces quantitative assumptions on~$\Sigma$ and~$\phi$.
Write `$\trace$' for the trace operator.

\begin{description}
\item\namedlabel{ass:domain2}{$\text{D}_2$} 
Suppose that $\lim_{x \to \infty} b'(x) =\lim_{x \to \infty} b''(x) = \lim_{x \to \infty} b(x)b''(x) = 0$, and that $\beta$ as defined at~\eqref{eq:beta-def} satisfies $\beta < 1$.
\item\namedlabel{ass:lln}{\text{A}} 
Suppose that there exists $\sigma^2 \in (0,\infty)$ for which
\begin{equation}
\label{eq:Sigma-limit}
 \lim_{x \to \infty} \sup_{y:(x,y) \in \cD}  \left| \trace \Sigma (x,y) - e^\tra_x \Sigma (x,y) e_x - \sigma^2 \right| = 0 .\end{equation}Finally, suppose that there exist $s_0, c_0 \in (0,\infty)$ such that
\begin{align}
\label{eq:e1-projection}
 \lim_{x \to \infty} \sup_{u \in \Sp{d-1}} \left| \langle \phi_x (u) , e_x \rangle - s_0 \right|& = 0, \\
\label{eq:z-projection}
\lim_{x \to \infty} \sup_{u \in \Sp{d-1}} \left| \langle \phi_x (u) , -e_u \rangle - c_0 \right| & = 0.
\end{align}
\end{description}
We make some comments on these assumptions in Remarks~\ref{rems:lln} below, after stating our main theorem. 
The last piece
of notation that we need is 
\begin{equation}
\label{eq:B-def}
 B(x) := \int_0^x b(u) \ud u , \text{ for } x \in \RP. \end{equation}
Since $b$ is bounded on compact intervals, and $b(x) > 0$ for $x >0$, 
$B(x) <\infty$ for all $x \in \RP$,
and $x \mapsto B(x)$
is continuous and strictly increasing.
Set $B (\infty ) : = \lim_{x \to \infty} B(x) \in (0, \infty]$.

If $t < \taue$, then in components write $Z_t = (X_t, Y_t) \in \cD$, with $X_t \in \RP$,  $\| Y_t \|_d \leq b (X_t)$.
Let $\Pr_z$ denote the law of $Z$ started from $Z_0 = z \in \cD$, and let $\Exp_z$
be the corresponding expectation; Theorem~\ref{thm:existence} shows that these
are well defined. 
Define the passage times
\begin{equation}
    \label{eq:sigma-def}
    \sigma_r := \inf \{ t \in \RP : X_t \geq r \}.
\end{equation}
Then (see Appendix~\ref{sec:construction}) $\taue := \lim_{r \to \infty} \sigma_r$. The next result states our main dichotomy: $B(\infty) < \infty$ implies explosion,
while if $B(\infty) = \infty$, there is non-explosive transience 
quantified via a superdiffusive
strong law of large numbers. 

\begin{theorem}
\label{thm:lln}
Suppose that~\eqref{ass:domain1}, \eqref{ass:domain2}, \eqref{ass:variance}, \eqref{ass:vector-field}, and~\eqref{ass:lln} hold. 
\begin{thmenumi}[label=(\roman*)]
\item[(i)] If  $B$ as defined at~\eqref{eq:B-def} satisfies $B(\infty) < \infty$,
then $\sup_{z \in \cD} \Exp_z \taue < \infty$, and
\[ \lim_{t \uparrow \taue} X_t = \lim_{t \uparrow \taue} L_t = \infty, \text{ $\Pr_z$-a.s.~for every $z \in \cD$.} \]
\item[(ii)] If $B(\infty) = \infty$, then, for every $z \in \cD$, $\Pr_z ( \taue = \infty ) =1$, and
\begin{equation}
    \label{eq:lln}
  \lim_{t \to \infty} \frac{B ( X_t )}{t}
 = \lim_{t \to \infty} \frac{B ( \| Z_t \|_{d+1} )}{t} = \frac{s_0 \sigma^2}{2c_0},  \text{ $\Pr_z$-a.s.}, \end{equation}
 and
\[ \lim_{r \to \infty} \frac{ \Exp_z \sigma_r}{B(r)} = \frac{2 c_0}{s_0 \sigma^2} .\]
 Moreover, for every $z \in \cD$,
 \begin{equation}
    \label{eq:lln-local-time}
  \lim_{t \to \infty} \frac{s_0 L_t}{X_t}
= 1 ,  \text{ $\Pr_z$-a.s.} \end{equation}
\end{thmenumi}
\end{theorem}

\begin{remarks}
\phantomsection
\label{rems:lln}
\begin{myenumi}[label=(\alph*)]
\setlength{\itemsep}{0pt plus 1pt}
\item
If $\beta$ defined at~\eqref{eq:beta-def} is finite, then $b(x) \leq x^{\beta+o(1)}$,
as for any $\eps >0$,
\[ \int_{x_1}^x \frac{b'(u)}{b(u)} \ud u \leq  (\beta+\eps) \int_{x_1}^x \frac{1}{u} \ud u , \text{ for all } x > x_1 ,\]
for some $x_1 \in (0,\infty)$ sufficiently large, 
and hence 
\begin{equation}
\label{eq:beta-b} \limsup_{x \to \infty} \frac{\log b(x)}{\log x} \leq   \beta  .\end{equation}
\item \label{rems:lln-b}
For $\beta <1$ we have from the preceding remark that, for some $\eps >0$,
$b(x) = O(x^{1-\eps})$, so that $B(x) = O(x^{2-\eps})$.
Thus~\eqref{eq:lln} shows that $\liminf_{t \to \infty} ( X_t^{2-\eps} / t) > 0$, a.s., which means that $X_t$ is (strictly) \emph{superdiffusive}. In contrast, $B$ can grow arbitrarily slowly, meaning that $X_t$ can grow arbitrarily fast; see Example~\ref{ex:lln} for
some representative examples.
\item 
Roughly speaking, assumption~\eqref{eq:Sigma-limit} says that the contribution to the total infinitesimal variance coming from
directions orthogonal to the horizontal stabilizes to a limit, $\sigma^2$; this variance can be distributed 
in any proportion among the $d$ components, that can vary across the domain, provided the ellipticity condition is satisfied.
\item 
Since $b(x) = o(x)$, we have $\sup_{u \in \Sp{d-1}} \| n_x (u) + e_u \|_{d+1} \to 0$ as $x \to \infty$.
Assumption~\eqref{eq:z-projection} thus shows that $\inf_{x \geq x_1} \inf_{u \in \Sp{d-1}} \langle \phi_x (u) , n_x(u) \rangle \geq c_0/2$ for $x_1$ sufficiently large,
so one may replace~\eqref{eq:orthogonal-component-strict} with the apparently weaker assumption that
\begin{equation}
\label{eq:orthogonal-component}
\langle \phi_x (u) , n_x (u) \rangle > 0, \text{ for all } x \in \RP \text{ and all } u \in \Sp{d-1} .
\end{equation}
\item 
We are not aware of previous work on explosion to infinity driven by reflection;
the closest relatives we have seen in the literature are
the phenomenon of corner trapping for reflecting processes in domains with non-smooth boundaries~\cite{vw},
or explosion for diffusions on manifolds~\cite{grigoryan}.
The general existence results for diffusions with oblique reflections
given in~\cite{ls} and~\cite{sv} do not admit explosion.
Indeed, \cite{ls} deals with bounded domains. 
In~\cite{sv} the domains are defined via $\cD = \{ z \in \R^{d+1} : \psi (z) > 0 \}$ for some bounded $\psi : \R^{d+1} \to \R$ with two bounded continuous derivatives. 
If the function $b:\RP\to\RP$ tends to zero as $x\to\infty$, the second derivative of the corresponding $\psi(z)=1- (\|y\|_d/b(x))$ 
is not bounded as $x\to\infty$
as the gradient of $\psi$ must change increasingly rapidly around the cross-sectional boundary at horizontal location~$x$. This phenomenon will occur with any $\psi$ (when $b$ tends to zero), satisfying other assumptions in~\cite{sv}. Hence~\cite{sv} also excludes the narrow domains that exhibit explosion as considered here.
\end{myenumi}
\end{remarks}

Here is an example that illustrates Theorem~\ref{thm:lln}.

\begin{example}
\label{ex:lln}
Take $d\in \N$, and suppose that $b$ satisfies
\[ b(x) = a_0 x^\alpha (1+\eps_0(x)) \text{ as } x \to 0, ~\text{and}~ b(x) = a_\infty x^\beta (1+\eps_\infty (x)) \text{ as } x \to \infty,\]
where $a_0, a_\infty \in (0,\infty)$, $\alpha \in [1/2,1)$, $\beta \in (-\infty, 1)$,
and $\eps_0(x) = o(1)$, $\eps'_0(x) = o(1/x)$, $\eps_0''(x) = o(1/x^2)$ as $x \to 0$ while $\eps_\infty (x) = o(1)$, $\eps'_\infty (x) = o(1/x)$, $\eps_\infty''(x) = o(1/x^2)$ as $x \to \infty$. 
Then~\eqref{ass:domain1} holds
(see Remark~\ref{rem:b-small-x})
and so does~\eqref{ass:domain2}, with~$\beta$ coinciding with~\eqref{eq:beta-def}. Let $\Sigma (z) = v^2 I_{d+1}$ (constant), where $I_{d+1}\in \cM_{d+1}^+$ is the identity;
then~\eqref{ass:variance} holds for $\sigma^2 = dv^2$.
Take $\phi_0 (u) = n_0 (u) = e_x$ and $\phi_x (u ) = n_x (u) + s_0 e_x  -c_0 e_u$ for $x >0$, where $s_0 , c_0 >0$.
If $\alpha = \arctan ( s_0 / c_0)$, then $\phi_0(u)$ is the vector obtained by rotating $n_x(u)$
in the plane containing $0$ and $n_x(u)$ by angle $\alpha$.
Then $B(x) \sim \frac{a_\infty}{1+\beta} x^{1+\beta}$ for $\beta > -1$, and $B(x) \sim a_\infty \log x$ for $\beta = -1$,
and $B(\infty) < \infty$ for $\beta < -1$.
Theorem~\ref{thm:lln} then yields the following for all $z \in\cD$.
\begin{itemize}
\item If $\beta < -1$, then $\Exp_z \taue < \infty$, i.e., explosion occurs.
\item If $\beta = - 1$, then $\Pr_z (\taue = \infty) = 1$, and
\[  \lim_{t \to \infty} \frac{ \log X_t}{t}  = 
 \lim_{t \to \infty} \frac{ \log \| Z_t \|_{d+1}}{t}  = 
\lim_{t \to \infty} \frac{ \log L_t}{t}  = \frac{d v^2 \tan \alpha }{ 2a_\infty  } , \text{ $\Pr_z$-a.s.} \]
\item If $\beta \in (-1,1)$, then $\Pr_z (\taue = \infty) = 1$, and,
$\Pr_z$-a.s.
\[ \lim_{t \to \infty}  t^{-\frac{1}{1+\beta}} X_t    
=  
\lim_{t \to \infty}  t^{-\frac{1}{1+\beta}} \| Z_t \|_{d+1}
=  
s_0 \lim_{t \to \infty}  t^{-\frac{1}{1+\beta}} L_t =
\left( \frac{(1+\beta) d v^2 \tan \alpha}{ 2 a_\infty  } \right)^{\frac{1}{1+\beta}}.    \]
\end{itemize}
Note that if $\beta \in (0,1)$, then $X$ is sub-ballistic (i.e., has speed zero: $\lim_{t \to \infty} X_t /t =0$, a.s.),
it has positive speed when $\beta=0$, 
and for $\beta \in (-1,0)$,
$\lim_{t\to\infty} X_t /t = \infty$, a.s.
\end{example}

We now comment briefly on some motivating applications for reflecting diffusions and refer to the literature.

\emph{Ideal gas dynamics.}
Consider an ideal gas in a container. 
A gas particle moves at constant velocity until it   either hits the domain boundary, where 
it reflects, 
or collides with another particle. If the density of the gas is sufficiently low (the  Knudsen regime~\cite{knudsen}), collisions between particles
can be neglected, and changes in velocity occur only on reflection at the boundary. 
Resulting billiards models may be deterministic  (see e.g.~\cite{tabachnikov})
or stochastic (e.g.~\cite{cpsv,evans,mvw,cdcmw}), depending on the reflection rule.
By contrast, in the high-density regime, intermolecular collisions are important. Tracking the dynamics of a single particle, one now observes Brownian motion
in the domain interior. Thus our reflecting diffusions can be motivated by single-particle dynamics in high-density ideal gases.

\emph{Queueing and communication networks.}
The domain $\RP^d$ has received particular attention over many years due to its connection
with stochastic models of queueing systems, loss networks, and communication systems. For example,
if there are $d$ queues (or $d$ customer classes) the process of queue lengths can often
be described by a Markov process on $\RP^d$ with boundary reflection; different service or transmission protocols lead to different models.
While these Markovian models often have discrete state-space (e.g.~$\ZP^d$), diffusion approximation and certain limiting regimes (in particular, heavy traffic) lead
naturally to reflecting diffusions:~see e.g.~\cite{mp,rr,harrison,foschini,bd}, among many other papers.

\emph{Other motivation.}
In mathematical finance, reflecting diffusions appear both directly as pricing models (see e.g.~\cite{hhl}) 
and via their intimate relation to  diffusions with rank-dependent interactions (e.g.~\cite{ipb,pp}).
In one dimension, systems of interacting Brownian motions that mutually reflect have 
been studied in the context of the KPZ universality class~\cite{wfs} and as scaling limits
of certain discrete interacting particle systems~\cite{gs}. A recent statistical application 
of reflecting processes is set estimation~\cite{cflp}.

The outline of the rest of this paper is as follows.
Section~\ref{sec:martingales} presents some results on semimartingales, with potential explosion, that will form the basis for our analysis; these
involve martingale-type criteria for estimating escape probabilities and expected hitting times, and for analysing explosion times. Section~\ref{sec:exit} turns to the reflecting diffusion given by~\eqref{eq:SDE-for-Z}, and establishes  (Theorem~\ref{thm:exit}) that the expected time to exit any bounded set is finite, an important and non-trivial ingredient in our proofs, as our assumptions that guarantee transience hold only for large $x$. Section~\ref{sec:lln} then presents the proof of Theorem~\ref{thm:lln}, which uses a 
suitable Lyapunov function to bring the results of Section~\ref{sec:martingales} to bear on the reflecting diffusion.
Finally, Appendix~\ref{sec:construction} discusses existence and uniqueness for the SDE~\eqref{eq:SDE-for-Z}.

\section{Explosions and growth bounds for semimartingales}
\label{sec:martingales}

\subsection{Overview and notation}
\label{sec:martingales-overview}

This section develops some semimartingale
tools for 
studying the quantitative asymptotic behaviour of possibly explosive semimartingales
via suitable Lyapunov functions.

Fix a probability space $(\Omega, \cF, \Pr)$
and a complete right-continuous filtration $(\cF_t)_{t\in\RP}$.
Let $\cF_\infty := \sigma ( \cup_{t\in\RP} \cF_t )$.
Consider an $(\cF_t)$-progressively measurable process 
$\kappa = (\kappa_t)_{t\in\RP}$,
taking values in $[0,\infty]$, 
where reaching the state $\infty$ in finite time 
represents explosion. Let $\cT$ denote the set of all $[0,\infty]$-valued stopping times with respect to $(\cF_t)_{t\in\RP}$.

For any $\ell, r \in \RP$ and stopping time $T\in \cT$, define the first-entry times (after $T$)
\begin{align}
    \label{eq:lambda-rho}
    \begin{split}
\lambda_{\ell,T} & :=T+ \inf \{ s\in\RP : T<\infty,\, \kappa_{T+s} \leq \ell \},\\
\rho_{r,T}& :=T+\inf \{ s\in\RP : T<\infty, \, \kappa_{T+s} \geq r \},
\end{split}
\end{align}
where we adopt the convention $\inf \emptyset := +\infty$.
If $T=0$, we denote
$\lambda_\ell := \lambda_{\ell,0}$
and
$\rho_r := \rho_{r,0}$.
The almost sure  limits $\rho_\infty := \lim_{r \to \infty} \rho_r$
 and
$\rho_{\infty,T} := \lim_{r \to \infty} \rho_{r,T}$
exist by monotonicity. 
Note that, provided $\kappa$ is right-continuous with left limits (rcll), by~\cite[Ch.~III,~Prop.~3.3]{ry} we have 
$\lambda_{\ell, T},\rho_{r,T}\in\cT$
for all $r,\ell\in\RP$
and hence
$\rho_\infty\in \cT$
(see~\cite[p.~46]{ry}). 
On the event 
$\{\rho_\infty < \infty\}$, 
we say that 
\emph{explosion} of $\kappa$ occurs.
By definition,
we have that
\begin{equation}
\label{eq:explosion-times}
\rho_{r,T} = \rho_r , \text{ on $\{ T \leq \rho_r \}$, and hence }
\rho_{\infty,T} = \rho_\infty, \text{ on $\{T < \rho_\infty\}$.}
\end{equation}

\subsection{Escape probability}
\label{sec:escape-probability}

The next result gives a supermartingale condition for an escape probability estimate applied in the proof of Theorem~\ref{thm:lln}. The ideas behind  Theorem~\ref{thm:transience} 
in the discrete-time (thus non-explosive) case go a long way back, see~\cite[Lem.~3.5.7]{mpw} and references therein. 

\begin{theorem}
\label{thm:transience}
Suppose that $\kappa = (\kappa_t)_{t \in \RP}$ is a 
$[0,\infty]$-valued
$(\cF_t)$-adapted rcll process.
Suppose that there exist~$x_1 \in \RP$ and a bounded continuous $f : \RP \to (0,\infty)$ such that
\begin{thmenumi}[label=(\alph*)]
\item\label{item:transience-1} $\inf_{y \in [0,x]} f(y) >0$ for all $x \in \RP$, and $\lim_{y \to \infty} f(y) = 0$;
\item\label{item:transience-2} for all $T \in \cT$ and  $r\in(x_1,\infty)$, the process
$(f(\kappa_{( t+T) \wedge S} )\1{T < \rho_\infty})_{t\in\RP}$, where 
$S:=\lambda_{x_1,T} \wedge \rho_{r,T}$,
is an $(\cF_{t+T})$-supermartingale, i.e., for
$0\leq s\leq t<\infty$,
\[ \Exp \bigl[  f(\kappa_{( t+T) \wedge S} )
 \bigmid \cF_{s + T} \bigr] \leq   f(\kappa_{(s + T) \wedge S } ), \text{ on the event } \{ T < \rho_\infty \}.\]
\end{thmenumi}
Then for any $\ell \in \RP$ and any $\eps >0$, there exists $x \in(\ell,\infty)$ such that, for every~$T \in \cT$,
\begin{equation}
\label{eq:escape-probability}
\Pr ( \lambda_{\ell , T} < \rho_{\infty} \mid \cF_T ) \leq \eps,  \text{ on the event } \{ \kappa_T \geq x ,\, T < \rho_\infty \} .\end{equation}
Moreover, if~\eqref{eq:escape-probability} holds and $\Pr ( \rho_r < \rho_\infty ) = 1$ for all $r \in \RP$, 
then $\lim_{t \uparrow \rho_\infty} \kappa_t = \infty$, a.s.
\end{theorem}

We remark that $x$ in~\eqref{eq:escape-probability} is chosen independently of the stopping time $T\in\cT$.

\begin{proof}[Proof of Theorem~\ref{thm:transience}.]
It suffices to prove~\eqref{eq:escape-probability} for $\ell \geq x_1$. Pick any $T \in \cT$ and $r\in(\ell,\infty)$ and note
$S\geq \lambda_{\ell,T} \wedge \rho_{r, T}\geq T$.
Moreover, on the event 
$\{T < \rho_\infty\}$, we have $S\leq \rho_\infty$.
Define the process $\zeta=(\zeta_t)_{t\in\RP}$ by
\[ \zeta_t :=
f( \kappa_{(t + T) \wedge \lambda_{\ell,T} \wedge \rho_{r, T}} )\1{T < \rho_\infty}, \text{ for $t\in\RP$.}\\
\]
By the fact that
$S\geq \lambda_{\ell,T} \wedge \rho_{r, T}$,
and hypothesis~\ref{item:transience-2},
the process
$\zeta$ is an $(\cF_{t+T})$-supermartingale stopped at $(\lambda_{\ell,T} \wedge \rho_{r, T})-T$. Thus, by~\cite[Ch.~II, Thm~3.3]{ry}, $\zeta$ is a non-negative supermartingale.
Hence,   for all $t\in\RP$,
\[ f(\kappa_T)=\zeta_0 \geq \Exp [ \zeta_{t} \mid \cF_T ] \geq \Exp [ \zeta_{t} \1 { \lambda_{\ell, T} < \rho_{r,T} } \mid \cF_T ], \text{ on } \{ T < \rho_\infty \}. \]
Since 
$\zeta_t\geq0$ for $t\in\RP$, for a sequence $t_k\uparrow\infty$, the (conditional) Fatou lemma yields
\begin{align*}
\Exp \Bigl[ \liminf_{k \to \infty} \zeta_{t_k} \1 { \lambda_{\ell, T} < \rho_{r,T} }  \Bigmid \cF_T \Bigr]
& \leq \liminf_{k \to \infty} \Exp [ \zeta_{t_k} \1 { \lambda_{\ell, T} < \rho_{r,T} } \mid \cF_T ] \leq \zeta_0 ,  \text{ on } \{ T < \rho_\infty \}.
\end{align*}
On $\{  \lambda_{\ell, T} < \rho_{r,T}, \, T < \rho_\infty \}$, 
we have
$\liminf_{k \to \infty} \zeta_{t_k} = f ( \kappa_{\lambda_{\ell,T}} ) \geq \inf_{z \in [0,\ell]} f(z)$,
since $f$ is continuous and the paths of $\kappa$ are right-continuous.
Thus, 
\[ \sup_{z \in[x,\infty)} f(z) \geq \zeta_0 \geq \Pr (  \lambda_{\ell, T} < \rho_{r,T} \mid \cF_T ) \inf_{z \in [0,\ell]} f(z) , \text{ on } \{ \kappa_T \geq x, \, T < \rho_\infty \},\]
for any $x\in(\ell,r)$. Put differently,
\[ \Pr (  \lambda_{\ell, T} < \rho_{r,T} \mid \cF_T ) \leq \frac{\sup_{z \geq x} f(z)}{\inf_{z \in [0,\ell]} f(z)}, \text{ on }  \{ \kappa_T \geq x, \, T < \rho_\infty \} .\]
On $\{T < \rho_\infty\}$,  by~\eqref{eq:explosion-times} we have
$\rho_\infty =  \lim_{r \to \infty} \rho_{r, T}$, implying
$$\cup_{r \in \N} \{ \lambda_{\ell,T} < \rho_{r,T}, T < \rho_\infty \}=\{\lambda_{\ell,T} < \rho_{\infty}, T < \rho_\infty\}.$$
By the (conditional) monotone convergence theorem, on the event 
$\{T < \rho_\infty\}$, we have
$
\lim_{r \to \infty} \Pr ( \lambda_{\ell,T} < \rho_{r,T} \mid \cF_T )
=\Pr ( \lambda_{\ell,T} < \rho_{\infty} \mid \cF_T )$.
Hence
\begin{align*} \Pr ( \lambda_{\ell,T} < \rho_{\infty} \mid \cF_T )
 \leq \frac{\sup_{z \geq x} f(z)}{\inf_{z \in [0,\ell]} f(z)}, \text{ on }  \{ \kappa_T \geq x, \, T < \rho_\infty \} ,\end{align*}
 for any $x\in (\ell,\infty)$.
Hypothesis~\ref{item:transience-1} now yields~\eqref{eq:escape-probability}.

Finally, suppose that
$\Pr ( \rho_r < \rho_\infty ) = 1$ for all $r \in \RP$
and that~\eqref{eq:escape-probability} holds. 
Pick arbitrary $\eps >0$ and  $\ell \geq x_1$.
Then apply~\eqref{eq:escape-probability}
with $T=\rho_r$, to get
\[ \Pr ( \lambda_{\ell,\rho_r} \geq \rho_\infty ) = \Exp \left[  \Pr ( \lambda_{\ell,\rho_r} \geq \rho_\infty \mid \cF_{\rho_r} ) \1 { \kappa_{\rho_r} \geq r } \right] \geq 1- \eps,\]
for some $r > \ell$ sufficiently large.
On the event $\{ \lambda_{\ell,\rho_r} \geq \rho_\infty \}$
we have $\kappa_t \geq \ell$ for all $t \in [ \rho_r , \rho_\infty)$, and
thus, since $\Pr ( \kappa_{\rho_r} \geq r ) = 1$, we have
$\Pr  ( \liminf_{t \uparrow \rho_\infty} \kappa_t \geq \ell ) \geq \Pr ( \lambda_{\ell,\rho_r} \geq \rho_\infty ) \geq 1 - \eps$. Since $\eps >0$ was arbitrary, this means that $\liminf_{t \uparrow \rho_\infty} \kappa_t \geq \ell$, a.s. Since $\ell \in[x_1,\infty)$ was arbitrary,
we conclude that $\liminf_{t \uparrow \rho_\infty} \kappa_t = \infty$, a.s.
\end{proof}

\subsection{Explosion and passage times}
\label{sec:explosion-and-passage-times}

In this section we estimate expected passage times
and establish a strong form of explosion given by $\Exp \rho_\infty < \infty$: see Theorem~\ref{thm:explosion} below.
Conditions for non-explosion, defined as $\Pr ( \rho_\infty = \infty ) = 1$, will be given in
Theorem~\ref{thm:non-explosion} below. Almost-sure behaviour in the non-explosive case is described in Theorem~\ref{thm:martingale-lln} below.

The conditions in Theorem~\ref{thm:explosion} are given in terms of a transformed, stopped,
and compensated process, 
which is defined as follows. 
Let 
\begin{equation}
\label{eq:def_f}  
f: \RP\! \to \RP \,\text{be non-decreasing and continuous, with $f(\infty) := \lim_{x \to \infty} f(x) \in [0,\infty]$.}
\end{equation}
Pick a stopping time $T \in \cT$, 
levels $\ell\in\RP$ and $r \in [0,\infty]$, satisfying $\ell<r$, and a positive parameter $\theta \in (0,\infty)$.
For any $t \in \RP$
define
\begin{align}
    \label{eq:v-def} 
v_t & := 
 \left( ( t + T ) \wedge \lambda_{\ell, T} \wedge \rho_{r, T}  - T \right)\1{T<\rho_\infty};\\
   \label{eq:zeta-def} 
 \zeta^{(f,\theta)}_t 
& := 
\left(f (\kappa_{T+v_t} ) - \theta v_t \right)\1{T < \rho_\infty}.
\end{align}
For ease of exposition, we suppress $T, \ell, r$ in  $v=(v_t)_{t\in\RP}$ and $\zeta^{(f,\theta)}=(\zeta^{(f,\theta)}_t)_{t\in\RP}$.

\begin{theorem}
\label{thm:explosion}
Suppose that $\kappa = (\kappa_t)_{t \in\RP}$ is a
$[0,\infty]$-valued
$(\cF_t)$-adapted rcll process with jumps of finite magnitude:
$|\kappa_t-\kappa_{t-}|<\infty$ for all $t\in\RP$~a.s., where $\kappa_{t-}:=\lim_{s\uparrow t}\kappa_s$
for $t>0$ and $\kappa_{0-}:=\kappa_0$. Suppose the following.
\begin{thmenumi}[label=(\alph*)]
\item\label{item:explosion-a}  For all $\ell \in \RP$ and all $\eps >0$, there exists $x > \ell$ such that~\eqref{eq:escape-probability} holds for all $T \in \cT$.
\item\label{item:explosion-b} For all $r \in \RP$, $\Pr ( \rho_r < \infty ) = 1$.
\item\label{item:explosion-c} There exists $x_1 \in \RP$ such that, for every $x\in(x_1,\infty)$,
there exists a constant $B_x \in \RP$ for which, for all $T \in \cT$, it holds that
\begin{equation}
\label{eq:bounded-escape}    
 \Exp [ \rho_{x,T} - T \mid \cF_{T} ] \leq B_x, \text{ on } \{ \kappa_T \leq x_1 , \, T < \rho_\infty \}.
 \end{equation}
\end{thmenumi}
Let $f$, 
$v$ and $\zeta^{(f,\theta)}$ be as in~\eqref{eq:def_f},~\eqref{eq:v-def},  and~\eqref{eq:zeta-def}, respectively.
Suppose also
that there exists
$\theta \in (0,\infty)$ such that for all $\ell \in[x_1,\infty)$, $r \in(\ell,\infty)$ and $T \in \cT$, 
the process $\zeta^{(f,\theta)}$
is either a supermartingale, i.e., for all  $s, t \in \RP$,  $t \geq s$,
we have
\begin{align}
    \label{eq:supermartingale-condition}
    \Exp [ \zeta^{(f,\theta)}_t \mid \cF_{s +T} ] \leq \zeta^{(f,\theta)}_s , \text{ on } \{ T  < \rho_\infty\},
\end{align}
or a submartingale, i.e. for all  $s, t \in \RP$, $t \geq s$, we have
\begin{align}
    \label{eq:submartingale-condition}
     \Exp [ \zeta^{(f,\theta)}_t \mid \cF_{s + T} ] \geq \zeta^{(f,\theta)}_s, \text{ on } \{ T  < \rho_\infty\}.
\end{align}  
Then the following statements hold.
\begin{thmenumi}[label=(\roman*)]
\item
\label{thm:explosion-i}
If $f(\infty) = \infty$ and~\eqref{eq:supermartingale-condition} holds, then for any sequence $r_n \to \infty$,
\begin{equation}
    \label{eq:hitting-time-lower-bound}
 \liminf_{n \to \infty} \frac{\Exp [\rho_{r_n} \mid \cF_0]}{\Exp[f(\kappa_{\rho_{r_n}}) \mid \cF_0]} \geq \frac{1}{\theta},\as,
 \end{equation}
 and thus $\Exp \rho_\infty = \infty$.
\item
\label{thm:explosion-ii}
Assume~\eqref{eq:submartingale-condition}.
If $f(\infty) = \infty$
and for every $r\in\RP$ there exists a constant $C_r \in\RP$ such that $\Exp [f(\kappa_{\rho_r})\mid \cF_0] \leq C_r$, a.s.,
then for any sequence $r_n \to \infty$,
\begin{align}
    \label{eq:hitting-time-upper-bound}
    \limsup_{n \to \infty} \frac{\Exp [\rho_{r_n} \mid \cF_0]}{\Exp[f(\kappa_{\rho_{r_n}}) \mid \cF_0 ]} & \leq \frac{1}{\theta}, \as
    \end{align}
If $f(\infty) < \infty$, then there exists a constant $C \in\RP$ such that $\Exp [\rho_\infty \mid \cF_0] \leq C$, a.s.    
\end{thmenumi}
\end{theorem}

The conditional expectation $\Exp[\rho_r \mid \cF_0]$ in~\eqref{eq:hitting-time-lower-bound} and~\eqref{eq:hitting-time-upper-bound} stresses that these bounds hold for any (possibly random) starting value $\kappa_0$.
In particular, this uniformity in the starting point yields the strong version of explosion of the reflected diffusion in Theorem~\ref{thm:lln}. 
The proof of Theorem~\ref{thm:explosion} will make use of the following lemma. Lemma~\ref{lem:bounded-set-correction}\ref{lem:bounded-set-correction-i} will also be used in Section~\ref{sec:lln} (see the proof of Lemma~\ref{lem:xi-qv}).

\begin{lemma}
\label{lem:bounded-set-correction}
Suppose that $\kappa = (\kappa_t)_{t \in \RP}$ is a $[0,\infty]$-valued
$(\cF_t)$-adapted rcll process with jumps of finite magnitude, and 
that
hypotheses~\ref{item:explosion-a}, \ref{item:explosion-b},
and~\ref{item:explosion-c} of Theorem~\ref{thm:explosion} hold.
 Then for any $f$ in~\eqref{eq:def_f} and $r\in\RP$
such that $\Exp [f(\kappa_{\rho_r}) \mid \cF_0] \leq C_r$ for some constant $C_r \in\RP$,  the following statements are true.
\begin{thmenumi}[label=(\roman*)]
\item
\label{lem:bounded-set-correction-i}
If~\eqref{eq:supermartingale-condition} holds,
there is a constant $C \in\RP$ such that, for all $t\in\RP$,
\[ \Exp\left[ f (\kappa_{t \wedge \rho_r } ) \mid \cF_0\right] \leq C + f(\kappa_0)+ \theta \Exp [ t \wedge \rho_r \mid \cF_0 ], \as \]
\item
\label{lem:bounded-set-correction-ii}
If~\eqref{eq:submartingale-condition} holds,
there is a constant $C \in\RP$ such that, for all $t\in\RP$, 
\[ \theta \Exp [ t \wedge \rho_r \mid \cF_0] 
\leq C + \Exp [f(\kappa_{t\wedge \rho_r}) \mid \cF_0] , \as \]
\end{thmenumi}
\end{lemma}
\begin{proof}
Let $x_1 \in \RP$ be the constant appearing in
hypotheses~\ref{item:explosion-c} of Theorem~\ref{thm:explosion}.
By~\eqref{eq:escape-probability}, there exists $x\in(x_1,\infty)$ so that
 for all $T \in \cT$,
\begin{equation}
\label{eq:return-estimate-for-explosion-1/2}
\Pr ( \lambda_{x_1, T} < \rho_\infty \mid \cF_T ) \leq 1/2, \text{ on } \{ \kappa_T \geq x, \, T < \rho_\infty \} . \end{equation}
Set $t_0 :=0$. Using~\eqref{eq:lambda-rho}, for $k \in \N$, we can define recursively the stopping times 
\[ s_k := \rho_{x,t_{k-1}} 
 \text{ and } 
t_k := \lambda_{x_1, s_k}. 
\]
By hypothesis~\eqref{eq:bounded-escape},  
$s_k<\infty$ on the event $\{ t_{k-1} < \rho_\infty \}$. By~\eqref{eq:explosion-times} we have  the following equality of events: $\{s_k=\rho_\infty, \, t_{k-1} < \rho_\infty\}=\{s_k=\rho_{\infty, \, t_{k-1}}, t_{k-1} < \rho_\infty\}$. However,  on this event, $\kappa$ can neither be continuous nor have a jump at $s_k$, as in both cases  this would imply  $\kappa_{s_k}<\infty$ (recall that $\kappa_{s_k-}\leq x$ and the jumps of $\kappa$ have finite magnitude by assumption) and thus 
$s_k<\rho_{\infty,t_{k-1}}$. 
Hence, for all $k\in\N$, we must have 
$\{ t_{k-1} < \rho_\infty \} = \{ s_k < \rho_\infty \}$ up to events
of probability~$0$. Thus we may apply~\eqref{eq:return-estimate-for-explosion-1/2} at times $T= s_k \in \cT$ to obtain 
\[ \Pr ( t_{k} < \rho_\infty \mid \cF_0) =
\Exp \left[ \Pr ( t_{k} < \rho_\infty \mid \cF_{s_{k}} ) \1 { t_{k-1} < \rho_\infty} \mid \cF_0  \right]
\leq (1/2) \Pr  ( t_{k-1} < \rho_\infty \mid \cF_0),\as\]
for all $k \in \N$.
Iterating this inequality shows that, for any $k \in \ZP$,
\begin{equation}
    \label{eq:return-tail}
    \Pr ( t_k < \rho_\infty \mid \cF_0 ) \leq 2^{-k} \as,\text{ implying $\Pr ( N  < \infty \mid \cF_0 ) = 1\as$,}
\end{equation}
where $N:=\max \{ k \in \ZP : t_k < \rho_\infty \}$.
In particular, the stopping times
 $t_0 < s_1 < t_1 < \cdots  < t_N < \rho_\infty$ 
satisfy $\kappa_{t_k} \leq x_1<x\leq \kappa_{s_k}<\infty$ for all $k<N$
and $t_k=\infty$ for $k>N$.

Pick any $r \in(x,\infty)$, 
define 
$\zeta_t := f (\kappa_{t \wedge \rho_r} ) - \theta (t \wedge \rho_r)$
for all $t\in\RP$
and note that,
by~\eqref{eq:return-tail}, the following sum is finite: 
\begin{align}
\label{eq:zeta-decomposition}
 \nonumber
 \zeta_t - \zeta_0 & = \sum_{k \in \ZP} 
 \left( \zeta_{t \wedge t_{k+1}} - \zeta_{t \wedge t_k} \right) \1 { t_k < \rho_r } \\
  & =  \sum_{k \in \ZP} 
 \left( \zeta_{t \wedge s_{k+1}} - \zeta_{t \wedge t_k} \right)\1 { t_k < \rho_r}
 + \sum_{k \in \ZP} \left( \zeta_{t \wedge t_{k+1}} - \zeta_{t \wedge s_{k+1}} \right) \1 { s_{k+1} < \rho_r} \nonumber\\
 & =  \sum_{k \in \ZP} 
 \left( \zeta_{t \wedge s_{k+1}} - \zeta_{t \wedge t_k} \right)\1 { t_k < \rho_r}
 + \sum_{k \in \N} \left( \zeta_{t \wedge t_k} - \zeta_{t \wedge s_k} \right) \1 { s_k < \rho_r}.
\end{align}

We now establish~\ref{lem:bounded-set-correction-ii}. 
By definitions~\eqref{eq:v-def} and~\eqref{eq:zeta-def} with
$T:=s_k$, $\ell:=x_1$ and $r$ chosen above,
for all $k\in\ZP$ and $t\in\RP$
the following holds
\begin{align}
\nonumber
\left( \zeta_{t \wedge t_k} - \zeta_{t \wedge s_k} \right) \1 { s_k < \rho_r}
 & =
\left( \zeta_{t \wedge t_k} - \zeta_{ s_k} \right) \1 { s_k < \rho_r\wedge t }\\
& =
\left(\zeta^{(f,\theta)}_{t- s_k} - \zeta^{(f,\theta)}_0
\right)\1 { s_k < \rho_r\wedge t},
\label{eq:zeta_sub_super_mart}
\end{align}
since on the event $\{ s_k <  \rho_r \}$ definition~\eqref{eq:lambda-rho} implies $\rho_{r,s_k}=\rho_r$.
Recall that for any two stopping times $T,S\in\cT$, the non-negative variable 
$(S-T)\1{T<S}$ is a stopping time in the filtration $(\cF_{T+t})_{t\in\RP}$.
Thus, as the stopping time $(t-s_k)\1{s_k<t}$ is bounded by $t$ and $\sup_{u\in\RP} |\zeta^{(f,\theta)}_u|\leq f(\kappa_{\rho_r})+\theta\rho_r$ has finite first moment by~\eqref{eq:bounded-escape},
the submartingale property in~\eqref{eq:submartingale-condition} of $\zeta^{(f,\theta)}$
and the optional sampling theorem applied at $(t-s_k)\1{s_k<t}$ yield
$\Exp [ (\zeta^{(f,\theta)}_{t- s_k} - \zeta^{(f,\theta)}_0 )\1 { s_k < \rho_r\wedge t }\mid \cF_0]\geq0$
and hence
\begin{equation}
    \label{eq:down-excursions}
 \Exp \left[ \sum_{k \in \N} \left( \zeta_{t \wedge t_k} - \zeta_{t \wedge s_k} \right) \1 { s_k < \rho_r} \biggmid \cF_0\right] \geq 0, \as
 \end{equation}
Recall that on $\{ t_k < \rho_r\}$ one has $\zeta_{t_k} = f(\kappa_{t_k})-\theta t_k$
and  $\kappa_{t_k} \leq x_1$ for $k \in \N$.
Also, $f$ is non-negative and non-decreasing,
and hence
$\zeta_{t \wedge s_{k+1}} - \zeta_{t \wedge t_k} \geq - f(x_1) - \theta (s_{k+1} - t_k)$
on $\{ t_k < \rho_r\}$, for every $k \in \N$.
In case $k =0$, we note that $t_0 = 0$, and
$\zeta_{t \wedge s_1} - \zeta_{t \wedge t_0}
= \zeta_{t \wedge s_1} - \zeta_0$ which is $0$ unless $\kappa_0 < x$,
hence $\zeta_{t \wedge s_1} - \zeta_{t \wedge t_0} \geq - f(x) - \theta s_1$.
It follows that 
\begin{align*}
  \sum_{k \in \ZP} \left( \zeta_{t \wedge s_{k+1}} - \zeta_{t \wedge t_k} \right) 
  \1 { t_k < \rho_r}
  & \geq - \sum_{k \in \ZP}
\left[ f( x ) + \theta ( s_{k+1} - t_k ) \right]
\1 { t_k < \rho_r }.
\end{align*}
Taking conditional expectations and applying~\eqref{eq:bounded-escape} with $T=t_k$,  we obtain
\begin{align}
\label{eq:up-excursions}
& {} \Exp \left[ \left(f( x  ) + \theta ( s_{k+1} - t_k )\right)
\1 { t_k < \rho_r } \mid \cF_0  \right] \nonumber\\
&{}\quad{} = \Exp \left[ \left( f( x ) + \theta \Exp [    s_{k+1} - t_k \mid \cF_{t_k} ] \right)
\1 { t_k < \rho_r} \mid \cF_0 \right] \nonumber\\
&{}\quad{} \leq \left( f (x  ) + \theta B_x \right) \Pr ( t_k < \rho_\infty \mid \cF_0 ),\as, \end{align}
for all $k \in \ZP$. 
Combining~\eqref{eq:zeta-decomposition}, \eqref{eq:down-excursions}
and~\eqref{eq:up-excursions}, we obtain
\[ \Exp[\zeta_t \mid \cF_0]\geq \Exp [ \zeta_t - \zeta_0 \mid \cF_0 ] \geq - ( f(x  ) + \theta  B_x ) \sum_{k \in \ZP} \Pr ( t_k < \rho_\infty \mid \cF_0) 
\geq - C,\as, \]
where, by~\eqref{eq:return-tail}, $C \in \RP$
is a constant that depends only on $\theta$, $x$, and $f(x)$.
Thus, by the definition of $\zeta_t$, 
the inequality in~\ref{lem:bounded-set-correction-ii} follows.

The proof of~\ref{lem:bounded-set-correction-i} is similar.
By~\eqref{eq:zeta_sub_super_mart} and the assumed supermartingale property in~\eqref{eq:supermartingale-condition},
the inequality in~\eqref{eq:down-excursions} is reversed.
Then from~\eqref{eq:zeta-decomposition}, we obtain
\[ \Exp [  \zeta_t - \zeta_0 \mid \cF_0 ] \leq 
\Exp \left[\sum_{k \in \ZP} 
 \left( \zeta_{t \wedge s_{k+1}} - \zeta_{t \wedge t_k} \right)\1 { t_k < \rho_r} \biggmid \cF_0\right]
 \leq f (x) \sum_{k \in \ZP} \Pr ( t_k < \rho_\infty \mid \cF_0 ) .\]
It follows that $\Exp[ f (\kappa_{t \wedge \rho_r }) \mid \cF_0]\leq C + f ( \kappa_0 ) + \theta \Exp [ t \wedge \rho_r \mid \cF_0 ]$,
which yields~\ref{lem:bounded-set-correction-i}.
\end{proof}

\begin{proof}[Proof of Theorem~\ref{thm:explosion}]
Since $\rho_r\leq \rho_{r'} \as$ for any $0\leq r\leq r'<\infty$, in part~\ref{thm:explosion-i} we may assume $\Exp \rho_r < \infty$ for all $r \in \RP$. 

Under the conditions of part~(i) of the theorem, 
we have from Lemma~\ref{lem:bounded-set-correction}\ref{lem:bounded-set-correction-i},
continuity of $f$,
and (the conditional) Fatou's lemma
\begin{align*}
    \Exp [ f ( \kappa_{\rho_r} ) \mid \cF_0 ]
& \leq \liminf_{t \to \infty}  \Exp [ f ( \kappa_{t \wedge \rho_r} ) \mid \cF_0 ] \\
& \leq C + f (\kappa_0) + \theta \liminf_{t \to \infty} \Exp [ t \wedge \rho_r \mid \cF_0 ] \\
& \leq C + f(\kappa_0) + \theta \Exp [ \rho_r \mid \cF_0],
\end{align*} 
by monotone convergence. Here, $\Exp [ f ({\kappa_{\rho_{r}}}) \mid \cF_0 ] \geq f(r) \uparrow f(\infty)$.
This proves~\ref{thm:explosion-i}.

For part~\ref{thm:explosion-ii}, 
If $f(\infty) < \infty$,
then 
by Lemma~\ref{lem:bounded-set-correction}\ref{lem:bounded-set-correction-ii} we have
$\theta \Exp [ t\wedge\rho_r \mid \cF_0 ] \leq C+f(\infty)<\infty$. Monotone convergence
gives $\Exp [\rho_\infty\mid\cF_0] = \lim_{r \to \infty} \Exp [\rho_r\mid\cF_0] \leq (C+f(\infty))/\theta <\infty$. 
If $f(\infty) = \infty$, as $t\to\infty$, we have 
$\Exp[t\wedge\rho_r\mid\cF_0]\to \Exp[\rho_r\mid\cF_0]$ (by conditional monotone convergence) and 
$\Exp[f(\kappa_{t\wedge\rho_r})\mid\cF_0]\to\Exp[f(\kappa_{\rho_r})\mid\cF_0]$ (by conditional dominated convergence since $f(\kappa_{t\wedge\rho_r}) \leq f(\kappa_{\rho_r})$ for all $t\in\RP$, $f$ is non-decreasing by~\eqref{eq:def_f} and $\Exp[f(\kappa_{\rho_r})\mid\cF_0]$ is bounded by a constant $C_r$ by assumption).
By taking the limit as $t\to\infty$ the display of  Lemma~\ref{lem:bounded-set-correction}\ref{lem:bounded-set-correction-ii}, we obtain
$$
\frac{\Exp[\rho_r \mid \cF_0]}{\Exp[f(\kappa_{\rho_r}) \mid \cF_0]} \leq \frac{1}{\theta} +
\frac{C}{\theta \Exp[f(\kappa_{\rho_r}) \mid \cF_0]} \leq \frac{1}{\theta} +
\frac{C}{\theta f(r)},
$$
since $\rho_r<\infty$, $r\leq \kappa_{\rho_r}$, a.s., and monotonicity of~$f$.
Taking $r = r_n \to \infty$
yields~\eqref{eq:hitting-time-upper-bound}.   
\end{proof}

\subsection{Non-explosion and transience}
\label{sec:non-explosion}

Our non-explosion result is Theorem~\ref{thm:non-explosion} below.
For any right-continuous semimartingale $X$, we denote by $[X]:=([X]_t)_{t \in \RP}$ the corresponding quadratic variation process. Let $f$ be as in~\eqref{eq:def_f}
and assume that,  for any $0 < \ell < r < \infty$, the process
$\zeta^{(f,\theta)}$, defined by~\eqref{eq:zeta-def},
is a supermartingale.
Note that for any $r'>r>\ell$ these supermartingales coincide on the stochastic interval
$[0,\rho_{r, T})$ and
the quadratic variation 
$[\zeta^{(f,\theta)}]_{\lambda_{\ell, T} \wedge \rho_{r, T}}$
is non-decreasing almost surely as $r\to\infty$. 
Thus the limit 
$\lim_{r\to\infty}[\zeta^{(f,\theta)} ]_{S\wedge \lambda_{\ell, T} \wedge \rho_{r, T}}$
exists in $[0,\infty]$ for any $(\cF_t)$-stopping time $S$ in $[0,\infty]$.

\begin{theorem}
\label{thm:non-explosion}
Suppose that $\kappa = (\kappa_t)_{t \in\RP}$ is a $[0,\infty]$-valued
$(\cF_t)$-adapted rcll process with jumps of finite magnitude. 
Assume $f$ satisfies~\eqref{eq:def_f}
and,
for every $r \in \RP$, we have $\Pr ( \rho_r < \infty ) = 1$.
Suppose also that the following hold.
\begin{thmenumi}[label=(\alph*)]
\item\label{thm:non-explosion-a}
For all $\ell \in \RP$ and all $\eps >0$, there exists $x > \ell$ such that~\eqref{eq:escape-probability} holds for all~$T \in \cT$.
\item\label{thm:non-explosion-b} 
There exist constants $\ell \in \RP$ and 
$\theta \in (0,\infty)$ such that for all $\ell < x < r < \infty$ and $T  := \rho_x$, 
the process $\zeta^{(f,\theta)}$, defined at~\eqref{eq:zeta-def}
is a supermartingale, i.e., for all  $s, t \in \RP$,  $t \geq s$,
\eqref{eq:supermartingale-condition} holds.
\item\label{thm:non-explosion-c} On $\{ \rho_\infty < \infty\}$, it holds that 
$\lim_{r\to\infty}[\zeta^{(f,\theta)} ]_{(\lambda_{\ell, \rho_x} \wedge \rho_{r})-\rho_x} < \infty$,
where
$\ell$ is as in~\ref{thm:non-explosion-b} (recall $\rho_r=\rho_{r,\rho_x}$~a.s. by~\eqref{eq:explosion-times} since $x<r$).
\item\label{thm:non-explosion-d} Assume
$\Exp[(f(\kappa_{\rho_x+T_s})-f(\kappa_{(\rho_x+T_s)-}))^2\1{T_s<\rho_\infty-\rho_x}]<\infty$, where
we define
\begin{equation}
\label{eq:QV_crossing}
    T_s := \inf \bigl\{ t \in \RP : \lim_{r\to\infty}[\zeta^{(f,\theta)} ]_{t\wedge ((\lambda_{\ell, \rho_x} \wedge \rho_{r})-\rho_x)} \geq s \bigr\}, \text{ for any } s>0.
\end{equation}
\end{thmenumi}
Then $ \rho_\infty = \infty$ almost surely.
\end{theorem}
\begin{proof}
Fix $\ell$ as in~\ref{thm:non-explosion-b} and $\eps >0$. Then
by hypothesis~\ref{thm:non-explosion-a} and~\eqref{eq:escape-probability} applied at $T=\rho_x$, we may take $x > \ell$ 
for which $\Pr ( \lambda_{\ell,\rho_x} < \rho_\infty \mid \cF_{\rho_x} ) \leq \eps$.

Take $r\in( x,\infty)$. Since 
$\Pr (\rho_r < \infty ) =1$ and $\kappa$ has jumps of finite magnitude, 
$\Pr ( \rho_r < \rho_\infty ) =1$. 
Hypothesis~\ref{thm:non-explosion-b} 
shows that $(\zeta_{t}^{(f,\theta)})_{t \in\RP}$ is a rcll  supermartingale
 adapted to $(\cF_{\rho_x+t})_{t \in \RP}$.
 Since it is also a 
 semimartingale,
it has (a.s.~unique) canonical decomposition $\zeta^{(f,\theta)}_t = \zeta^{(f,\theta)}_0 + M^r_t + A^r_t$, 
which agrees with its Doob--Meyer decomposition~\cite[p.~116--117]{protter};
here $\zeta^{(f,\theta)}_0 = f (\kappa_{\rho_x} )$,
$(M^r_t)_{t\in\RP}$
is a local martingale, $M^r_0 = A^r_0 = 0$, and $(A^r_t)_{t \in\RP}$ 
is non-increasing. Note that,
 by~\eqref{eq:explosion-times}, $\rho_r = \rho_{r, \rho_x}  \geq \rho_x$ whenever $r \geq x$.
By uniqueness of the decomposition, for $u \geq r$ we have
$M^{u}_{t \wedge ( \rho_{r} - \rho_x)} = M^r_t$ 
and $A^u_{t \wedge ( \rho_{r} - \rho_x)}=A^r_t$ 
for all $t \in \RP$.
Thus if we define $M_t := \limsup_{r \to \infty} M^r_t$,
$A_t:=\limsup_{r \to \infty} A^r_t$,
we have that $M_t = M^{r}_t$ and $A_t=A^r_t$ on $\{ t \leq \rho_{r} - \rho_x \}$,
and $M_t = \lim_{r \to \infty} M^r_t$, 
$A_t = \lim_{r \to \infty} A^r_t$
on $\{ t < \rho_\infty -\rho_x\}$.
Define $\zeta=(\zeta_t)_{t\in\RP}$
by
$\zeta_t := f(\kappa_{\rho_x})+ M_t + A_t$.
Then, since $A_t\leq0$ for all $t\in\RP$, we have 
\begin{equation}
    \label{eq:non-explosion-martingale}
    \zeta_{\rho_{r}-\rho_x} \leq 
    f ( \kappa_{\rho_x}) + M_{\rho_{r}-\rho_x}, \text{ for any $r > \ell$.}
\end{equation}

For $n \in \N$, let 
$T_n$
be as in~\eqref{eq:QV_crossing} with $s=n$.
Then, for all $u>\ell$, we have 
\begin{equation}
\label{eq:bound_QV}
[ \zeta ]_{t \wedge (\rho_u-\rho_x) \wedge T_n} \leq n+(f(\kappa_{\rho_x+T_n})-f(\kappa_{(\rho_x+T_n)-}))^2. \end{equation}
Note that by~\cite[Cor.~II.6.3]{protter}
we have 
$\Exp ( [\zeta]_{t \wedge (\rho_u-\rho_x) \wedge T_n} )=\Exp\left[ (M^u_{t \wedge T_n})^2\right]$.
Thus, by~\eqref{eq:bound_QV} and the assumption in~\ref{thm:non-explosion-d}, for any fixed $n \in \N$, we have $$\sup_{u>\ell}\sup_{t\in\RP}\Exp\left[ (M^u_{t \wedge T_n})^2\right]<\infty.$$
Let $\N_x := \N \cap (x, \infty)$.
Since $M_{(\rho_m-\rho_x) \wedge T_n}=M^m_{(\rho_m-\rho_x) \wedge T_n}$
and
$\sup_{m \in \N_x} \Exp [ M_{(\rho_m-\rho_x) \wedge T_n}^2 ] 
= \sup_{m \in \N_x} \Exp  [ ( M^m_{(\rho_m-\rho_x) \wedge T_n} )^2 ] < \infty$, for any fixed $n\in\N$, the discrete-time process $(M_{(\rho_m-\rho_x) \wedge T_n})_{m \in \N_x}$
is an $L^2$-bounded martingale.
Hence, for each $n \in \N$,
the limit $\lim_{m \to \infty} M_{(\rho_m-\rho_x) \wedge T_n} =: Q_n$, exists and is finite, a.s.

Hypothesis~\ref{thm:non-explosion-c} can be expressed in terms of $\zeta$, defined above, as follows:
$[ \zeta ]_{\rho_\infty-\rho_x} < \infty$ 
on $\{ \rho_\infty < \infty \}$. 
Thus,  on $\{ \rho_\infty < \infty\}$, 
there exists a random $n_0\in\N$ such that 
$T_{n_0} = \infty$.
 Hence, on $\{ \rho_\infty < \infty\}$, we have 
 $Q_{n_0} = \lim_{m \to \infty} M_{(\rho_m-\rho_x) \wedge T_{n_0}}
 = \lim_{m \to \infty} M_{(\rho_m-\rho_x)}$ and $Q_{n_0}<\infty$ a.s. 
We conclude that
\[
\limsup_{m \to \infty}
f ( \kappa_{\lambda_{\ell,\rho_x} \wedge \rho_{m}} ) 
\leq \limsup_{m \to \infty} \zeta_{\rho_m-\rho_x} 
+ \theta \rho_\infty < \infty , \text{ on } \{ \rho_\infty < \infty \},\]
where the second inequality follows from~\eqref{eq:non-explosion-martingale}.
Thus we find
\begin{equation}
\label{eq:bound_non-explosion}
\sup_{m > \ell} \left( f (\kappa_{\rho_m}) \1{ \rho_m < \lambda_{\ell,\rho_x} } \right) < \infty , \text{ on } \{ \rho_\infty < \infty \}.
\end{equation}
Since $f(\kappa_{\rho_m})\geq f(m)$,
we have: 
$f(\kappa_{\rho_m}) \1{ \rho_m < \lambda_{\ell,\rho_x} } \to \infty$ if and only if $\{ \rho_\infty \leq \lambda_{\ell,\rho_x} \}$.
Hence~\eqref{eq:bound_non-explosion} implies 
$\{ \rho_\infty < \infty \}\subseteq \{ \rho_\infty > \lambda_{\ell,\rho_x} \}$,
yielding
$\Pr ( \rho_\infty < \infty \mid \cF_{\rho_x} )
\leq \Pr ( \lambda_{\ell,\rho_x} < \rho_\infty \mid \cF_{\rho_x} ) \leq \eps$, by the choice of~$x > \ell$ (see assumption in~\ref{thm:non-explosion-a}). Thus $\Pr ( \rho_\infty < \infty) \leq \eps$
and, since $\eps>0$ was arbitrary, the result follows.
\end{proof}

Theorem~\ref{thm:martingale-lln}
furnishes 
upper and lower bounds on the almost-sure growth rate of a non-explosive process on~$\R$.
The conditions~\ref{thm:martingale-lln-upper} and~\ref{thm:martingale-lln-lower}
 bound  the rate
at which the process accumulates drift outside a bounded set, which  plays a role familiar from Foster--Lyapunov conditions in discrete time.

\begin{theorem}
\label{thm:martingale-lln}
Suppose that $\kappa = (\kappa_t)_{t \in\RP}$ is an~$\RP$-valued (i.e.~$\Pr ( \rho_\infty = \infty) = 1$)
$(\cF_t)$-adapted rcll process.
Suppose the following.
\begin{thmenumi}[label=(\alph*)]
\item\label{thm:martingale-lln-a}
For every $r \in \RP$, $\Pr ( \rho_r < \infty ) = 1$. 
For all $\ell \in \RP$ and all $\eps >0$, there exists $x > \ell$ such that~\eqref{eq:escape-probability} holds for all $T \in \cT$.
\item\label{thm:martingale-lln-b} 
Let $f(\kappa)$ be a semimartingale for some $f$ satisfying~\eqref{eq:def_f}.
Assume that for some $\eta \in (0,2)$, $\lim_{t \to \infty} t^{-\eta} \Exp \bigl( [ f(\kappa) ]_{t} \bigr)  = 0$.
\end{thmenumi}
Then the following statements hold.
\begin{thmenumi}[label=(\roman*)]
\item
\label{thm:martingale-lln-upper}
If there exist constants $\ell \in \RP$ and 
$\theta \in (0,\infty)$ such that for all $x\in(\ell,\infty)$ and $T  := \rho_x$, 
the process $\zeta^{(f,\theta)}$, defined at~\eqref{eq:zeta-def} (with 
$r :=\infty$)
is a supermartingale (i.e., for all  $s, t \in \RP$,  $t \geq s$,
\eqref{eq:supermartingale-condition} holds),
then 
$\limsup_{t \to \infty} (f(\kappa_t)/t) \leq \theta$, a.s.
\item
\label{thm:martingale-lln-lower}
If there exist constants $\ell \in \RP$ and 
$\theta \in (0,\infty)$ such that for all $x\in(\ell,\infty)$ and $T  := \rho_x$, 
the process $\zeta^{(f,\theta)}$, defined at~\eqref{eq:zeta-def} (with 
$r :=\infty$)
is a submartingale (i.e., for all  $s, t \in \RP$,  $t \geq s$,
\eqref{eq:submartingale-condition} holds),
then $\liminf_{t \to \infty} (f(\kappa_t)/t) \geq \theta$, a.s.
\end{thmenumi}
\end{theorem}

\begin{remarks}
\phantomsection
\begin{myenumi}
\setlength{\itemsep}{0pt plus 1pt}
\item[{\rm (a)}] One cannot deduce~\ref{thm:martingale-lln-lower} directly from~\ref{thm:martingale-lln-upper} in Theorem~\ref{thm:martingale-lln},
as the conditions in the theorem are not symmetric under a sign change. The symmetric part of the proof
is extracted as Lemma~\ref{lem:lln-lemma} below.
\item[{\rm (b)}]
Theorem~\ref{thm:martingale-lln} is inspired by the discrete-time Theorem~3.12.2 of~\cite{mpw} (see also~\cite{mwa}).
That result was stated for a single process satisfying a two-sided drift condition.
 The separation of the upper and lower bounds in Theorem~\ref{thm:martingale-lln} is an improvement essential to our present application:
we obtain the upper and lower bounds in Theorem~\ref{thm:lln} via two
(slightly) different Lyapunov functions, each satisfying only a one-sided drift condition, but sharing a similar quadratic variation estimate required for Theorem~\ref{thm:lln}\ref{thm:martingale-lln-b}.
\end{myenumi}
\end{remarks}

The next lemma is a key ingredient in the proof of Theorem~\ref{thm:martingale-lln}.
As well as Theorem~3.12.2 of~\cite{mpw}, neighbouring results in discrete time include~\cite[Cor.~4]{sheu-yao}.

\begin{lemma}
\label{lem:lln-lemma}
Suppose that $\varphi = (\varphi_t)_{t \in\RP}$ is a $\R$-valued rcll semimartingale.
If there exist  $\alpha \in \R$ and a stopping time  $\tau \in[0,\infty]$, such that 
$( \varphi_{t\wedge \tau} -  \alpha ( t\wedge \tau ) )_{t \in\RP}$ is a submartingale
and the quadratic variation $[\varphi]$ satisfies  $\lim_{t \to \infty}  t^{-\eta} \Exp \bigl( [ \varphi]_{t\wedge\tau} \bigr)  = 0$ for some $\eta \in (0,2)$,
then $\liminf_{t \to \infty} (\varphi_t/t) \geq \alpha$, on the event $\{ \tau = \infty \}$.
\end{lemma}
\begin{proof}
By~\cite[Thm.~III.2.16]{protter}, there exist a local martingale $M=(M_t)_{t \geq 0}$ and a non-decreasing process 
$A=(A_t)_{t\geq 0}$, with $A_0=M_0=0$,
satisfying 
$\varphi_{t\wedge\tau}-\alpha (t\wedge \tau)= \varphi_0+M_t+A_t$
for all $t\in\RP$.
Since $A$ has paths of finite variation, we have
$[\varphi]_{t\wedge\tau}=[M]_t$ for all $t\in\RP$ a.s.
Thus
$0\leq \Exp [ M_t^2 ] = \Exp  [ \varphi ]_{t\wedge\tau} 
= o ( t^{\eta} )$ as $t \to \infty$, by assumption.

By Doob's maximal quadratic inequality~\cite[Thm.~I.2.20]{protter}, applied to the martingale $M$,
for all $\delta \in (0,\frac{2-\eta}{2})$ we have:
\[ \Pr \left[ \sup_{s \in [0,t]} M_s^2 \geq 
t^{2-\delta} \right] \leq t^{\delta-2} \Exp \left[ \sup_{s \in [0,t]} M_s^2\right]  \leq 
4 t^{\delta-2}  \Exp [ M_t^2 ]  
= O ( t^{-\delta} ) , \text{ as } t \to \infty .\]
Applied to the sequence of times $t= t_k = 2^k$, the Borel--Cantelli lemma shows that, a.s.,
for all but finitely many $k$,
$\sup_{s \in [0,t_k]} M_s^2 \leq t_k^{2-\delta}$.
Every $t \geq 1$ has $t_k \leq t < t_{k+1}$ for some $k = k(t) \in \ZP$, 
with $\lim_{t \to \infty} k(t) = \infty$, 
so that, a.s., for all $t$ sufficiently large,
\[ \sup_{s \in [0,t]} M_s^2 \leq \sup_{s \in [0,t_{k+1}]} M_s^2 \leq t_{k+1}^{2-\delta}
\leq 4 t^{2-\delta} ,\]
since $t_{k+1} = 2 t_k \leq 2 t$. Thus we conclude that
\begin{equation}
\label{eq:martingale-limit}
 \lim_{t \to \infty} t^{-1} M_t =0 , \as 
\end{equation}
Since $A_t\geq 0$ for all $t\in\RP$, the limit in~\eqref{eq:martingale-limit} implies
\[ \liminf_{t \to \infty} \frac{\varphi_t}{t}
= \liminf_{t \to \infty} \frac{\varphi_{t\wedge\tau}}{t}
 \geq   \liminf_{t \to \infty} \frac{M_t+\alpha (t\wedge\tau)}{t} = \alpha, \text{ on } \{ \tau = \infty \} ,\]
which completes the proof of the lemma.
\end{proof}

Now we can complete the proof of Theorem~\ref{thm:martingale-lln}.

\begin{proof}[Proof of Theorem~\ref{thm:martingale-lln}.]
First we prove part~\ref{thm:martingale-lln-lower}. 
Let $\eps >0$. 
Take $\ell$ as specified in the hypotheses of part~\ref{thm:martingale-lln-lower} of the theorem.
Then condition~(a) says that we may choose $x > \ell$ such that~\eqref{eq:escape-probability} holds; fix this $x$. Note that $\rho_x < \infty$, a.s., by~(a). 
Write $\varphi_t := f(\kappa_{\rho_x+t})$ for $t \in \RP$. Set
$\tau = \lambda_{\ell,\rho_x} - \rho_x$.
By the hypothesis of part~\ref{thm:martingale-lln-lower},
the submartingale $\zeta^{(f,\theta)}$, defined in~\eqref{eq:zeta-def} (with 
$r=\infty$), satisfies  
\[
\zeta^{(f,\theta)}_t = \varphi_{t \wedge \tau} - \theta  (t \wedge \tau),   \text{ for all } t\in\RP.   \]
Thus we may apply Lemma~\ref{lem:lln-lemma} with $\alpha = \theta$ to conclude that
$\liminf_{t \to \infty} ( \varphi_t / t) \geq \theta$ on $\{ \tau = \infty \}$.
In other words, since $f(\kappa_t) \geq 0$,
\begin{align*} 
\liminf_{t \to \infty} \frac{f(\kappa_t)}{t} \1 { \lambda_{\ell,\rho_x} = \infty } & =
\liminf_{t \to \infty} \frac{f(\kappa_{\rho_x+t})}{\rho_x +t} \1 { \rho_x < \infty, \,  \lambda_{\ell,\rho_x} = \infty} \\
& 
= \liminf_{t \to \infty} ( \varphi_t / t) \1 { \tau = \infty }
\geq \theta \1 {  \lambda_{\ell,\rho_x} = \infty} , \as 
\end{align*}
since $\Pr(\rho_x < \infty)=1$. Hence, by~\eqref{eq:escape-probability} applied at $T = \rho_x$, we obtain
\begin{align*}
     \Pr \Bigl( \liminf_{t \to \infty} \frac{f(\kappa_t)}{t} \geq \theta \Bigr)
     & = \Exp \Bigl[ 
     \Pr \Bigl( \liminf_{t \to \infty} \frac{f(\kappa_t)}{t} \geq \theta \Bigmid \cF_{\rho_x} \Bigr)
     \1 { \rho_x < \infty } \Bigr] \\
     & \geq \Pr \Bigl[ \Pr \bigl( \lambda_{\ell, \rho_x} = \infty 
    \bigmid \cF_{\rho_x} \bigr)
     \1 { \rho_x < \infty } \Bigr] \\
     & \geq (1-\eps) \Pr ( \rho_x < \infty ) = 1-\eps.
\end{align*}
Since $\eps>0$ was arbitrary, we obtain part~\ref{thm:martingale-lln-lower}. 
The argument for part~\ref{thm:martingale-lln-upper} is similar:
define $\zeta^{(f,\theta)}$ by~\eqref{eq:zeta-def} and set $\varphi_t := -f(\kappa_{\rho_x+t})$ and $\alpha :=  -\theta$.  Apply Lemma~\ref{lem:lln-lemma} to the $\R$-valued submartingale $-\zeta^{(f,\theta)}$ to conclude the proof.
\end{proof}

\section{Exit from a bounded set by the reflecting diffusion}
\label{sec:exit}

In this section we provide some estimates for a process
$Z = (X, Y)$ satisfying~\eqref{eq:SDE-for-Z}
up to explosion time $\taue$ (see Appendix~\ref{sec:construction} for a rigorous definition of such processes).
Denote by $(\cF_t)_{t\in\RP}$ the filtration of the driving Brownian motion in~\eqref{eq:SDE-for-Z} and recall that $(\cF_t)$-stopping times 
$\sigma_r = \inf \{ t \in \RP : X_t \geq r \}$
and $\taue$
satisfy
$\taue = \lim_{r \to \infty} \sigma_r$.
Recall that the domain $\cD$ is defined by~\eqref{eq:domain-def}.
The main aim of this section is 
prove Theorem~\ref{thm:exit}. In particular, we 
show that~$\sup_{z\in\cD}\Exp_z \sigma_r < \infty$, which is an important ingredient in the proof of (fixed-$z$ large-$r$) asymptotics given in Theorem~\ref{thm:lln} above.

\begin{theorem}
\label{thm:exit}
Suppose that~\eqref{ass:domain1}, \eqref{ass:domain2}, \eqref{ass:variance}, and~\eqref{ass:vector-field} hold. 
Then for every $r \in \RP$, 
$\sup_{z\in\cD} \Exp_z \sigma_r < \infty$.
Moreover, $\Pr_z ( \limsup_{t \uparrow \taue} X_t = +\infty ) = 1$ for all $z \in \cD$.
\end{theorem}

\begin{remarks}
\label{rems:exit}
\phantomsection
\begin{myenumi}[label=(\alph*)]
\setlength{\itemsep}{0pt plus 1pt}
\item
Theorem~\ref{thm:existence} gives existence and uniqueness of the process $Z$ under the conditions of Theorem~\ref{thm:exit}, and, moreover, implies that~$\lim_{t \uparrow \taue} X_t = \infty$ on the event $\{\taue < \infty\}$. 
However, on the event $\{ \taue = \infty \}$ it may be the case that
$\liminf_{t \uparrow \taue} X_t < \infty$, if the process is recurrent. The assumptions of 
Theorem~\ref{thm:exit} permit 
recurrence but guarantee that the process $Z$ is almost surely not confined, since 
$\Pr_z ( \limsup_{t \uparrow \taue} X_t = +\infty ) = 1$.
\item
In the proof of Theorem~\ref{thm:exit}, 
it suffices
to work with the process on a compact set. Thus, for existence and uniqueness of solutions of~\eqref{eq:SDE-for-Z} up to time~$\sigma_r$, the bounded-domain results of~\cite{ls} 
are in fact sufficient:
see Section~\ref{sec:construction} for existence and uniqueness theory for SDE~\eqref{eq:SDE-for-Z} on non-compact domains.
\end{myenumi}
\end{remarks}

The proof of Theorem~\ref{thm:exit}
has three main ingredients.
First, we show that starting very close to the boundary the process
moves a positive distance into the interior in a short time (Lemma~\ref{lem:exit1} below),
and once away from the boundary, it has positive
probability of reaching horizontal distance~$r$ 
before getting too close to the boundary again (Lemma~\ref{lem:exit2}). Moreover,
 the process cannot spend a long time in a bounded subset of the interior (Lemma~\ref{lem:exit3}).
We start with a preliminary smoothness result.

\begin{lemma}
\label{lem:C2-domain}
If Assumption~\eqref{ass:domain1} holds, then~$\cD$ is a $C^2$ domain.
\end{lemma}

The lemma follows from assumption~\eqref{ass:domain1},
since 
  $\lim_{x \to 0} b(x) =0$, $\lim_{x\to0}b'(x)=\infty$ and $b' (x) > 0$ for all $x>0$ sufficiently small (see Remark~\ref{rem:b-small-x}) imply that 
the inverse $b^{-1}$ 
is twice
continuously differentiable in a neighbourhood of the origin and satisfies 
\[ \frac{\ud}{\ud s} b^{-1} (s) = \frac{1}{b' ( b^{-1} (s) )}, ~~~
\frac{\ud^2}{\ud s^2} b^{-1} (s) = - \frac{b'' ( b^{-1} (s))}{(b' ( b^{-1} (s) ))^3} \quad\&\quad \lim_{s\to0} b^{-1}(s)=0.\]
Hence $\lim_{s \to 0} (\ud/\ud s) b^{-1} (s) = 0$
and $\lim_{s \to 0} (\ud^2/\ud s^2) b^{-1} (s)$ exists in $\RP$,
making $b^{-1}$ twice continuously differentiable at $0$, thus implying Lemma~\ref{lem:C2-domain} ($\partial \cD\setminus\{0\}$ is clearly $C^2$). 

The shortest squared distance of $z \in \cD$ from $\partial\cD$, defined as
\begin{equation}
    \label{eq:D-def}
    D (z) := \inf_{z' \in \partial \cD}  \| z - z'\|_{d+1}^2,
\end{equation}
is a $C^2$ function on $\cD$ under the assumption of Lemma~\ref{lem:C2-domain}.
Denote its Hessian (i.e. a non-negative definite matrix of the second partial derivatives of $D$) by $H_D(z)$.

\begin{lemma}
\label{lem:D-Hessian}
Suppose that~\eqref{ass:domain1}, \eqref{ass:domain2},
\eqref{ass:variance}, and~\eqref{ass:vector-field}  hold.
Then for any $r \in (0,\infty)$,
there exists $h_r \in (0,\infty)$ 
(depending on $\sup_{z \in \cD} \| \Sigma (z) \|_{\rm op}$ and $\delta$ as well as $r$) such that 
\begin{equation}
\label{eq:trace-bound}
  \trace \bigl[ \Sigma(z)  H_D (z) \bigr] \geq \delta , \text{ for all } z = (x,y) \in \cD \text{ with } D(z) \leq h_r^2 
  \text{ and } x \in [0,r].
\end{equation}
Moreover, the gradient  of $D$ vanishes on the boundary: $\nabla D(z)=0$ for all $z\in\pcD$.
\end{lemma}

It is intuitively clear that the gradient of the distance function $D:\cD\to\RP$ near the boundary $\pcD$
points in the normal direction to the boundary. In fact, it is not hard to see that the magnitude of the gradient equals $2D(z)$, implying $\nabla D(z)=0$ for all $z\in\pcD$. Elementary, but somewhat tedious calculations show that 
$$\left\|H_D(z)-2\frac{\nabla D(z)(\nabla D(z))^\tra}{\|\nabla D(z)\|^2_{d+1}}\right\|_{\rm{op}} \to 0, \text{ as 
$z$ approaches $\pcD$,}$$
which, together with the uniform ellipticity in assumption~\eqref{ass:variance}, implies~\eqref{eq:trace-bound}.
The routine calculations are omitted.

Recall definition~\eqref{eq:sigma-def} of  $\sigma_r$. 
Define 
$D_t := D (Z_t)$, for $0 \leq t < \taue$,
and, for $h \in (0,\infty)$, 
\[ \tau_{h} := \inf \{ t \in \RP : D_t \leq h^2 \}, ~\text{and}~ \upsilon_h := \inf \{ t \in \RP : D_t \geq h^2 \} .\]

\begin{lemma}
\label{lem:exit1}
Suppose that~\eqref{ass:domain1}, \eqref{ass:domain2},
\eqref{ass:variance}, and~\eqref{ass:vector-field}  hold, with~$\delta>0$ the constant in~\eqref{ass:variance}. 
Then for any $r \in (0,\infty)$,
there exists $h_r \in (0,\infty)$ such that for all $h \in (0, h_r)$,
\[ \sup_{z \in \cD} \Exp_z \left[  \upsilon_h \wedge \sigma_r \right] \leq \frac{2 h^2}{\delta} .\]
\end{lemma}
 \begin{proof}
Recall that $D_t = D ( Z_t )$. It suffices to assume~$D_0 \leq h^2$. 
For $0 \leq t < \taue$, by It\^o's formula and~\eqref{eq:SDE-for-Z},
\[ \ud D_t = 
  \nabla D (Z_t)^\tra  \Sigma^{1/2} (Z_t ) \ud W_t +
\frac{1}{2} \trace \bigl[ \Sigma (Z_t) H_D  (Z_t) \bigr] \ud t, \]
where $H_D (z)$ is the Hessian of~$D$,
since $\langle \nabla D (z) , \phi (z) \rangle = 0$ for all $z \in \pcD$ by Lemma~\ref{lem:D-Hessian}.
Also by Lemma~\ref{lem:D-Hessian}, we have that
$\trace  [ \Sigma  (z) H_D  (z)  ] \geq \delta$
whenever $z = (x,y) \in \cD$ has $x \leq r$ and $D(z) \leq h^2 \leq h_r$.
Thus if $\tau = \upsilon_h \wedge \sigma_r$,
we have that for $0 \leq s \leq t$, 
\[ D_{t \wedge \tau} - D_{s \wedge \tau} \geq M_{t \wedge \tau} - M_{s \wedge \tau} + (\delta/2) (t \wedge \tau - s \wedge \tau ) ,\]
where $M$ is a martingale with $[ M ]_{t \wedge \tau} \leq C t$, for a constant $C < \infty$, 
using the fact that $\| \Sigma^{1/2} (z) \|_{\rm op}$ is bounded, by~\eqref{ass:variance}, and that $\| \nabla D (z) \|$ is bounded
on bounded subsets of $\cD$, by the continuity of the gradient. 
It follows that  $D_{t \wedge \tau} - \frac{\delta}{2} (t \wedge \tau)$ is a submartingale.
By optional stopping applied at the bounded stopping time $t \wedge \tau$,
and since $0 \leq D_{t \wedge \upsilon_h} \leq h^2$ for all $t\geq0$, 
 we get 
\[ \frac{\delta}{2} \Exp_z [ t \wedge \upsilon_h \wedge \sigma_r ] \leq  
\Exp_z \left[ D_{t \wedge \upsilon_h \wedge \sigma_r} - D_0 \right]
\leq h^2 . \]
Then, by monotone convergence, 
$\Exp_z [  \upsilon_h \wedge \sigma_r ]
= \lim_{t \to \infty} \Exp_z [ t \wedge \upsilon_h \wedge \sigma_r] \leq 2 h^2/\delta$.
\end{proof}

The next result shows that, starting at distance at least~$h$ from~$\partial\cD$, the  probability of crossing level~$r$ before getting within distance~$h' \in (0,h)$ of~$\partial\cD$ 
is uniformly positive. 

\begin{lemma}
\label{lem:exit2}
Suppose that~\eqref{ass:domain1}, \eqref{ass:variance}, and~\eqref{ass:vector-field}  hold. 
Fix $r \in \RP$ and $0 < h' < h < \infty$. Then there exists $\eps >0$ such that
\[   \Pr_z \left( \sigma_r < \tau_{h'}  \right) \geq \eps, \text{ for all } z \in \cD \text{ with } D(z) \geq h^2 .\]
\end{lemma}
\begin{proof}
Fix $r \in \RP$ and $0 < h' < h < \infty$. For any $\gamma\geq0$, define 
$\cD_{\gamma,r} := \{ z = (x,y) \in \cD : D(z) \geq \gamma^2, \, x \leq r \}$. 
For a $C^2$ function $u : \cD \to \R$, we define the $\Sigma$-Laplacian of $u$ by
the formula
$\Delta_\Sigma u(z) := \trace[H_u(z)\Sigma(z)]$, where
$H_u(z)$ is the Hessian of~$u$.
Denote 
$\cS_r:=\{z=(x,y)\in\R^{d+1}:x=r\}$ and 
let $\hat \cD_{h',r}$ be a closed domain with $C^2$ boundary satisfying 
$$\cD_{h,r} \subset \hat \cD_{h',r}\subset \cD_{h',r},\quad \cD_{h'',r}\cap \cS_r = \hat \cD_{h',r}\cap \cS_r, 
\text{ for some $h''\in(h',h)$.}
$$
A domain $\hat \cD_{h',r}$ can be obtained from $\cD_{h',r}$  by smoothing corners appropriately.
  A Dirichlet problem on $\hat \cD_{h',r}$ with boundary condition~$f: \partial \hat \cD_{h',r} \to \R$
  is given by 
\begin{align}
    \label{eq:dirichlet1}  \frac{1}{2} \Delta_\Sigma u & = 0, \text{ on } \Int \hat \cD_{h',r}; \\
    \label{eq:dirichlet2} 
    u & = f, \text{ on } \partial \hat \cD_{h',r},
\end{align}
where $\Int \hat \cD_{h',r}$ denotes the interior of $\hat \cD_{h',r}$ in $\R^{d+1}$.

Choose a continuous $f: \partial \hat \cD_{h',r} \to \R$, such that
$f\equiv 1$ on $\cD_{h,r}\cap \cS_r$
and $f\equiv 0$ on
$\partial \hat \cD_{h',r} \setminus (\hat \cD_{h',r}\cap \cS_r)$.
Then, by~\cite[pp.~364--366]{ks}, the function 
\begin{equation*}
      u (z) := \Exp_z f ( Z_{\tau} ), \text{ where $\tau:=\inf\{t\in\RP: Z_t\in \partial \hat \cD_{h',r}\}$,} 
     \end{equation*}
solves the Dirichlet problem in~\eqref{eq:dirichlet1}--\eqref{eq:dirichlet2}.
Moreover 
$f(Z_\tau)\leq \1{\sigma_r<\tau_{h'}}$ a.s., implying that 
$u(z)\leq \Pr_z(\sigma_r<\tau_{h'})$ for all $z\in\cD_{h,r}$.
Since $f$ is continuous, the domain $\hat \cD_{h',r}$ has $C^2$ boundary and thus 
satisfies the inside sphere property (see~\cite[p.~55]{friedman} for definition),
and the coefficients in~\eqref{eq:dirichlet1}--\eqref{eq:dirichlet2} are continuous and uniformly elliptic by Assumption~\eqref{ass:variance},
the maximum principle~\cite[Thm~21, p.~55]{friedman} yields  
$\inf_{z\in\cD_{h,r}} u(z)>0$, implying the lemma.
\end{proof}

Next is a diffusive upper bound on the exit from a compact subset of the interior of~$\cD$.

\begin{lemma}
\label{lem:exit3}
Suppose that~\eqref{ass:domain1}, \eqref{ass:variance}, and~\eqref{ass:vector-field}  hold, 
and let $\delta>0$ be the constant in~\eqref{ass:variance}. Then, for any $h \in (0,\infty)$,
$r\in\RP$ and $z\in\cD$,
we have
\[  \Exp_z \left[ \sigma_r \wedge \tau_{h}  \right] \leq \frac{\max\{r^2,x_0^2\}}{\delta},\]
where $x_0$ is such that 
$\Pr_z(X_0=x_0)=1$.
\end{lemma}
\begin{proof}
Let $r \in \RP$ and write $\tau = \tau_h \wedge \sigma_r$. 
From~\eqref{eq:SDE-for-Z}, we have 
\[ \ud X_t = e_x^\tra \Sigma^{1/2} (Z_t) \ud W_t + e_x^\tra \phi (Z_t) \ud L_t , ~\text{for}~ 0 \leq t < \taue.\]
The process $X^\tau=(X_{t\wedge\tau})_{t\in\RP}$ is by assumption~\eqref{ass:variance} a (true) martingale and
\[ \ud [ X^\tau ]_t = e_x^\tra \Sigma (Z_{t\wedge\tau}) e_x \ud t.\]
The process $M = (X^\tau)^2 - [X^\tau]$ is a martingale under $\Pr_z$ since it has no infinitesimal drift and is bounded by $\max\{r^2, x_0^2\}$.
Under assumption~\eqref{ass:variance} we have that $e_x^\tra \Sigma (z) e_x \geq \delta > 0$
for all $z \in \cD$, so if $Q_t = X_t^2 - \delta t$,
we see that $| Q_{t \wedge \tau} | \leq \max\{r^2, x_0^2\}+ \delta t < \infty$,
and $(Q_{t \wedge \tau})_{t \geq 0}$ is a submartingale.
Hence, for all $t \in \RP$ and all $z \in \cD$, 
\[ 0 \leq x_0^2= \Exp_z [ X_0^2] \leq \Exp_z [ Q_{t \wedge \tau} ] \leq \max\{r^2, x_0^2\} - \delta \Exp_z [ t \wedge \tau ] .\]
It follows by monotone convergence and the fact that $\tau \leq \sigma_r \leq \taue$, a.s., that $\Exp_z \tau = \lim_{t \to \infty} \Exp_z [ t \wedge \tau ] \leq \max\{r^2, x_0^2\} / \delta$, 
for all $z \in \cD$. 
\end{proof}

Now we can combine the preceding lemmas to complete the proof of Theorem~\ref{thm:exit}.

\begin{proof}[Proof of Theorem~\ref{thm:exit}.]
Fix $r \geq 1$ 
and note that it suffices to consider $z\in\cD$ with the first coordinate $x_0\leq r$.
Pick $0 < h' < h < h_r$, where $h_r$ is as in Lemma~\ref{lem:exit1},
and then fix $\eps>0$ so that (by Lemma~\ref{lem:exit2})
$\Pr_z \left( \sigma_r < \tau_{h'}  \right) \geq 2\eps$ whenever $D(z) \geq h^2$.
Markov's inequality and Lemma~\ref{lem:exit3} show
for  $C := 1/(\eps\delta) \in (0,\infty)$
we have
$$\Pr_z \left( \sigma_r \wedge \tau_{h'} \geq C r^2 \right) \leq \Exp_z[\sigma_r \wedge \tau_{h'}]/(Cr^2)\leq 1/(C\delta)=\eps,$$ 
whenever $D(z) \geq h^2$.
(Note that $C$ will depend on~$r$, since $\eps$ does.)
Hence
\[
    \Pr_z \left( \sigma_r < \tau_{h'}, \, \sigma_r \leq C r^2 \right)
    \geq  \Pr_z ( \sigma_r < \tau_{h'} ) -   \Pr_z \left( \sigma_r \wedge \tau_{h'} \geq C r^2 \right) \geq 2\eps -\eps ,
\]
whenever $D(z) \geq h^2$. 
It follows that
\begin{align}
\label{eq:exit-right}
 \Pr_z \left( \sigma_r \leq C r^2 \right) \geq \eps , \text{ for all } z \in \cD \text{ with } D(z) \geq h^2 .\end{align}

Define $T_0 := 0$ and, for $k \in \N$, the stopping times
\begin{align*}
    S_{k}  := \inf \{ t \geq T_{k-1}  : D_t \leq (h')^2 \} , \text{ and }
    T_{k}  := \inf \{ t \geq S_k + 1 + C r^2 : D_t \geq h^2 \} .\end{align*}
 The strong Markov property (applied at $S_k + 1 + C r^2$) and Lemma~\ref{lem:exit1} imply that
\begin{equation}
    \label{eq:boundary-escape}
 \Exp [ T_k \wedge \sigma_r - S_k \mid \cF_{S_k} ] \leq B_1, \text{ on } \{ S_k < \sigma_r \} ,
 \end{equation}
 where $B_1 := 1 + C r^2 + (2h^2/\delta) < \infty$. 
It follows from~\eqref{eq:boundary-escape} that $T_k\wedge\sigma_r < \infty$
 whenever $S_k\wedge\sigma_r < \infty$, a.s. On the other hand, the strong Markov property and
 Lemma~\ref{lem:exit3} show 
 \begin{equation}
    \label{eq:reach-boundary}
 \Exp [ S_{k+1} \wedge \sigma_r - T_k \mid \cF_{T_k} ] \leq \max\{r^2,X_{T_k}^2\}/\delta\leq r^2/\delta, \text{ on } \{ T_k < \sigma_r \} .
 \end{equation}
 In particular, 
 $T_k \wedge \sigma_r < \infty$ implies $\sigma_r \wedge S_{k+1} < \infty$, a.s.
Since $T_0\wedge \sigma_r=0$ a.s., 
 it follows that $\sigma_r \wedge T_k < \infty$, a.s., for every $k \in \ZP$.
 Since $T_k \geq k$, a.s., we therefore have that $\sigma_r = \lim_{k \to \infty} ( \sigma_r \wedge T_k )$, a.s.
 
Note that $D_{T_k}\geq h^2$ on $\{T_k<\infty\}$ a.s. The strong Markov property (at $T_k$) and~\eqref{eq:exit-right} show that
 \begin{equation}
     \label{eq:exit-right-stopped}
     \Pr ( \sigma_r \leq T_k + C r^2 \mid \cF_{T_k} ) \geq \eps, \text{ on } \{ T_k < \infty \}. 
 \end{equation}
Let $K := \inf \{ k \in \ZP : \sigma_r \leq T_k + C r^2 \}$.
Since $$ \{ K > k+1 \} = \{ K > k,\sigma_r > T_{k+1} + C r^2  \} \subseteq \{ K > k \} \subseteq \{ T_k < \sigma_r-Cr^2 \} \subseteq \{ T_k < \infty \},$$ and
$T_{k+1} >T'_k$ for the stopping time $T'_k:= T_k + Cr^2$, we have that,
for any $k \in \ZP$,
\begin{align*}
     \Pr ( K > k+1 \mid \cF_{T'_k} )
     & =\Pr ( \sigma_r > T_{k+1} + C r^2 \mid \cF_{T'_k } ) \1 { K > k }  \\
     & = \Exp[\Pr ( \sigma_r > T_{k+1} + C r^2 \mid \cF_{T_{k+1} } ) \1{T_{k+1}<\infty} \mid \cF_{T'_k } ] \1 { K > k } \\
     & \leq ( 1 - \eps) \1 { K > k } , \end{align*}
     using  $\{ K > k \}\in\cF_{T'_k }$ and~\eqref{eq:exit-right-stopped}.
     It follows that $\Pr_z ( K > k+1 ) \leq (1-\eps) \Pr_z (K > k)$ for all $z\in\cD$.
     Iterating this argument shows that $\Pr_z ( K > k ) \leq (1-\eps)^k$
     for $k \in \ZP$, and so 
     \begin{equation}
         \label{eq:uniform_bound_exp_K}
     \Exp_z K \leq 1/\eps < \infty, \text{ for all\ } z\in\cD.
      \end{equation}
Now, on $\{ T_k < \sigma_r \}$,
\begin{align*}
    T_{k+1} \wedge \sigma_r - T_k \wedge \sigma_r & = \left( T_{k+1} \wedge \sigma_r 
    - S_{k+1} + S_{k+1} \wedge \sigma_r
    - T_k \right) \1 { \sigma_r > S_{k+1}} \\
    & {} \qquad {} 
    + \left( S_{k+1} \wedge \sigma_r  - T_k \right) \1 { \sigma_r \leq S_{k+1} } \\
    & \leq \left( T_{k+1} \wedge \sigma_r - S_{k+1} \right)  \1 {  S_{k+1} < \sigma_r }
    + \left(  S_{k+1} \wedge \sigma_r
    - T_k \right) .
\end{align*}
It follows that, on $\{ T_k < \sigma_r\}$,
\begin{align}
\nonumber
     \Exp \bigl[ T_{k+1} \wedge \sigma_r - T_k \wedge \sigma_r  \bigmid \cF_{T_k} \bigr]
& \leq \Exp \bigl[  \left( T_{k+1} \wedge \sigma_r - S_{k+1} \right)  \1 {  S_{k+1} < \sigma_r } \bigmid \cF_{T_k} \bigr] \\
\nonumber
& {} \qquad {} + \Exp \bigl[  S_{k+1} \wedge \sigma_r
    - T_k  \bigmid \cF_{T_k} \bigr]  \\
    & \leq B_1 + \delta^{-1} r^2 ,
    \label{eq:bound_on_K}
    \end{align}
    using~\eqref{eq:boundary-escape} (with $\cF_{S_{k+1}}\supset \cF_{T_k}$) and~\eqref{eq:reach-boundary}.
    Since, on the event $\{ T_k \geq \sigma_r\}$, we have 
    $T_{k+1} \wedge \sigma_r - T_k \wedge \sigma_r=0$, it holds that
    $T_{k+1} \wedge \sigma_r - T_k \wedge \sigma_r=(T_{k+1} \wedge \sigma_r - T_k \wedge \sigma_r)\1 { T_k < \sigma_r }$.
    In particular,
  for any $k \in \ZP$,
\begin{align*}
     \Exp_z [ T_k \wedge \sigma_r ] & = \sum_{\ell =0}^{k-1} \Exp_z  \bigl[
   ( T_{\ell+1} \wedge \sigma_r - T_\ell \wedge \sigma_r )  \1 { T_\ell < \sigma_r }  \bigr] \\
   & \leq ( B_1 + \delta^{-1} r^2 ) \Exp_z \sum_{\ell=0}^{k-1} \1 { T_\ell < \sigma_r } \leq ( B_1 + \delta^{-1} r^2 ) \Exp_z [ K  ],
   \end{align*}
   since $T_{K+1} > T_K + Cr^2 \geq \sigma_r$, a.s. We conclude
   by monotone convergence that
   \[ \Exp_z \sigma_r = \lim_{k \to \infty} \Exp_z [ T_k \wedge \sigma_r ]
 \leq ( B_1 + \delta^{-1} r^2 )/\eps < \infty, \]
 by~\eqref{eq:uniform_bound_exp_K},
 which shows that $\sup_{z \in \cD} \Exp_z \sigma_r < \infty$.
 
We now prove $\Pr_z ( \limsup_{t \uparrow \taue} X_t = +\infty ) = 1$ for all $z \in \cD$. By Theorem~\ref{thm:existence} we know that $\lim_{t\uparrow\taue} X_t=\infty$ on the event $\{\taue<\infty\}$.
 Thus it suffices to prove 
 that $\limsup_{t \uparrow \infty} X_t \geq r $ holds $\Pr_z$-a.s. on the event 
 $\{\taue=\infty\}$
 for all $z \in \cD$ and all $r \in \RP$.
To see this, fix $r \in \RP$, set $t_0 := \sigma_{r+1}$ and, for $k \in \N$,
 \[ s_k := \inf \{ t \geq t_{k-1} : X_t \leq r \}, ~\text{and}~ t_k := \inf \{ t \geq 1+s_k : X_t \geq r + 1 \} .\]
If $s_k = \infty$ for some $k$, then $\liminf_{t \to \infty} X_t \geq r$, 
as required.
For every $k\in\N$, on the event
$\{ s_k < \infty \}$
we have $t_k<\infty$,
because $\Exp[t_k-s_k \mid \cF_{s_k}]<\infty$ as a consequence of the strong Markov property and
the fact that $\sup_{z\in\cD}\Exp_z[\sigma_{r+1}]<\infty$.
Thus, almost surely, either there exists $k\in\N$ such that $s_k=\infty$, or for all $k\in\N$
we have
$t_k<\infty$ and 
$t_k\uparrow\infty$ as $k\to\infty$. In either cases $\limsup_{t \to \infty} X_t \geq r$, 
as required.
\end{proof}

\section{Explosions and strong laws for reflecting diffusions}
\label{sec:lln}

\subsection{Lyapunov functions}
In this section, we turn to the proof of our main result, Theorem~\ref{thm:lln}.
A key element in our proofs is a Lyapunov function~$g$, mapping $\cD$ to $\RP$, that will enable us to apply the martingale
results from Section~\ref{sec:martingales} to the multidimensional reflecting diffusion.
Recall that $b$~is the function that defines the domain~$\cD$ via~\eqref{eq:domain-def},
and suppose that $b$ is twice  differentiable, as in~\eqref{ass:domain1}. 
Pick a $C^2$ function $\tilde b:\RP\to(0,\infty)$ that coincides with $b$ on $(1,\infty)$
(i.e. $\tilde b(x)=b(x)$  for all $x > 1$)
and has bounded derivatives on compact sets in $\RP$.
Define $g(z) := g(x,y)$ for $z = (x,y) \in \RPRd$ 
by 
\begin{equation}
\label{eq:g-def}   g(x,y)  := x + \gamma \frac{\| y \|_d^2}{\tilde b(x)}.
\end{equation}
The parameter  $\gamma \in \R$ will be tuned below in the proof of our strong law, see Lemma~\ref{lem:xi-qv}. 
(We use $\tilde b$ in~\eqref{eq:g-def} rather than $b$ to avoid a blow-up of the derivatives of $g$ at the origin;
for what we need in this section only the large-$x$ behaviour of~$g$, is important.)
The intuition behind the choice of $g$ at~\eqref{eq:g-def} is that $g(x,y) \approx x$ but, while $X_t$ has zero drift in the interior, the curved level sets of $g$ produce (positive but small) drift for $g(Z_t)$, while
$\gamma$ can be tuned to control the sign of the local-time drift arising from the reflection: see Figure~\ref{fig:lyapunov-function} for a pictorial description.

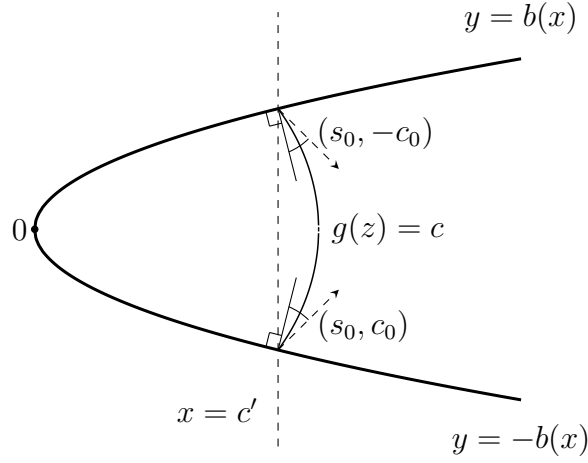
\begin{figure}[!h]
\begin{center}
\begin{tikzpicture}[domain=0:8, scale = 0.8]
\filldraw (0,0) circle (1.5pt);
\node at (-0.25,0) {$0$};
\draw[black, line width = 0.40mm]   plot[smooth,domain=0:8,samples=500] ({\x},  {(\x)^(1/2)});
\draw[black, line width = 0.40mm]   plot[smooth,domain=0:8,samples=500] ({\x},  {-(\x)^(1/2)});
\node at (5.8,0.0) {$g(z)=c$};
\draw[black,dashed] (4,-3.6) -- (4,3.6);
\node at (3,-3) {$x=c'$};
\draw[black, line width = 0.20mm]   plot[smooth,domain=0:0.667,samples=100] ({4+\x},  {(abs((14*(4+\x)^(1/2)-3*(4+\x)^(3/2))))^(1/2)});
\draw[black, line width = 0.20mm]   plot[smooth,domain=0:0.667,samples=100] ({4+\x},  -{abs(((14*(4+\x)^(1/2)-3*(4+\x)^(3/2))))^(1/2)});
\draw[black, line width = 0.19mm] (4.6666,0.035) -- (4.6666,-0.035);
\draw (4,2) -- (4.3,0.8);
\draw (3.8,1.95) -- (3.86,1.71);
\draw (4.06,1.76) -- (3.86,1.71);
\draw (4.5,1.5) arc (-45:-67.5:1);
\node at (8, 3.5)       {$y = b(x)$};
\node at (8, -3.5)      {$y = - b(x)$};
\draw (4,-2) -- (4.3,-0.8);
\draw (4.5,-1.5) arc (45:67.5:1);
\draw (3.8,-1.95) -- (3.86,-1.71);
\draw (4.06,-1.76) -- (3.86,-1.71);
\draw[black,->,>=stealth,dashed] (4,-2) -- (5,-1);
\draw[black,->,>=stealth,dashed] (4,2) -- (5,1);
\node at (5.6, 1.6)       {$(s_0,-c_0)$};
\node at (5.4, -1.6)       {$(s_0,c_0)$};
\end{tikzpicture}
\end{center}
\caption{\label{fig:lyapunov-function} An illustration of a level curve of function $g(z)$.
Note that the parameter $\gamma$ in the definition of $g$ in~\eqref{eq:g-def} modulates the curvature of the level set $\{z\in\cD:g(z)=c\}$. The vector field $\phi$ driving the reflection is asymptotically tangent to the level set of $g$ at the boundary $\partial \cD$ for appropriate choice of $\gamma$, see~\eqref{eq:nu-def} and  Lemma~\ref{lem:nu-properties} below.}
\end{figure}

Note that $\sup_{y :\| y \|_d \leq b(x)} | g(x,y) - x | \leq |\gamma| b(x)$.
Define the $\Sigma$-Laplacian of $g$ by $\Delta_\Sigma g(z) := \trace[H_g(z)\Sigma(z)]$, where
$H_g(z)$ is the Hessian of $g$.
Lemma~\ref{lem:g-derivatives} below gives some basic asymptotic properties of~$g$ and its derivatives.

In coordinates, write $z \in \R^{d+1}$ as $z = (x,y) = (x, y_1, \ldots, y_d)$,
where $y \in \R^d$ and $x, y_1, \ldots, y_d \in \R$. 
Let $\partial_x$, $\partial_{y_i}$ denote partial differentiation with respect to $x$, $y_i$, respectively,
and write $\nabla_y = (\partial_{y_1}, \ldots, \partial_{y_d})$ for the partial gradient
with respect to $y$, and $\nabla = (\partial_x, \nabla_y)$ for the (total) gradient. The proofs of Lemmas~\ref{lem:g-derivatives}
and~\ref{lem:nu-properties} is based on deterministic calculations
is in Subsection~\ref{subsec:deterministic_proofs} below.

\begin{lemma}
\label{lem:g-derivatives}
Suppose that $b : (0,\infty) \to (0,\infty)$ is twice differentiable. Then
\begin{align}
\label{eq:partial-g}
\partial_x g(x,y)   = 1 - \gamma\frac{b'(x)}{b(x)^2} \| y \|_d^2, \text{ and } \nabla_y g(x,y)   = 2\gamma \frac{y}{b(x)}  , \text{ for } x > 1  .\end{align}
Moreover, if  $\lim_{x \to \infty} b'(x) = 0$, then
\begin{thmenumi}[label=(\roman*)]
\item $\lim_{x \to \infty} \sup_{y : \| y\|_d \leq b(x)} | x^{-1} g(x, y) - 1 | = 0$; and
\item $\sup_{z \in \cD}  \| \nabla g(z) \|_{d+1} < \infty$.
\end{thmenumi}
If, in addition, $\lim_{x \to \infty} b(x)b''(x) = 0$,
$\sup_{z \in \cD} \| \Sigma (z) \|_{\rm op} < \infty$, and~\eqref{eq:Sigma-limit} holds, then
\begin{thmenumi}[label=(\roman*),resume] 
\item   $\lim_{x \to \infty} \sup_{y : \| y \|_d \leq b(x)} | \frac{1}{2} b(x) \Delta_{\Sigma} g(x,y) - \gamma \sigma^2 | = 0$.
\end{thmenumi}
\end{lemma}

Consider 
$\nu: \pcD \to \R$ defined by
\begin{equation}
\label{eq:nu-def}
\nu (z)   := \langle \phi( z ) ,  \nabla g( z ) \rangle , \end{equation}
which appears in the local-time contribution to the drift of the process~$g(Z)$.
The next result will allow us to control the sign of~$\nu$, which will enable
us to dispense with local-time terms appearing in our It\^o formula calculations
for~$g(Z_t)$ and related processes.

\begin{lemma}
\label{lem:nu-properties}
Suppose that $b: (0,\infty) \to (0,\infty)$ is twice differentiable, that $\lim_{x \to \infty} b'(x) =0$,
that $\sup_{z \in \pcD} \| \phi(z) \|_{d+1} < \infty$,
 and that~\eqref{eq:e1-projection} and~\eqref{eq:z-projection} hold. 
Then, 
\begin{thmenumi}[label=(\roman*)]
\item if $s_0 - 2 \gamma c_0 < 0$, then $\sup_{y : \|y\|_d = b(x)} \nu (x,y) \leq 0$ for all $x$ sufficiently large;
\item if $s_0 - 2 \gamma c_0 > 0$, then $\inf_{y : \|y\|_d = b(x)} \nu (x,y) \geq 0$ for all $x$ sufficiently large.
\end{thmenumi}
\end{lemma}

\subsection{Escape probability} 
The next result gives an escape probability estimate and establishes `transience'.
Recall that $(\cF_t)_{t\in\RP}$ denotes the filtration of the driving Brownian motion in~\eqref{eq:SDE-for-Z}. 
Let $\gamma >0$. 
For $0 \leq t < \taue$ set $\kappa_t := g(Z_t)$, where $g$ satisfies~\eqref{eq:g-def},
 and define $\kappa_t := \infty$ for $t \geq \taue$. 
 Note that $\kappa=(\kappa_t)_{t\in\RP}$ is a continuous process taking values in $[0,\infty]$.
 We use the notation $\lambda$ and $\rho$ for the passage times for $\kappa$ as defined in~\eqref{eq:lambda-rho}.

\begin{proposition}
\label{prop:transience}
Suppose that~\eqref{ass:domain1}, \eqref{ass:domain2}, \eqref{ass:variance}, 
\eqref{ass:vector-field}, and~\eqref{ass:lln} hold. 
Then for all $\ell \in \RP$ and all $\eps >0$, there exists $x > \ell$ such that, for every $(\cF_t)$-stopping time $T$,
\[
\Pr ( \lambda_{\ell , T} < \rho_\infty \mid \cF_T ) \leq \eps, \text{ on } \{ \kappa_T \geq x , \, T < \rho_\infty \} .
\]
Moreover, $\lim_{t \uparrow \rho_\infty} \kappa_t = \infty$ a.s.
\end{proposition}

Since, by~\eqref{eq:g-def}, we have
 $x \leq g(x,y) \leq x + \gamma b(x)$ for all $x>1$ and, by assumption~\eqref{ass:domain2}, the upper bound $x\mapsto x+\gamma b(x)$ is monotonically increasing for large $x$, it follows that
\begin{equation}
    \label{eq:rho-sigma}
\rho_r \leq \sigma_r \leq \rho_{r+ \gamma b(r)}, \text{ for all sufficiently large $r \in \RP$.} 
\end{equation}
Thus, $\taue=\rho_\infty$. Moreover, $\lim_{t\uparrow\taue} X_t=\infty$ if and only if $\lim_{t \uparrow \rho_\infty} \kappa_t = \infty$.

\begin{proof}[Proof of  Proposition~\ref{prop:transience}]
By It\^o's formula and~\eqref{eq:SDE-for-Z}, 
\begin{align}
\label{eq:SDE-g}
\kappa_t & 
 =g(z_0)+ \int_0^t \nu (Z_s) \ud L_s +  \frac{1}{2} \int_0^t \Delta_{\Sigma} g (Z_s) \ud s 
+ M'_t , \text{ for } 0 \leq t < \taue , \end{align}
where~$\nu$ is as defined at~\eqref{eq:nu-def}, and $M'$ is a local martingale given by 
\begin{equation}
    \label{eq:M-dash}
 M'_t = \int_0^t  \langle \nabla g( Z_s) , \Sigma^{1/2} (Z_s) \ud W_s \rangle   , \text{ for } 0 \leq t < \taue .\end{equation}
Note that
\begin{equation}
\label{eq:M-dash-qv}
 [ M' ]_t \leq \int_0^t  \| \nabla g( Z_s) \|_{d+1}^2 \| \Sigma^{1/2} (Z_s) \|_{{\rm op}}^2 \ud s \leq C t ,    ~\text{for}~ 0 \leq t < \taue ,\end{equation}
for a constant $C < \infty$,
by~\eqref{ass:variance} and Lemma~\ref{lem:g-derivatives}(ii).
Define
$f : [0,\infty] \to [0,1]$ by
$f(x) := 1/(1+x)$ if $x \in \RP$, and $f(\infty) := 0$.
By It\^o's formula and~\eqref{eq:SDE-g}, for $0 \leq t < \taue $,
\begin{align}
\label{eq:SDE-f-kappa}
 f ( \kappa_t ) 
& = f(g(z_0))- \int_0^t  (1+ \kappa_t  )^{-2} 
\left( G(Z_t)  
  \ud t  
+  \nu (Z_t) \ud L_t + \ud M'_t \right) , \end{align}
where, for all $z=(x,y)$ with $x>1$, $G$ satisfies 
\[ G (z) = G(x,y) =  \frac{1}{2} \Delta_{\Sigma} g(z) -   (1 + g(z) )^{-1}
 \| \Sigma^{1/2} (z) \nabla g (z) \|_{d+1}^2 .\]
Note $\lim_{x \to \infty} ( b(x) (1 + g(z) )^{-1} ) = \lim_{x \to \infty}  x^{-1} b(x) = 0$.
By Lemma~\ref{lem:g-derivatives}(ii) and the boundedness of $\Sigma^{1/2}$,  
$\| \Sigma^{1/2} (z) \nabla g (z) \|_{d+1}^2 = O(1)$ as $x\to\infty$. Hence Lemma~\ref{lem:g-derivatives}(iii) yields
\begin{equation}
\label{eq:F-lim}
 \lim_{x \to \infty} \sup_{y: \| y \|_d \leq b(x) } \bigl|b(x) G(x,y) - \gamma \sigma^2 \bigr| = 0 .\end{equation}
Suppose that $\gamma \in (0, \frac{s_0}{2c_0})$.
By~\eqref{eq:F-lim}, Lemmas~\ref{lem:g-derivatives}(i) and~\ref{lem:nu-properties}, 
and the fact that~$b$ is bounded on compact sets, 
 there exists $\ell_0 > 1$ so that,
 for every $x \geq \ell_0$ we have 
\begin{equation}
\label{eq:F_nu_bound}
b(x)G(x,y) \geq \gamma \sigma^2 /2
\text{ if } (x,y) \in \cD,
\text{ and }
\nu(x,y) \geq 0
\text{ if } (x,y) \in \pcD.
\end{equation}

For any stopping time $T$ and any
$\ell > \ell_0$ and $r > \ell$, 
define the stopping time  $S:=\lambda_{\ell,T}\wedge\rho_{r,T}$,
where $\lambda_{\ell,T}$ and $\rho_{r,T}$ are given in~\eqref{eq:lambda-rho}. Note that for any time 
$t\in [T,S]$
in the stochastic interval 
we have $\kappa_t\geq \ell_0$.
Thus, by~\eqref{eq:SDE-f-kappa} and~\eqref{eq:F_nu_bound}, on the event 
$\{T<\rho_\infty\}$,
for any $0 \leq s \leq t$
we have 
\begin{equation}
\label{eq:loc_mart_N}
\Exp\bigl[ f(\kappa_{(t+T)\wedge S})-f(\kappa_{(s+T)\wedge S})\bigmid \cF_{s+T}\bigr]\leq
-\Exp [ N_{t-s} \mid \cF_{s+T} ],
\end{equation}
where the local martingale 
$N_v := \int_{(s+T)\wedge S} ^{(v+s+T)\wedge S} (1+ \kappa_u)^{-2} \ud M'_u$
has quadratic variation bounded as
$[N ]_{v } \leq [M']_{(v+s+T)\wedge S}- [M']_{(s+T)\wedge S} \leq C v$ for all $v\in\RP$ (the constant $C$ is as in~\eqref{eq:M-dash-qv}).
Thus $N$ is a true martingale and~\eqref{eq:loc_mart_N} implies 
\begin{equation}
\label{eq:sup_mart_property_reflection}
\Exp\left[f(\kappa_{(t+T)\wedge S})\mid \cF_{s+T}\right]\leq
f(\kappa_{(s+T)\wedge S}), \text{ on the event } \{T<\rho_\infty\}.
\end{equation}
This implies the hypothesis~\ref{item:transience-2} in Theorem~\ref{thm:transience}. Since hypothesis~\ref{item:transience-1} clearly holds for the function $f(x)=(1+x)^{-1}$,  Theorem~\ref{thm:transience}
implies the proposition. 
 \end{proof}

\subsection{Linearization transformation}

To quantify the rate of escape of our process, we transform $g(Z_t)$
to obtain a process that grows approximately linearly, in a sense we describe shortly.
Recall the definition of~$B$ from~\eqref{eq:B-def} and that the Lyapunov function~$g$ satisfies~\eqref{eq:g-def}.
Consider the $[0,\infty]$-valued process $B(\kappa)=(B(\kappa_t))_{t\in\RP}$, satisfying 
$B(\kappa_t) = B ( g ( Z_t) )$ for $t < \taue$ and $B(\kappa_t) = B(\infty)$ for $t \geq \taue$ (recall also $\taue=\rho_\infty$ a.s.).

\begin{lemma}
\label{lem:xi-properties}
Suppose that~\eqref{ass:domain1}, \eqref{ass:domain2}, \eqref{ass:variance}, \eqref{ass:vector-field}, and~\eqref{ass:lln} hold. 
There exist a function $\mu : \cD \to \R$ and a process $M=(M_t)_{t\in[0,\rho_\infty)}$ 
such that
\begin{align}
\label{eq:xi-ito}
   B(\kappa_t) & = B(g(z))+ \int _0^t \mu ( Z_s ) \ud s +\int _0^t  b ( g ( Z_s ) ) \nu ( Z_s) \ud L_s +   M_t , \text{ for } 0 \leq t < \rho_\infty,\end{align}
  where the function $\nu$ is given in~\eqref{eq:nu-def}.
  Moreover, $\mu$ and $M$ have the following properties.
\begin{thmenumi}[label=(\roman*)]
      \item  It is the case that
      \begin{equation}
\label{eq:mu-bound}
\lim_{x \to \infty} \sup_{y : \|y\|_d \leq b(x)} \bigl| \mu (x,y) - \gamma \sigma^2 \bigr| = 0 .\end{equation}
      \item 
      	There exists a constant $C < \infty$ such that, a.s., for all $0 \leq t < \rho_\infty$,
	\begin{equation}
	\label{eq:M-qv-bound}
	[ M ]_t \leq C \int_0^t b ( g(Z_s) )^2 \ud s .\end{equation}
	\item
	\label{lem:xi-properties-iii}
	The process $N=(N_t)_{t\in\RP}$,
	given by 
$N_t:=(M_{(t +T)\wedge \lambda_{\ell,T}\wedge  \rho_{r,T}}-M_T)\1{T<\rho_\infty}$, defined for any $0\leq \ell\leq r$ (where $\lambda_{\ell,T}$ and $\rho_{r,T}$ are given in~\eqref{eq:lambda-rho} for $\kappa=g(Z)$) and $(\cF_t)$-stopping time $T$, is 
a continuous, uniformly integrable $\R$-valued $(\cF_{t+T})$-martingale.
  \end{thmenumi}
\end{lemma}

\begin{proof}
By It\^o's formula, for $0 \leq t < \rho_\infty$,
\[ \ud B(\kappa_t) = b ( g (Z_t) ) \ud \kappa_t + \frac{1}{2} b' ( g(Z_t) ) \ud [\kappa ]_t ,\]
where $\kappa$ satisfies~\eqref{eq:SDE-g}.
Thus we obtain~\eqref{eq:xi-ito} with
\begin{align*}
 \mu ( z ) & := \frac{1}{2} \Bigl[ b ( g ( z ) ) \Delta_{\Sigma} g ( z ) + b' ( g( z) ) \| \Sigma^{1/2} (z) \nabla g ( z ) \|_{d+1}^2  \Bigr]  , \end{align*}
and $\ud M_t = b ( g (Z_t) ) \ud M'_t$ where $M'$ is given by~\eqref{eq:M-dash}.
Note that $\ud [ M]_t = b ( g (Z_t) )^2 \ud [ M']_t$ where $ [M']_t \leq C t$, by~\eqref{eq:M-dash-qv}. 
Thus, by integrating, we obtain~\eqref{eq:M-qv-bound} establishing~(ii).

Set $C_r = C \sup_{(x,y) \in \cD, \, x \leq r} b(g(x,y))^2 < \infty$. On the event $\{T<\rho_\infty\}$, analogous argument to the one that established~\eqref{eq:M-qv-bound} implies the following inequality for all $t \in \RP$:
\[ [ M ]_{(t +T)\wedge \lambda_{\ell,T}\wedge  \rho_{r,T}}-[M]_T \leq C \int_T^{(t+T) \wedge \rho_{r,T}} b ( g(Z_s) )^2 \ud s \leq C_r ( (t+T)\wedge \rho_{r,T}-T ) .\]
Since, by~\eqref{eq:rho-sigma}, 
we have $\rho_r\leq \sigma_r$ a.s.~for all $r$ sufficiently large ($\sigma_r$ is defined in~\eqref{eq:sigma-def}),
the continuous 
local martingale $N$ is a uniformly integrable martingale since it is bounded in $L^2$ by Theorem~\ref{thm:exit}:
\begin{align}
\nonumber
\sup_{t \geq 0} \Exp_{z_0} [ N^2_t] 
& \leq C_r \Exp_{z_0} [(\rho_{r,T}-T)\1{T<\rho_\infty}]\\
&= C_r \Exp_{z_0}\left[\Exp [\rho_{r,T}-T\mid\cF_{T}]\1{T<\rho_\infty}\right]
\leq C_r \sup_{z\in\cD}\Exp_z \sigma_r < \infty,
\label{eq:bounded_expected_exit_time}
\end{align}
for any starting point $z_0\in\cD$ of $Z$.  
This implies~(iii).

Finally, by the mean value theorem, $b(g(x,y)) - b(x) = (g(x,y) - x) b'(x+\theta (g(x,y) - x))$,
where $\theta = \theta(x,y) \in [0,1]$, and since $| g(x,y) - x| = O ( b(x) )$ (by~\eqref{eq:g-def})
and $b'(x) = o(1)$ (by~\ref{ass:domain2}) as $x\to\infty$, we have $| b(g(x,y)) - b(x) | = o ( b(x) )$ as $x\to\infty$.
Then~\eqref{eq:mu-bound} follows from Lemma~\ref{lem:g-derivatives}, and we obtain~(i).
\end{proof}

Let $\eps>0$.
Note that, for $\theta_\pm=\gamma \sigma^2 \pm \eps$,
Lemmas~\ref{lem:nu-properties} and~\ref{lem:xi-properties} suggest
that, before exiting a bounded set, the process 
$\zeta^{(B,\theta_+)}$ 
(resp.\ $\zeta^{(B,\theta_-)}$),
defined in~\eqref{eq:v-def}--\eqref{eq:zeta-def},
is a supermartingale (resp.\ submartingale) if
$2\gamma c_0 > s_0$
(resp.\ $2\gamma c_0 < s_0$).
In order to understand whether the process $X$ is explosive or superdiffusive using the theory of Section~\ref{sec:martingales}, Lemma~\ref{lem:xi-local-martingales} establishes such a property, starting after an arbitrary $(\cF_t)$-stopping time. 
 The localisation by~$\rho_{r,T}$ is removed 
in the proof of Theorem~\ref{thm:lln}, thus yielding our law of large numbers via Theorem~\ref{thm:martingale-lln}.

\begin{lemma}
\label{lem:xi-local-martingales}
Suppose that~\eqref{ass:domain1}, \eqref{ass:domain2}, \eqref{ass:variance}, \eqref{ass:vector-field}, and~\eqref{ass:lln} hold. 
Recall the function $B$, given in~\eqref{eq:B-def}, and that $\kappa=g(Z)$. For arbitrary $\eps >0$ and $\gamma$, define $\theta_\pm:=\gamma \sigma^2 \pm \eps$ (see~\eqref{eq:Sigma-limit} for the definition of $\sigma^2$).
\begin{thmenumi}[label=(\roman*)]
\item
\label{lem:xi-local-martingales-i}
Pick $\gamma > \frac{s_0}{2c_0}$.
Then there exists $x_1 \in \RP$
such that, for all $r \geq\ell\geq x_1$,
the process 
$\zeta^{(B,\theta_+)}$, 
defined in~\eqref{eq:v-def}--\eqref{eq:zeta-def},
and any $(\cF_t)$-stopping time $T$
satisfy
\[ 
\Exp [ \zeta^{(B,\theta_+)}_t \mid \cF_{s + T} ] \leq \zeta^{(B,\theta_+)}_{s} , \text{ on }
\{  T < \rho_\infty\}, 
\text{ for all $t \geq s \geq 0$}.\] 
\item
\label{lem:xi-local-martingales-ii}
Pick $\gamma \in(0, \frac{s_0}{2c_0})$.
Then there exists $x_1 \in \RP$
such that, for all $r \geq \ell \geq x_1$,
the process 
$\zeta^{(B,\theta_-)}$, 
defined in~\eqref{eq:v-def}--\eqref{eq:zeta-def},
and any $(\cF_t)$-stopping time $T$
satisfy
\[ 
\Exp [ \zeta^{(B,\theta_-)}_t \mid \cF_{s + T} ] \geq \zeta^{(B,\theta_-)}_{s} , \text{ on } \{ T<\rho_\infty\},  \text{ for all $t \geq s \geq 0$}.\] 
\end{thmenumi}
\end{lemma}
\begin{proof}
Suppose that $\gamma > \frac{s_0}{2c_0}$ and $\eps >0$.  
By~\eqref{eq:mu-bound} and Lemma~\ref{lem:nu-properties}, 
there exists $x_1 \in (0,\infty)$ such that $\nu (x,y) \leq 0$ and 
$\mu (x,y) \leq \gamma \sigma^2 + \eps = \theta_+$ for all $x \geq x_1$ and all $y$.
Pick $r \geq \ell\geq x_1$ and a stopping time~$T$. On $\{ T < \rho_\infty,\, \kappa_T \leq \ell \}$, by~\eqref{eq:v-def}, one has $\lambda_{\ell,T}=T$, implying $v_t = v_s =0$
and hence
$\zeta^{(B,\theta_+)}_{t} =\zeta^{(B,\theta_+)}_{s}$
for all $t \geq s \geq 0$.
On $\{ T < \rho_\infty, \, \kappa_T > \ell \}$,
one has $\kappa_u \geq \ell$ for all $u \in [T+v_s, T+v_t]$ since, by definition~\eqref{eq:v-def}, $v_t+T\leq \lambda_{\ell,T}$. Thus, by~\eqref{eq:xi-ito} and the inequalities in the beginning of the paragraph, we get
\begin{align*}
\zeta^{(B,\theta_+)}_{t} - \zeta^{(B,\theta_+)}_{s}  =
-\theta_+ ( v_t - v_s) +
\int_{T+v_s}^{T+v_t} \ud B(\kappa_u)\leq M_{T+v_t} - M_{T+v_s}, \text{ on } \{ T < \rho_\infty \}.
\end{align*}
Lemma~\ref{lem:xi-properties}(iii) implies that 
$( M_{T+v_{t}} - M_T)_{t \in\RP}$ is a uniformly integrable $(\cF_{t+T})$-martingale and thus
$\Exp [ M_{T+v_t} \mid \cF_{s +T} ] =  M_{T+v_s} $ a.s.
It follows that
$\Exp [ \zeta^{(B,\theta_+)}_t - \zeta^{(B,\theta_+)}_s \mid \cF_{s+T}] \leq 0$,
which gives~(i). The argument for~(ii) is similar.
\end{proof}

\subsection{Explosion and passage times}

Theorem~\ref{thm:exit} shows that $\Exp_z \sigma_r < \infty$. The next result 
gives quantitative estimates for  $\Exp_z \sigma_r$ in terms of $B(r)$
under the stronger assumptions in force in Theorem~\ref{thm:lln}.

\begin{proposition}
\label{prop:hitting-time-asymptotics}
Suppose that~\eqref{ass:domain1}, \eqref{ass:domain2}, \eqref{ass:variance}, \eqref{ass:vector-field}, and~\eqref{ass:lln} hold. 
If $B(\infty) < \infty$, then $\sup_{z\in\cD} \Exp_z \taue < \infty$.
On the other hand, if $B(\infty) = \infty$, then, for all $z \in \cD$,
\[ \lim_{r \to \infty} \frac{ \Exp_z \sigma_r}{B(r)} = \frac{2 c_0}{s_0 \sigma^2} .\]
\end{proposition}

The following lemma, proved in Subsection~\ref{subsec:deterministic_proofs} below, gives certain properties of the functions $b$ and $B$, useful in what follows. 

\begin{lemma}
\label{lem:b-B-bound}
Assume $\beta$, defined in~\eqref{eq:beta-def}, satisfies $\beta<1$. Then the following hold. 
\begin{thmenumi}[label=(\roman*)]
\item
\label{lem:b-B-bound-i}
If $B (\infty) < \infty$, then $b(x) = O(1/x)$ as $x \to \infty$,
while if $B (\infty) = \infty$, then there exist $\delta \in (1,2)$
and $C < \infty$ such that 
$b(x)^2 \leq C (1+  B(x)^{2-\delta})$, for all $x \in \RP$.
\item
\label{lem:b-B-bound-ii}
For any $\omega \in \R$, it is the case that
$\lim_{x \to \infty} B (x + \omega b(x))/B(x) = 1$.
\end{thmenumi}
\end{lemma}

  \begin{proof}[Proof of Proposition~\ref{prop:hitting-time-asymptotics}.]
 We will apply Theorem~\ref{thm:explosion} to the process $\kappa=g(Z)$.
 Let $\gamma \geq 0$.  
A consequence of~\eqref{eq:rho-sigma} is that $\Pr ( \taue = \rho_\infty ) = 1$. 
Note that hypothesis~\ref{item:explosion-a}
of Theorem~\ref{thm:explosion}
follows from Proposition~\ref{prop:transience}.
Assumptions~\ref{item:explosion-b} and~\ref{item:explosion-c} of Theorem~\ref{thm:explosion}
are satisfied by~\eqref{eq:rho-sigma}, Theorem~\ref{thm:exit}
and the strong Markov property for $Z$ (cf. the inequality in~\eqref{eq:bounded_expected_exit_time}). 
Take the function~$f$ in Theorem~\ref{thm:explosion} to be $f = B$ as defined at~\eqref{eq:B-def}. 

Suppose first that $B (\infty) = \infty$.
Pick arbitrary $\eps>0$
and
$\gamma>\frac{s_0}{2c_0}$.
Lemma~\ref{lem:xi-local-martingales}\ref{lem:xi-local-martingales-i}
shows that the hypotheses of Theorem~\ref{thm:explosion}\ref{thm:explosion-i} are satisfied for $\theta:=\theta_+=\gamma \sigma^2+\eps$. Since $\rho_r\leq \sigma_r$ by~\eqref{eq:rho-sigma}, and the continuity of $\kappa$ implies $B(r)=B(\kappa_{\rho_r})$,
by Theorem~\ref{thm:explosion}\ref{thm:explosion-i}
we get
\[ \liminf_{r \to \infty} \frac{\Exp_z \sigma_r}{B(r)} \geq
\frac{1}{\theta_+} > \frac{2 c_0}{s_0 \sigma^2} -\eps',\]
for any $\eps'>0$, where the second inequality follows by choosing $\eps$ and $\gamma$ arbitrarily close to $0$ and $\frac{s_0}{2c_0}$, respectively. 
Since $\eps' >0$ was arbitrary, we get the `$\liminf$' half of the desired conclusion in the $B(\infty) = \infty$ case.
The corresponding `$\limsup$' result follows similarly; now using  Lemma~\ref{lem:xi-local-martingales}\ref{lem:xi-local-martingales-ii} 
and Theorem~\ref{thm:explosion}\ref{thm:explosion-ii} (with $\theta:=\theta_-=\gamma \sigma^2-\eps$ and $\gamma\in(0,\frac{s_0}{2c_0})$), and the upper bound
$\sigma_r\leq \rho_{r+\gamma b(r)}$ in~\eqref{eq:rho-sigma}, shows that
\[ \limsup_{r \to \infty}  \frac{\Exp_z \sigma_r}{B(r+ \gamma b(r))} \leq  \frac{2 c_0}{s_0 \sigma^2} .\]
The proof of the $B(\infty) = \infty$ case is then completed by Lemma~\ref{lem:b-B-bound}\ref{lem:b-B-bound-ii}.

If $B(\infty) < \infty$, Lemma~\ref{lem:xi-local-martingales}\ref{lem:xi-local-martingales-ii}
shows that we may apply Theorem~\ref{thm:explosion}(ii). 
In particular, $\sup_{z\in\cD}\Exp_z \taue < \infty$
follows from the final claim in Theorem~\ref{thm:explosion}\ref{thm:explosion-ii}.
\end{proof}

\subsection{Non-explosion and the strong law}

The next result establishes non-explosion, via an application of Theorem~\ref{thm:non-explosion}.

\begin{lemma}
\label{lem:non-explosion}
Suppose that~\eqref{ass:domain1}, \eqref{ass:domain2}, \eqref{ass:variance}, \eqref{ass:vector-field}, and~\eqref{ass:lln} hold. 
If $B(\infty) = \infty$, then $\Pr_z ( \taue = \infty ) = 1$ for every~$z\in\cD$.
\end{lemma}
 \begin{proof}
 Take $\gamma > \frac{s_0}{2c_0}$, and consider~$B(\kappa)^\alpha$ for $\alpha \in (0,1/2)$. By It\^o's formula and~\eqref{eq:xi-ito},
\begin{align*} 
  B(\kappa_t)^\alpha & =  B(g(z))^\alpha+ \int _0^t \alpha B(\kappa_s)^{\alpha-1} \mu ( Z_s ) \ud s +\int _0^t \alpha B(\kappa_s)^{\alpha-1} b ( g ( Z_s ) ) \nu ( Z_s) \ud L_s \\
  & {} \qquad {} +   \int_0^t \alpha B(\kappa_s)^{\alpha-1} \ud M_s + \int_0^t \frac{\alpha(\alpha-1)}{2} B(\kappa_s)^{\alpha-2} \ud [ M ]_s , \text{ for } 0 \leq t < \rho_\infty . \end{align*}
By choice of $\gamma$, with Lemmas~\ref{lem:nu-properties} and~\ref{lem:xi-properties},
for $g(z) \geq \ell$ large enough, 
we have $\nu (z) \leq 0$  and $\mu (z) \leq C < \infty$. 
Fix $\ell$ as above and pick $r>x>\ell$. Thus, there exists $C\in \RP$ such that
\begin{equation}
    \label{eq:xi-alpha-drift}
 B(\kappa_{t\wedge\lambda_{\ell,\rho_x}\wedge\rho_{r,\rho_x}})^\alpha \leq C t + \int_0^{t\wedge\lambda_{\ell,\rho_x}\wedge\rho_{r,\rho_x}} \alpha B(\kappa_s)^{\alpha-1} \ud M_s. 
 \end{equation}
Moreover, 
by~\eqref{eq:M-qv-bound} and Lemma~\ref{lem:b-B-bound}\ref{lem:b-B-bound-i},
for $t\in[0,\lambda_{\ell,\rho_x}\wedge\rho_{r,\rho_x}]$
we obtain
\begin{align}
\nonumber
[ B(\kappa)^\alpha ]_t & \leq 
\int_0^t \alpha^2 B(\kappa_s)^{2\alpha -2} \ud [ M ]_s
\leq \int_0^t C B ( g (Z_s) )^{2\alpha-2} b ( g(Z_s) )^2 \ud s \\
& \leq \int_0^t C' B ( g (Z_s) ) ^{2\alpha-\delta} \ud s \leq C't, \text{ for some constant $C'>0$,}
\label{eq:finite_QV_at_finite_time}
\end{align}
since $2\alpha <1$ and the $\delta$ from Lemma~\ref{lem:b-B-bound}\ref{lem:b-B-bound-i} satisfies $\delta>1$.

Pick $\theta>C$. The process 
$\zeta^{(f,\theta)}$, defined in~\eqref{eq:zeta-def} with $f(x):=B(x)^\alpha$, is a supermartingale by~\eqref{eq:xi-alpha-drift}.
Thus hypothesis~(b) of Theorem~\ref{thm:non-explosion}
holds for  $\zeta^{(f,\theta)}$ with $\theta > C$. 
Hypothesis~(c) of Theorem~\ref{thm:non-explosion}
also holds by~\eqref{eq:finite_QV_at_finite_time}. 
Hypothesis~(a) holds by Proposition~\ref{prop:transience} 
and Hypothesis~(d) is trivial for continuous processes.
Since
$\taue=\rho_\infty$ by~\eqref{eq:rho-sigma}, applying Theorem~\ref{thm:non-explosion} completes the proof of the lemma.
\end{proof}

In order to apply the non-explosive law of large numbers results from Section~\ref{sec:martingales} (namely Theorem~\ref{thm:martingale-lln}) we also need bounds on the quadratic variation of $B(\kappa)$,
where $\kappa = g(Z)$ and the function $g$ satisfies~\eqref{eq:g-def}.

\begin{lemma}
\label{lem:xi-qv}
Suppose that~\eqref{ass:domain1}, \eqref{ass:domain2}, \eqref{ass:variance}, \eqref{ass:vector-field}, and~\eqref{ass:lln} hold. 
Pick $\gamma >0$ and assume $B(\infty) = \infty$.
 For any $z\in\cD$, there exists $\delta\in(1,2)$ 
 such that,  $$\Exp_z \bigl( [B(\kappa)]_t \bigr) = O ( t^{3-\delta} ) ,\text{ as } t \to \infty. $$
\end{lemma}
\begin{proof}
For any $z\in\cD$, the inequalities in~\eqref{eq:rho-sigma}, assumption $B(\infty) = \infty$ and Lemma~\ref{lem:non-explosion} imply
$\Pr_z(\rho_\infty=\infty)=1$.
Pick $\gamma >0$ and define the stopping time $\rho_r$ by~\eqref{eq:lambda-rho} (with $T=0$) for some large $r$, where the process $\kappa=g(Z)$ with $g$ satisfying~\eqref{eq:g-def} with our chosen $\gamma$.
By~\eqref{eq:xi-ito} we have that $[B(\kappa)]_{t \wedge \rho_r} = [ M]_{t \wedge \rho_r}$ for all
$t, r \in \RP$, where~$(M_{t \wedge \rho_r} )_{t \geq 0}$
is a martingale whose quadratic variation satisfies the bound in~\eqref{eq:M-qv-bound} of Lemma~\ref{lem:xi-properties}. 

By Lemma~\ref{lem:b-B-bound}\ref{lem:b-B-bound-i}, we have $b(x)^2 \leq C + C B(x)^{2-\delta}$ for some $\delta \in (1,2)$, $C>0$ and all $x \geq 0$. 
Now, for all $r,t\in\RP$ and stopping time $S$, we have 
\[ \Exp_z \bigl[ b ( \kappa_{t \wedge S} )^2 \bigr] \leq C + C \Exp_z \bigl[ B ( \kappa_{t \wedge S}  )^{2 - \delta} \bigr]
\leq C + C \bigl(  \Exp_z  [ B ( \kappa_{t \wedge S} ) ] \bigr)^{2 - \delta} , \text{ on } \{S<\rho_\infty\}, \]
by Jensen's inequality, which is applicable since $2 - \delta \in (0,1)$.
Define $\gamma_0 := \gamma \vee \frac{s_0}{c_0}$, so that $\gamma_0 > \frac{s_0}{2c_0}$. Let $g_{\gamma_0}$ satisfy~\eqref{eq:g-def} with $\gamma_0$ instead of $\gamma$. Thus
$g(z) \leq g_{\gamma_0} (z)$ for all $z \in \cD$ (recall that $g$ satisfies~\eqref{eq:g-def} with $\gamma$, fixed in the beginning of the proof, and $\gamma\leq \gamma_0$), and $B$ is non-decreasing.
Let the stopping time $\rho'$ be
defined by~\eqref{eq:lambda-rho} for $\kappa'=g_{\gamma_0}(Z)$ and $T=0$.
By~\eqref{eq:M-qv-bound} and the previous display with $S=\rho_r'$ there exists a constant $C'>0$ such that 
\begin{equation}
    \label{eq:qv-gamma-gamma0}
\Exp \bigl( [ M]_{t \wedge \rho'_r} \bigr) \leq C' t + C' \int_0^{t} \bigl(  \Exp_z  [ B ( g_{\gamma_0} (Z_{s \wedge \rho'_r} ) )] \bigr)^{2 - \delta} \ud s.\end{equation}
Lemma~\ref{lem:xi-local-martingales}\ref{lem:xi-local-martingales-i},
applied 
to the process $g_{\gamma_0}(Z)$
with $\theta_+:=
\gamma_0\sigma^2+\eps$ (for arbitrary $\eps>0$),
yields
$\Exp_z  [ B ( g_{\gamma_0} (Z_{t \wedge \rho'_r} )) ] 
\leq C'' + C'' \Exp_z [ t \wedge \rho'_r ]\leq C''(1+t)$
for some constant $C''>1$ (cf.~definitions~\eqref{eq:v-def} and~\eqref{eq:zeta-def}).
Thus from~\eqref{eq:qv-gamma-gamma0} we conclude that
\[ \Exp_z \bigl( [ B(\kappa)]_t\bigr) =\lim_{r \to \infty} \Exp_z \bigl( [ B(\kappa)]_{t \wedge \rho'_r} \bigr) \leq C'  t + C' (C'')^2\int_0^t (1+s)^{2 - \delta} \ud s 
= O ( t^{3-\delta} ) , 
\]
as $t \to \infty$.
\end{proof}

 Finally, we can complete the proof of our main theorem.

\begin{proof}[Proof of Theorem~\ref{thm:lln}.]
Suppose first that $B$ is bounded.
Then Proposition~\ref{prop:hitting-time-asymptotics} shows that
$\sup_{z \in \cD} \Exp_z \taue < \infty$.
The fact that $\lim_{t \uparrow \taue} X_t = \lim_{t \uparrow \taue} L_t = \infty$, a.s.,
is contained in Theorem~\ref{thm:existence}. 
This completes the proof of part~(i) of Theorem~\ref{thm:lln}.

Suppose now that $B(\infty) = \infty$. Here Lemma~\ref{lem:non-explosion}
shows that  $\Pr_z ( \taue = \infty ) = 1$ for every~$z\in\cD$.
The limiting behaviour of the expectations $\Exp_z\sigma_r$, as $r\to\infty$,
is given in  Proposition~\ref{prop:hitting-time-asymptotics}.
It remains to prove the strong law of large numbers in~\eqref{eq:lln}
and the almost sure limit in~\eqref{eq:lln-local-time};
we first apply Theorem~\ref{thm:martingale-lln} to obtain `$\liminf$' and `$\limsup$' results for~$X_t$.
Take $\gamma > 0$, to be tuned later, and let $g$ be given by~\eqref{eq:g-def}. Take $\kappa = g(Z)$, which is $\R_+$-valued for all $t \in \RP$. Hypotheses~\ref{thm:martingale-lln-a} and~\ref{thm:martingale-lln-b} of Theorem~\ref{thm:martingale-lln} hold, by Proposition~\ref{prop:transience}
and Lemma~\ref{lem:xi-qv}, respectively. 

First take $\gamma > \frac{s_0}{2c_0}$. Fix $\eps >0$. 
Let $\zeta=(\zeta_t)_{t\in\RP}$ be as in~\eqref{eq:zeta-def}
with $f=B$, $\theta_+:=\gamma\sigma^2+\eps$, 
$\ell>x_1$ (where $x_1$ is as in Lemma~\ref{lem:xi-local-martingales}),
$T = \rho_x$ for $x > \ell$, 
and $r:=\infty$. 
For any $0<s<t$,
define $u_t:=(t+\rho_x)\wedge\lambda_{\ell,\rho_x}$,
$u_s:=(s+\rho_x)\wedge\lambda_{\ell,\rho_x}$.
By~\eqref{eq:xi-ito} 
in Lemma~\ref{lem:xi-properties}   we have
\begin{align*}
\zeta_{t} - \zeta_{s}  =
-\theta_+ (u_t-u_s) +
\int_{u_s}^{u_t} \ud B(\kappa_u)\leq M_{u_t} - M_{u_s}.
\end{align*}
The process $(M_{u_t})_{t\in\RP}$, defined in the proof of  Lemma~\ref{lem:xi-properties}, is by Lemma~\ref{lem:xi-qv} a local martingale with
integrable quadratic variation (i.e.~$\Exp_z[M]_t<\infty$ for all $t\in\RP$ and $z\in\cD$). Thus, by~\cite[p.~130]{ry}, it follows
$\Exp[M_{u_t} - M_{u_s}\mid\cF_{s+\rho_x}]=0$ a.s., implying that 
$\zeta$ 
is a supermartingale (note that the martingale property here cannot be obtained directly from Lemma~\ref{lem:xi-properties}\ref{lem:xi-properties-iii} as the process is not stopped at $\rho_r$).
Thus, by Theorem~\ref{thm:martingale-lln}\ref{thm:martingale-lln-upper},
we get 
$\limsup_{t\to\infty}B(g(Z_t))/t\leq \gamma\sigma^2+\eps$.
Since $\gamma > \frac{s_0}{2c_0}$ and $\eps >0$ were arbitrary
and, by the monotonicity of $B$, we have $B ( g( Z_t) ) \geq B (X_t)$,
it follows that
\[ \limsup_{t \to \infty} \frac{B(X_t)}{t} \leq \frac{s_0\sigma^2}{2 c_0}, \as \]

Next, take $\gamma \in (0, \frac{s_0}{2c_0})$ and $\eps >0$. 
Then a similar argument based on 
 an application of Theorem~\ref{thm:martingale-lln}(ii)
with $\theta_- := \gamma \sigma^2 - \eps$
shows that $\liminf_{t \to \infty} B(g(Z_t)) /t \geq \frac{s_0\sigma^2}{2 c_0}$, a.s.
Now
$B ( g( Z_t) ) \leq B (X_t  + \gamma b (X_t ) )$,
and so  Lemma~\ref{lem:b-B-bound}\ref{lem:b-B-bound-ii} shows that
\[ \liminf_{t \to \infty} \frac{B(X_t)}{t}
= \liminf_{t \to \infty} \frac{B (X_t  + \gamma b (X_t ) )}{t}
\geq \frac{s_0\sigma^2}{2 c_0}, \as \]
Combining these results gives the limit for $B(X_t)$ in~\eqref{eq:lln}. The limit for $B( \| Z_t \|_{d+1} )$
in~\eqref{eq:lln} follows from  Lemma~\ref{lem:b-B-bound}\ref{lem:b-B-bound-ii} and the fact that $| \| Z_t \|_{d+1} - X_t | \leq \gamma b (X_t)$. 

Finally, we observe that
\begin{align}
   \label{eq:x-projection} 
 X_{t} & = \langle e_x, z \rangle + m_{t} + \ell_t , ~~t < \taue, \text{ where }\\
    \label{eq:m-l-def}
m_{t} & := \int_0^{t}  e_x^\tra \Sigma^{1/2}(Z_s) \ud W_s, ~~  \ell_{t} :=  \int_0^{t} \langle e_x , \phi (Z_s) \rangle \ud L_s, ~\text{for}~ 0 \leq t < \taue.
\end{align}
We have from~\eqref{eq:lln} that $t^{-(1/2)-\eps} X_t \to \infty$ a.s. as $t\to\infty=\taue$~for some $\eps>0$~(see Remark~\ref{rems:lln}\ref{rems:lln-b}),
while the martingale $m$ satisfies $[m]_t \leq C t$ for all $t\in\RP$.
Since $m$ can be viewed as a Brownian motion time-changed by the quadratic variation
$[m]$, we have 
$t^{-(1/2)-\eps}m_t\to0$ a.s.~as $t\to\infty$.
Hence $\lim_{t \to \infty} (m_t / X_t) = 0$, a.s., and so from~\eqref{eq:x-projection} we obtain
$\lim_{t \to \infty} (\ell_t / X_t ) = 1$, a.s.
In particular $\ell_t\to\infty$ a.s. as $t\to\infty$.
Moreover, from~\eqref{eq:e1-projection} and the fact that $X_t \to \infty$, a.s., we have that
$\lim_{t \to \infty} \langle e_x , \phi (Z_t) \rangle = s_0$, a.s.
From~\eqref{eq:m-l-def},
for any $\eps>0$, there exists an a.s. finite random variable $\xi_\eps$,
such that 
$ |\ell_t-s_0 L_t|\leq \eps L_t + \xi_\eps$ a.s. for all $t\in\RP$.
Thus,
$L_t\to\infty$ and
$\ell_t / L_t \to s_0$, a.s., as $t\to\infty$ and we obtain~\eqref{eq:lln-local-time}.
\end{proof}

\subsection{Deterministic calculations and estimates}
\label{subsec:deterministic_proofs}

In this subsection we prove deterministic Lemmas~\ref{lem:g-derivatives},~\ref{lem:nu-properties} and~\ref{lem:b-B-bound}.

\begin{proof}[Proof of Lemma~\ref{lem:g-derivatives}]
Statement~\eqref{eq:partial-g} is direct from differentiation of~\eqref{eq:g-def}, and (i) and (ii) follow 
since $\sup_{y :\| y \|_d \leq b(x)} | g(x,y) - x | \leq |\gamma| b(x)$, 
and $\lim_{x \to \infty} b'(x) = 0$ implies that $\sup_{x \geq 1} | b' (x) | < \infty$ and
$\lim_{x \to \infty} x^{-1} b (x) = 0$ also, while $\| \nabla g \|_{d+1}$ is bounded on bounded subsets of~$\cD$, by assumption.
For $x>1$, differentiating~\eqref{eq:partial-g}, we obtain
\begin{align*}
\partial_x \partial_x g(x,y) & = \gamma \left( \frac{2b'(x)^2}{b(x)^3} - \frac{b''(x)}{b(x)^2} \right) \| y \|^2_d , \\
\partial_{y_i} \partial_{y_j} g(x,y) & =  \frac{2 \gamma}{b(x)} \1 { i = j}, \\
\partial_x \partial_{y_i} g(x,y) & = -2 \gamma \frac{b'(x)}{b(x)^2} y_i .
\end{align*}
Denoting $\Sigma (z) = ( \Sigma_{ij} (z) )_{0 \leq i,j \leq d}$ for $z = (x,y) \in \cD$,
where index $i$ corresponds to coordinate $x$ if $i=0$ and to $y_i$ if $1 \leq i \leq d$, it follows
that, for $x>1$, 
\begin{align*}  \Delta_{\Sigma} g(z) = 
\gamma \Sigma_{00} (z) \left( \frac{2b'(x)^2}{b(x)^3} - \frac{b''(x)}{b(x)^2} \right) \| y \|^2_d + \frac{2\gamma}{b(x)} \sum_{i=1}^d \Sigma_{ii} (z) - 4\frac{b'(x)}{b(x)^2}   \sum_{i=1}^d y_i \Sigma_{0i} (z).\end{align*}
Since,
 from~\eqref{eq:Sigma-limit}, $\sup_{y : \| y \|_d \leq b(x)} | \sum_{i=1}^d \Sigma_{ii} (x,y) - \sigma^2 | \to 0$ as $x \to \infty$,
we obtain
\begin{align*}
     \sup_{y : \| y \|_d \leq b(x)}  \left| \frac{1}{2} b(x) \Delta_\Sigma g(x,y)  - \gamma \sigma^2 \right| \leq C \| \Sigma (z) \|_{\rm op} \left[ b'(x) + b'(x)^2 + b(x) b''(x) \right] + o(1),
     \end{align*}
     for some constant~$C < \infty$. By assumption, this tends to $0$ as $x \to \infty$, giving~(iii).
\end{proof}

\begin{proof}[Proof of Lemma~\ref{lem:nu-properties}]
Suppose that $x > 1$, and consider $z = (x,y) \in \pcD$. Then $\| y \|_d = b(x) > 0$ and $\hat y = y / b(x) \in \Sp{d-1}$. 
Write $\phi( z ) = \langle \phi(z) , e_x \rangle e_x + \langle \phi(z) , e_{\hat y} \rangle e_{\hat y}$. 
By the expression for~$\nabla g(z)$ from~\eqref{eq:partial-g},  
we have $\partial_x g(x,y) = 1 -\gamma b'(x)$, $\nabla_y g(x,y) = 2 \gamma \hat y$, and
\begin{align*}
\nu (z) & =    \langle \phi(z) , e_x \rangle \partial_x g(z) 
+     \langle \phi(z) , e_{\hat y} \rangle {\hat y}^\tra\nabla_y g (z) 
=     \langle \phi(z) , e_x \rangle + 2 \gamma \langle \phi(z) , e_{\hat y} \rangle + o(1)   , \end{align*}
as $x \to \infty$, provided $b'(x) \to 0$, using the fact that $\| \phi(z) \|_{d+1}$ is bounded. 
The conclusion of the lemma now follows, since assumptions~\eqref{eq:e1-projection} and~\eqref{eq:z-projection} show that
\[ \lim_{x \to \infty} \sup_{y : \| y \|_d = b(x)} \left| \nu (x,y) - s_0 + 2 \gamma c_0 \right| = 0 . \qedhere \]
\end{proof}

\begin{proof}[Proof of Lemma~\ref{lem:b-B-bound}]
First we prove~\ref{lem:b-B-bound-i}. 
By definition of $\beta$ at~\eqref{eq:beta-def}, 
for any $\eps >0$, there exists $x_1 \in \RP$ such that
$(\beta+\eps) b(x) \geq x b'(x)$ for all $x \geq x_1$. Hence, for $x > x_1$,
\begin{align*}
B(x) = \int_0^x b(s) \ud s &   \geq \frac{1}{\beta+\eps} \int_{x_1}^x s b'(s) \ud s  \\
& = \frac{1}{\beta + \eps} \bigl[ s b(x) \bigr]_{x_1}^x - \frac{1}{\beta+\eps} \int_{x_1}^x b(s) \ud s .\end{align*}
It follows that, as $x \to \infty$,
\begin{equation}
    \label{eq:B-lower-bound}
 \left[ 1 + \frac{1}{\beta +\eps} \right] B(x) \geq \frac{x b(x)}{\beta+\eps} + O(1) .
 \end{equation}
Thus, for some $C < \infty$,
\begin{equation}
    \label{eq:b-square-bound}
x^2 b(x)^2 \leq C B(x)^2 + C, \text{ for all } x \in \RP.
\end{equation}
If $B(\infty) < \infty$, then~\eqref{eq:b-square-bound} gives $b(x) = O(1/x)$ as $x \to \infty$.
Suppose that $B(\infty) = \infty$.
Then, since $b(x) = O( x^{\beta + \eps} )$ for any $\eps >0$, we  
must have $\beta \geq -1$ (or else $B$ would be bounded). Let $\delta \in (1, 2 \wedge \frac{2}{1+\beta})$.
Then 
$B(x)^\delta= O( x^{2} )$ as $x \to \infty$.
Hence, by~\eqref{eq:b-square-bound},
there is a constant $C < \infty$ such that, for all $x \geq 1$, say, 
\[ b(x)^2 \leq C \frac{B(x)^2}{x^2} = C \frac{B(x)^\delta B(x)^{2-\delta} }{x^2} \leq C' B(x)^{2-\delta} , \]
for some $C' < \infty$. Since $b$ is bounded on compact intervals, part~\ref{lem:b-B-bound-i} follows.

For part~\ref{lem:b-B-bound-ii},
 we have that, for fixed $\omega \in \R$, for all $x$ sufficiently large
\begin{align*}
\bigl| B (x + \omega b(x) ) - B(x) \bigr| = \biggl| \int_x^{x+\omega b(x)} b (s ) \ud s\biggr| \leq | \omega | \cdot b(x) \cdot \sup_{x/2 \leq s \leq 2x} b (s) .
\end{align*}
Thus from~\eqref{eq:b-square-bound} we see that
$| B (x + \omega b(x) ) - B(x) | \leq C|\omega| (1+ B(x)^2 )/ x^2$ for all large enough~$x$,
which together with the fact that $B(x) = o(x^2)$ as $x \to \infty$ yields~\ref{lem:b-B-bound-ii}.
\end{proof}
 
\appendix
\section{Solutions, existence, and uniqueness}
\label{sec:construction}

This section defines formally
the terminology in Theorem~\ref{thm:existence} below
and then gives its proof.
The first step is to describe the function space
on which our (possibly explosive) solutions to~\eqref{eq:SDE-for-Z}
will live, then we proceed to define
the concept of a solution up to a predictable stopping time,
and discuss existence and uniqueness; 
we draw in part on the approach of~\cite[\S1.5]{ce}
for solution theory of potentially explosive SDEs.

Recall from~\eqref{eq:domain-def} the definition of $\cD \subseteq \R^{d+1}$,
which inherits the usual topology from $\R^{d+1}$. 
Let  $\barcD := \cD \cup \{ \partial \}$
denote the one-point (Alexandroff) compactification of $\cD$
whose open sets are the open sets in $\cD$ together with all $U = (\cD \setminus B ) \cup \{ \partial \}$ over compact $B \in \cD$.
The adjoined state $\partial$ will accommodate explosion.
Since $\cD$ is open in $\barcD$, $\{ \partial \}$ is closed. 
For $z = (x,y) \in \cD$, let $P_1 (x,y) := x \in \RP$ denote projection
onto the first coordinate, and extend to $P_1 : \barcD \to [0,\infty]$ by setting $P_1 (\partial ) = \infty$. Then $z_n \in \barcD$ has $z_n \to \partial$
if and only if $P_1 (z_n) \to \infty$,
since $z_n \to \partial$ if and only if for every compact $B$ it is the case that $z_n \in ( \cD \setminus B ) \cup \{ \partial \}$ for all $n$ sufficiently large.
Thus $P_1 : \barcD \to [0,\infty]$ is continuous, where $[0,\infty] := \RP \cup \{ \infty \}$ also has the topology of the one-point compactification.

Let $C := C ( \RP, \barcD)$ denote the set of continuous functions $f : \RP \to \barcD$. By choice of topology on $\barcD$, any $f \in C$ has the properties:
\begin{itemize}
    \item[(i)] If $t \in \RP$ is such that $f(t) \in \cD$, then $\lim_{s \to t} f(s) = f(t)$.
    \item[(ii)] If $t \in \RP$ is such that $f(t) = \partial$, then $\lim_{s \to t} P_1 (f(s)) = \infty$.
\end{itemize} 

With the usual convention that $\inf \emptyset := \infty$, 
define $\cE : C \to [0,\infty]$  by 
\begin{equation}
    \label{eq:E-def}
\cE (f) := \inf \{ t \in \RP : f(t) = \partial \}.
\end{equation}
By continuity of $f$, $f (\cE(f) ) = \partial$ if $\cE(f) < \infty$,
so property~(ii) above shows that $\lim_{t \uparrow \cE(f)} P_1 (f(t)) = \infty$.
For $f \in C$ define $S_x : C \to [0,\infty]$ by
\begin{equation}
    \label{eq:Sx-def}
 S_x (f) := \inf \left\{ t \in \RP : P_1 (f(t)) \in  [x, \infty]  \right\} .
 \end{equation}
 We claim that
 \begin{equation}
     \label{eq:Sx-limit}
     \cE (f) = \lim_{x \to \infty} S_x (f), \text{ for every } f \in C.
 \end{equation}
 Indeed, $S_{x'} \geq S_x$ for all $x' \geq x$, so $S (f) := \lim_{x \to \infty} S_x(f)$ exists in $[0,\infty]$. Clearly $S_x (f) \leq \cE(f)$, so $S (f) \leq \cE(f)$.
If $S(f) < \cE(f)$ then,
$S(f) + 2\eps < \cE (f) $ for some $\eps >0$,
and $M := \sup_{0 \leq s \leq S(f) +\eps} P_1 (f(s)) < \infty$,
by uniform continuity of $s \mapsto P_1 (f(s))$ on compact intervals before $\cE(f)$. Then for $x > M$ we would have $S_x (f) \geq S(f) +\eps$ and hence $S(f) \geq S(f) +\eps$, which is a contradiction;
this establishes~\eqref{eq:Sx-limit}.

Endow $C$ with the compact-open topology, 
that is, the topology generated by $T(K,U) = \{ f \in C : f(K) \subseteq U\}$
over compact $K \subseteq \RP$  and open $U \subseteq \barcD$. 
Suppose that $f_n \to f$.
If $U \subseteq \cD$ is open, then $f[0,t] \subseteq U$ implies that $t<\cE(f)$,
and so the requirement that $f_n[0,t] \in U$ for all $n$ sufficiently large
means that $\liminf_{n \to \infty} \cE(f_n) > t$ and $f_n$ converges to $f$ uniformly over $[0,t]$.
It follows that $f_n \to f$ implies  that
\begin{align}
    \label{eq:uniform-on-compacts}
\sup_{0 \leq s \leq t} \| f_n(s) - f(s) \|_{d+1} & \to 0, \text{ for all } t < \cE (f); \\
\label{eq:continuity-of-Sx}
\lim_{n \to \infty} S_x (f_n) & = S_x (f) ; \\
 \label{eq:liminf-of-explosion-times}
\liminf_{n \to \infty} \cE (f_n) & \geq \cE(f) .
\end{align}
Here~\eqref{eq:continuity-of-Sx} follows from~\eqref{eq:uniform-on-compacts}. Indeed, suppose that~\eqref{eq:uniform-on-compacts} holds, and $S_x (f) = S \in [0,\infty]$.
For any $t < S$, for all $n$~sufficiently large, 
$\sup_{0 \leq s \leq t} P_1 ( f_n(s) ) < x$, so $S_x (f_n) > t$,
and $\liminf_{n \to \infty} S_x (f_n) \geq S_x (f)$.
If $S_x(f) = \infty$ then this is a limit; otherwise, a similar argument
in the other direction shows that $\limsup_{n \to \infty} S_x (f_n) \leq S_x (f)$.
In any case, we obtain~\eqref{eq:continuity-of-Sx}. 
Since $\cE(f_n) \geq S_x(f_n)$, this implies~\eqref{eq:liminf-of-explosion-times}.

Let $\barC := \barC (\RP, \barcD)$ denote the set of $f \in C$
satisfying $f (t) = \partial$ for all $t \geq \cE (f)$,
endowed with the compact-open topology inherited from $C$. 
We will show that solutions to~\eqref{eq:SDE-for-Z}
can be interpreted as trajectories~$f \in \barC$ with potential explosion time~$\cE(f)$.

We will talk about solutions of~\eqref{eq:SDE-for-Z} in the sense of \emph{solutions up to a predictable stopping time}: cf.~the discussion in~\cite[\S1.5]{ce}. To describe this, we need some more notation and definitions.  
Define for $n \in \ZP$ the stopping time $T_n$ given by
\begin{equation}
    \label{eq:predicting-sequence}
T_n (f) := n \wedge S_n(f),  \text{ for } f \in C .
\end{equation}
 Then $T_n$ is a 
 \emph{predicting sequence} for $\cE$,
meaning that it has the following properties:
\begin{itemize}
    \item[(i)] $T_n(f) \leq T_{n+1}(f)$;
    \item[(ii)] $T_n(f) \leq \cE(f)$, and $T_n(f) < \cE(f)$ if $\cE(f) > 0$;
    \item[(iii)] $\lim_{n \to \infty} T_n(f) = \cE(f)$.
\end{itemize}
We say that $\cE(f)$ is a \emph{predictable stopping time} with predicting sequence $T_n$.
Property (i) follows since $S_{n+1} (f) \geq S_n (f)$, and (iii) since $\lim_{n \to \infty} S_n (f) = \cE(f)$. Suppose that $0 < \cE(f) < \infty$, then, 
since $P_1(f(t)) \to \infty$ as $t \uparrow \cE(f)$, 
for any $x \in \RP$ we can find $t < \cE(f)$ such that $P_1( f(t) ) \in [x,\infty)$;
hence $S_x (f) < \cE (f) < \infty$ for all $x \in \RP$.
On the other hand, if $\cE(f) =\infty$ then $n < \cE(f)$ for all $n$. This establishes~(ii).

Say that~\eqref{eq:SDE-for-Z} has a \emph{strong solution up to time $\cE^-$} if 
for every probability space $(\Omega, \cF, \Pr)$ with a complete, right-continuous filtration $(\cF_t, t \in \RP)$ and an adapted $(d+1)$-dimensional Brownian motion $W = (W_t, t \in \RP)$,
and every~$z \in \cD$,
there exists a pair $(Z,L)$ 
with $\Pr ( Z \in \barC ) = 1$
and, 
for every $r \in \RP$, $(Z_{t \wedge \sigma_r}, t \in \RP)$ is an adapted semimartingale on $\cD$ and $(L_{t \wedge \sigma_r}, t \in \RP)$ is a bounded variation process on $\RP$ for which
\begin{equation}
    \label{eq:strong-solution-stopped}
\begin{split}
 Z_{t \wedge \sigma_r} & = z + \int_0^{t \wedge \sigma_r} \Sigma^{1/2} ( Z_s) \ud W_s + \int_0^{t \wedge \sigma_r} \phi ( Z_s) \ud L_s , \\
& \qquad\qquad \text{ and } L_t = \int_0^{t \wedge \sigma_r} \1 { Z_s \in \pcD } \ud L_s ,
\end{split}
\end{equation}
where $\sigma_r := \inf \{ t \in \RP : P_1 (Z_t) \geq r \}$,
i.e., $\sigma_r = S_r (Z)$, a.s., in the notation at~\eqref{eq:Sx-def}. We define $\taue := \lim_{r \to \infty} \sigma_r$, so that~$\taue = \cE (Z)$, a.s., with the notation at~\eqref{eq:E-def}. Thus if we have a strong solution up to time $\cE^-$, we have the triple $(Z,L,\taue)$ as described.

Moreover, we say \emph{pathwise uniqueness} holds if
for every probability space $(\Omega, \cF, \Pr)$ equipped with a complete, right-continuous filtration $(\cF_t, t \in \RP)$ and an adapted $(d+1)$-dimensional Brownian motion $W = (W_t, t \in \RP)$,
if there exist two strong solutions up to time $\cE^-$ on $(\Omega, \cF, \Pr)$ with respect to $W$,
denoted by $(Z,L,\taue)$ and $(Z',L',\taue')$, say, then~$\Pr (Z_0 = Z_0') = 1$ implies that $\Pr ( Z = Z', L = L', \taue = \taue') = 1$.

\begin{theorem}
\label{thm:existence}
Suppose that~\eqref{ass:domain1}, \eqref{ass:variance}, and~\eqref{ass:vector-field} hold.
Then there exists a strong
solution 
$(Z,L,\taue)$ satisfying~\eqref{eq:SDE-for-Z}, and there is pathwise uniqueness. In particular,~\eqref{eq:SDE-for-Z} defines a continuous strong Markov process $Z$ over time interval $[0,\taue)$ and
\begin{equation*}
\lim_{t \uparrow \taue} \|Z_t\| 
= \lim_{t \uparrow \taue} L_t = \infty, \text{ on } \{ \taue < \infty\} .\end{equation*}
\end{theorem}

\begin{remark}
The idea of the proof of 
Theorem~\ref{thm:existence} 
is to apply 
existence and uniqueness results from~\cite{ls} for diffusions with oblique reflections on bounded domains to an increasing sequence of bounded domains. 
The main technical contribution of the proof of Theorem~\ref{thm:existence}
is establishing 
$\lim_{t \uparrow \taue} \|Z_t\| =\infty$, which is required for
$\Pr ( Z \in \barC ) = 1$
in the definition of a solution. 
\end{remark}

\begin{proof}[Proof of Theorem~\ref{thm:existence}.]
Let $(\Omega, \cF, \Pr)$ be a probability space accommodating a
$(d+1)$-dimensional Brownian motion $W$
adapted to a complete, right-continuous filtration $(\cF_t, t \geq 0 )$. 
We show how to construct a pair $(Z,L)$ which satisfy~\eqref{eq:strong-solution-stopped} for every~$r$. To do so, we approximate $\cD$ by an increasing sequence of bounded domains, on which we can use the results of~\cite[\S 4]{ls}, and then take a limit. 

The assumption~\eqref{ass:domain1} 
implies that $\cD$ is a $C^2$ domain, by Lemma~\ref{lem:C2-domain}. 
Let $(\cD^{(u)}, u > 0)$ be a sequence of bounded $C^2$ domains in $\R^{d+1}$, and define $\cD^{(u)}_r = \{ ( x, y) \in \cD^{(u)} : x \leq r \}$
and $\cD_r = \{ ( x, y) \in \cD : x \leq r \}$.
Suppose that for all $u \geq r > 0$, $\cD^{(u)}_r =  \cD_r$ and $\cD_r \cap \pcD^{(u)} = \cD_r \cap \pcD$. Suppose also that
 $\phi^{(u)} : \pcD^{(u)} \to \R^{d+1}$ is such that $\phi^{(u)} (z) = \phi (z)$
for all $z \in \cD_u$, that $\phi^{(u)}$ is $C^2$,
and that the analogue of~\eqref{eq:orthogonal-component-strict} holds.

Fix $z = (x, y) \in \cD$, and take $r \in (x, \infty)$.
Since $\Sigma$ is Lipschitz and uniformly elliptic, the symmetric
square root $\Sigma^{1/2}$ is also Lipschitz~\cite[p.~131]{sv}. Then, since $\cD^{(r)}$ is $C^2$ and bounded,
 $\Sigma^{1/2}$ is Lipschitz, 
and $\phi^{(r)}$ satisfies the conditions described above,
all the conditions of Theorem~4.3 of~\cite{ls} are satisfied
for domain $\cD^{(r)}$ and vector field $\phi^{(r)}$. That result then implies that
there is an $\cF_t$-adapted continuous 
semimartingale $Z^{(r)}$ with $Z_t^{(r)} \in \cD^{(r)}$ for all $t \geq 0$, and a bounded variation process $L^{(r)}$, such that, for all $t \in \RP$,
\begin{equation}
\label{eq:SDE-r-global}
\begin{split}
 Z^{(r)}_{t} & = z + \int_0^{t} \Sigma^{1/2} ( Z^{(r)}_s) \ud W_s + \int_0^{t} \phi^{(r)} ( Z^{(r)}_s) \ud L^{(r)}_s , \\
& \qquad\qquad \text{ and } L^{(r)}_{t} = \int_0^{t} \1 { Z^{(r)}_s \in \pcD^{(r)}  } \ud L^{(r)}_s .  
\end{split}
\end{equation}
Moreover, the results of~\cite{ls} show that the pair $(Z^{(r)},L^{(r)})$ is essentially unique, in that any other pair for which~\eqref{eq:SDE-r-global} holds must be a.s.~identical.

Define $\sigma^{(r)}_w := \inf \{ t \geq 0 : Z^{(r)} (t) \geq w \}$. Then
stopping the process at time $\sigma^{(r)}_r$,
from~\eqref{eq:SDE-r-global} and using the facts that
$\phi^{(r)} ( Z^{(r)}_{t \wedge \sigma^{(r)}_r} ) = \phi  ( Z^{(r)}_{t \wedge \sigma^{(r)}_r} ) $ and 
$Z^{(r)}_{t \wedge \sigma^{(r)}_r} \in \cD_r$, we have 
\begin{equation}
\label{eq:SDE-r}
\begin{split}
 Z^{(r)}_{t\wedge \sigma^{(r)}_r} & = z + \int_0^{t\wedge \sigma^{(r)}_r} \Sigma^{1/2} ( Z^{(r)}_s) \ud W_s + \int_0^{t\wedge \sigma^{(r)}_r} \phi ( Z^{(r)}_s) \ud L^{(r)}_s , \\
& \qquad\qquad \text{ and } L^{(r)}_{t\wedge \sigma^{(r)}_r} = \int_0^{t\wedge \sigma^{(r)}_r} \1 { Z^{(r)}_s \in  \pcD   } \ud L^{(r)}_s .
\end{split}
\end{equation}
Note that uniqueness of $(Z^{(r)},L^{(r)})$ in~\eqref{eq:SDE-r-global}
implies uniqueness of $(Z^{(r)},L^{(r)},\sigma^{(r)}_r)$ in~\eqref{eq:SDE-r}. 
On the same probability space,  
 we can for $u > r$ define $Z^{(u)}$ such that, 
\[
\begin{split}
 Z^{(u)}_{t\wedge \sigma^{(u)}_u} & = z + \int_0^{t\wedge \sigma^{(u)}_u} \Sigma^{1/2} ( Z^{(u)}_s) \ud W_s + \int_0^{t\wedge \sigma^{(u)}_u} \phi ( Z^{(u)}_s) \ud L^{(u)}_s , \\
& \qquad\qquad \text{ and } L^{(u)}_{t\wedge \sigma^{(u)}_r} = \int_0^{t\wedge \sigma^{(u)}_u} \1 { Z^{(u)}_s \in  \pcD   } \ud L^{(u)}_s .
\end{split}
\]
In particular, since $\sigma^{(u)}_r \leq \sigma^{(u)}_u$,
\[
\begin{split}
 Z^{(u)}_{t\wedge \sigma^{(u)}_r} & = z + \int_0^{t\wedge \sigma^{(u)}_r} \Sigma^{1/2} ( Z^{(u)}_s) \ud W_s + \int_0^{t\wedge \sigma^{(u)}_r} \phi ( Z^{(u)}_s) \ud L^{(u)}_s , \\
& \qquad\qquad \text{ and } L^{(u)}_{t\wedge \sigma^{(u)}_r} = \int_0^{t\wedge \sigma^{(u)}_r} \1 { Z^{(u)}_s \in   \pcD   } \ud L^{(u)}_s .
\end{split}
\]
Hence $(Z^{(u)},L^{(u)},\sigma^{(u)}_r)$ solves~\eqref{eq:SDE-r} and so,
by uniqueness, we have $\sigma^{(u)}_r   = \sigma^{(r)}_r$ for all $u \geq r$,
and so if we write $\sigma_r := \lim_{u \to \infty} \sigma^{(u)}_r$, we have
\[ \text{ for all } u \geq r, ~\sigma_r = \sigma^{(r)}_r = \sigma^{(u)}_r, \text{ and }  (Z^{(u)}_{t \wedge \sigma_r} , L^{(u)}_{t \wedge \sigma_r } ) = (Z^{(r)}_{t \wedge \sigma_r} , L^{(r)}_{t \wedge \sigma_r } )  .\]
Since $\sigma_r$ is the hitting time of a closed set, it is a stopping time for $(\cF_t)_{t \geq 0}$. 
On the same probability space we then define $\taue = \lim_{r \to \infty} \sigma_r$, also a stopping time for $(\cF_t)_{t \geq 0}$~\cite[p.~7]{ks}.

On the same probability space,
we may now define
\[ Z_t = \begin{cases} \lim_{r \to \infty} Z^{(r)}_{t \wedge \sigma_r} &\text{if } t < \taue , \\
\partial & \text{if } t \geq \taue , \end{cases} \]
and set $L_t = \lim_{r \to \infty} L^{(r)}_{t \wedge \sigma_r}$ for $t < \taue$. 
Note that for every $t < \taue$, $t < \sigma_r$ for all $r \geq r(t)$ sufficiently large, so $Z_t = Z^{(u)}_t$
and $L_t = L^{(u)}_t$ for all $u \geq r(t)$, i.e., the limits are eventually constant. Moreover, $Z_{t \wedge \sigma_r} = Z^{(r)}_{t \wedge \sigma_r}$ and 
$L_{t \wedge \sigma_r} = L^{(r)}_{t \wedge \sigma_r}$ 
so,
by~\eqref{eq:SDE-r}, for any $r \in \RP$,
\[
\begin{split}
 Z_{t\wedge \sigma_r} & = z + \int_0^{t\wedge \sigma_r} \Sigma^{1/2} ( Z_s) \ud W_s + \int_0^{t\wedge \sigma_r} \phi ( Z_s) \ud L_s , \\
& \qquad\qquad \text{ and } L_{t\wedge \sigma_r} = \int_0^{t\wedge \sigma_r} \1 { Z_s \in  \pcD   } \ud L_s .
\end{split}
\]
Thus we have shown that $(Z,L)$ satisfy~\eqref{eq:strong-solution-stopped}.

We have defined $L$ such that $L_t$ is nondecreasing for $t < \taue$, so we complete the definition by setting, if $\taue < \infty$, 
$L_\taue = \lim_{t \to \taue} L_t$ and $L_t = L_\taue$ for all $t \geq \taue$.
Recall the definition of $m$ and $\ell$
from~\eqref{eq:m-l-def}, and from~\eqref{eq:x-projection}
that
$X_{t} 
= \langle e_x, z \rangle + m_{t} + \ell_{t}$. 
The local martingale~$(m_{t \wedge \sigma_r})_{t \geq 0}$
has $\Exp ( [ m ]_{t \wedge \sigma_r} ) \leq Ct$ for all $t, r \in \RP$,
so  is a martingale.
Set 
$[m]_{\taue} := \lim_{r \to \infty} [m]_{\sigma_r} \in [0,\infty]$ and
$T_n := \inf \{ t \in \RP : [ m ]_{t \wedge \taue} \geq n \}$, for $n \in \N$.
Then $\sup_{r \geq 0}\Exp  [ m^2_{\sigma_r \wedge T_n} ] < \infty$,
and hence optional stopping shows that $(m_{\sigma_k \wedge T_n} )_{k \in \N}$
is a martingale uniformly bounded in~$L^2$. Hence $Q_{n,\infty} = \lim_{k \to \infty} m_{\sigma_k \wedge T_n}$
exists and is finite, for each~$n$. On $\{ \taue < \infty \}$, we have $T_{n_0} = \infty$
for some a.s.~finite $n_0$, and hence $\lim_{k \to \infty} m_{\sigma_k} = Q_{n_0, \infty}=: m_\taue$, say, is
a.s.~finite. 

Recall that $\cD_r = \{ (x,y) \in \cD : x \leq r \}$.
We next claim that for every $ r \in (0,\infty)$, there exists a constant $c_r >0$ for which
\begin{equation}
    \label{eq:passage-time-lower-bound}
\inf_{z \in \cD_r} \Pr_z ( \sigma_{2r} > c_r ) \geq \frac{1}{3}.
\end{equation}
We now prove~\eqref{eq:passage-time-lower-bound}. 
For $0 \leq t < \taue$, define $X'_t := \langle e_x, z \rangle + m_{t} + \ell^+_{t}$,
where
\[ \ell^+_t := \int_0^t \langle e_x, \phi(Z_s) \rangle^+ \ud L_s , \text{ for } 0 \leq t < \taue.\]
Then $0 \leq X_t \leq X_t'$ for all $0 \leq t < \taue$ and, for any $r >0$, $(X_{t \wedge \sigma_r})_{t \in \RP}$
is a non-negative submartingale. Hence, for any $t \in \RP$,
\begin{align}
\nonumber
    \Pr ( \sigma_{2r} \leq t )   = \Pr \left( \sup_{0\leq s \leq t \wedge \sigma_{2r} } X_s \geq 2r \right) 
   & \leq \Pr \left( \sup_{0\leq s \leq t \wedge \sigma_{2r} } X'_s \geq 2r \right) \\
    \nonumber
   & \leq \frac{1}{2r} \Exp  X'_{t \wedge \sigma_{2r}}  \\
   & \leq \frac{1}{2r} \left( \langle e_x , z \rangle + C \Exp   L_{t \wedge \sigma_{2r}}   \right),
   \label{eq:local_time_bound_exit}
    \end{align}
by the maximal inequality for non-negative submartingales~(e.g.~Thm.~3.8(i) of~\cite[p.~13]{ks}),
and the fact that $\ell_{t \wedge \sigma_{2r}}^+ \leq C L_{t \wedge \sigma_{2r}}$, where $C := \sup_{z \in \pcD} \| \phi(z) \|_{d+1} < \infty$ by~\eqref{ass:vector-field}. To bound $\Exp   L_{t \wedge \sigma_{2r}}$, consider $N : \R^{d+1} \to \R$
with the property that $\nabla N(z) = \phi_x (u)$
for every $z = (x, ub(x)) \in \pcD$, for which $N$ is $C^2$, and so all its partial derivatives of up to second order are
bounded on compact sets. For $U_t := N (Z_t)$, $0 \leq t < \taue$,  It\^o's formula implies
\[  U_{t \wedge \sigma_{2r}} - U_0 \geq \tilde m_{t } +  L_{t \wedge \sigma_{2r}} - C_r t,\]
where $(\tilde m_{t } )_{t \in \RP}$ is  a martingale and the constant $C_r<\infty$ exists since $\Sigma$ and the second derivatives of $N$ are bounded on $\cD_{2r}$.  Hence
$\Exp L_{t \wedge \sigma_{2r}} \leq  C_r t + \Exp |U_{t \wedge \sigma_{2r}} - U_0 | $. Moreover, $\sup_{t \in \RP} U_{t \wedge \sigma_{2r}}$ is bounded by a constant, and so by bounded convergence $\lim_{t \to 0} \Exp L_{t \wedge \sigma_{2r}} = 0$
for every fixed $r \in (0,\infty)$. Thus, by~\eqref{eq:local_time_bound_exit}, we can choose $t = t_0 >0$ small enough (depending on $r$) such that
$\inf_{z \in \cD_r} \Pr_z ( \sigma_{2r} \geq t_0 ) \geq 1/3$, since $z = (x,y)$ satisfies $x \leq r$. This completes the proof of~\eqref{eq:passage-time-lower-bound}.

From~\eqref{eq:passage-time-lower-bound} it follows that
$\lim_{t \uparrow \taue} X_t = \infty$ on the event $\{\taue < \infty\}$. 
To see this,
define stopping times
$s_0:=0$, and, for $k \in \N$, $t_k := \inf \{ t \geq s_{k-1} : X_t \leq 2r \}$
and $s_k := \inf \{ t \geq t_k : X_t \geq 4r\}$.
On the event $\{\taue < \infty, \, \liminf_{t \uparrow \taue} X_t \leq r \}$,
we have $s_0 < t_1 < s_1 < \cdots < \taue < \infty$. However,
by~\eqref{eq:passage-time-lower-bound} and
the strong Markov property, $\Pr ( s_k - t_k \geq c_{2r} \mid \cF_{t_k} ) \geq 1/3$ on $\{ t_k < \infty \}$.
From L\'evy's extension of the Borel--Cantelli lemma (Cor.~9.21 in~\cite[p.~197]{kall}), it follows that
$s_k - t_k \geq c_{2r}$ occurs infinitely often, a.s., on the event $\{\taue < \infty, \, \liminf_{t \uparrow \taue} X_t \leq r \}$.
But then $\liminf_{k \to \infty} s_k = \sum_{\ell=1}^{\infty} (s_{\ell}-s_{\ell-1}) \geq \sum_{\ell=1}^{\infty} (s_{\ell} - t_\ell ) = \infty$, contradicting the fact that $\limsup_{k \to \infty} s_k \leq \taue < \infty$. Thus $\lim_{t \uparrow \taue} X_t = \infty$ on the event $\{\taue < \infty\}$.

Since $X_{t} 
= \langle e_x, z \rangle + m_{t} + \ell_{t}$, and $m_t$ has $\lim_{t \uparrow \taue} m_t < \infty$ on $\{\taue < \infty\}$,
as argued above, it follows that 
$\lim_{t \uparrow \taue} \ell_t = \infty$ on $\{ \taue < \infty \}$ as well. 
Also,
 $\ell_t \leq K L_t$, where  $K := \sup_{z \in \pcD} \| \phi (z)\|_{d+1} < \infty$ by~\eqref{ass:vector-field}.
Hence
we conclude that
\begin{equation}
   \label{eq:no-oscillations}
\lim_{t \uparrow \taue} X_t =
\lim_{t \uparrow \taue} \ell_t
= \lim_{t \uparrow \taue} L_t = \infty, \text{ on } \{ \taue < \infty\} .\end{equation}
 In particular, this verifies the claim that
$\Pr ( Z \in \barC ) = 1$.

Finally, we turn to uniqueness.
As already described, $Z^{(u)}_{t \wedge \sigma_r} = Z^{(r)}_{t \wedge \sigma_r}$, and hence
$Z_{t \wedge \sigma_r} = \lim_{u \to \infty} Z^{(u)}_{t \wedge \sigma_r} = Z^{(r)}_{t \wedge \sigma_r}$.
If $(Z',L')$ is another strong solution of the SDE~\eqref{eq:SDE-for-Z} up to $\cE^-$, 
then, as already argued, $(Z',L')$ solves the SDE~\eqref{eq:SDE-r} over time interval $[0,\sigma_r]$. But uniqueness for~\eqref{eq:SDE-r}
means that $Z'_{t \wedge \sigma_r} = Z^{(r)}_{t \wedge \sigma_r} = Z_{t \wedge \sigma_r}$. This is true for all $r > x$,
so $Z'$ coincides with $Z$, establishing uniqueness.
\end{proof}
 
 \section*{Acknowledgements}
 
 AM was supported by EPSRC grants EP/V009478/1 and EP/P003818/1, the Turing Fellowship funded by the Programme on Data-Centric Engineering of Lloyd’s Register Foundation, and 
by The Alan Turing Institute under the EPSRC grant EP/N510129/1.

\end{document}